\RequirePackage{amsthm}

\documentclass[sn-mathphys-ay]{sn-jnl}

\usepackage{lmodern}
\usepackage{graphicx}%
\usepackage{multirow}%
\usepackage{amsmath,amssymb,amsfonts}%
\usepackage{mathrsfs}%
\usepackage[title]{appendix}%
\usepackage{xcolor}%
\usepackage{textcomp}%
\usepackage{manyfoot}%
\usepackage{booktabs}%
\usepackage{algorithm}%
\usepackage{algorithmicx}%
\usepackage{algpseudocode}%
\usepackage{listings}%
\usepackage{hyperref}
\usepackage{cleveref}
\usepackage{accents}
\usepackage{commath}
\usepackage{anyfontsize}
\usepackage{relsize}
\theoremstyle{thmstyleone}%
\newtheorem{theorem}{Theorem}
\newtheorem{proposition}[theorem]{Proposition}%
\newtheorem{corollary}[theorem]{Corollary}
\newtheorem{lemma}[theorem]{Lemma}

\theoremstyle{thmstyletwo}%
\newtheorem{remark}{Remark}%

\theoremstyle{thmstylethree}%

\let\cal\mathcal

\newcommand{\one}{\boldsymbol{1}}

\def\Xint#1{\mathchoice
{\XXint\displaystyle\textstyle{#1}}%
{\XXint\textstyle\scriptstyle{#1}}%
{\XXint\scriptstyle\scriptscriptstyle{#1}}%
{\XXint\scriptscriptstyle\scriptscriptstyle{#1}}%
\!\int}
\def\XXint#1#2#3{{\setbox0=\hbox{$#1{#2#3}{\int}$ }
\vcenter{\hbox{$#2#3$ }}\kern-.6\wd0}}

\usepackage{listings}
\usepackage{color}

\definecolor{codegreen}{rgb}{0,0.6,0}
\definecolor{codegray}{rgb}{0.5,0.5,0.5}
\definecolor{codepurple}{rgb}{0.58,0,0.82}
\definecolor{backcolour}{rgb}{0.95,0.95,0.92}

\lstdefinestyle{mystyle}{
  backgroundcolor=\color{backcolour},   commentstyle=\color{codegreen},
  keywordstyle=\color{magenta},
  numberstyle=\tiny\color{codegray},
  stringstyle=\color{codepurple},
  basicstyle=\footnotesize,
  breakatwhitespace=false,         
  breaklines=true,                 
  captionpos=b,                    
  keepspaces=true,                 
  numbers=left,                    
  numbersep=5pt,                  
  showspaces=false,                
  showstringspaces=false,
  showtabs=false,                  
  tabsize=2
}

\newcommand{\proofparagraph}[1]{\noindent~~\newline\newline\noindent\textbf{ \underline{#1}.}\\}

\def\/{|\!|\!|}
\def\${|\!|\!|}
\def\l|{\left|\!\left|\!\left|}
\def\r|{\right|\!\right|\!\right|}

\def\R{{\mathbb{R}}}

\def\T{{\mathcal{T} }}

\def\L{{\mathcal{L}}}
\def\Z{{\mathbb{Z}}}

\newcommand{\Rplus}{\mathbb{R}_{+}}

\newcommand{\uhp}{\mathbb{H}}

\newcommand{\E}[0]{\mathbb{E}}

\newcommand\scal[2][ ]{\ifthenelse{\equal{#1}{ }}{\langle#2\rangle}{}
        \ifthenelse{\equal{#1}{b}}{\bigl\langle#2\bigr\rangle}{}
        \ifthenelse{\equal{#1}{B}}{\Bigl\langle#2\Bigr\rangle}{}
        \ifthenelse{\equal{#1}{bb}}{\biggl\langle#2\biggr\rangle}{}
        \ifthenelse{\equal{#1}{BB}}{\Biggl\langle#2\Biggr\rangle}{}}

\def\independenT#1#2{\mathrel{\rlap{$#1#2$}\mkern2mu{#1#2}}}

\newcommand{\Prob}{\mathbb{P}}
\newcommand{\Proba}[1]{\mathbb{P}\left[#1\right]}

\newcommand{\iid}{\text{i.i.d.}}

\newcommand{\Expe}[1]{\mathbb{E}\left[ #1\right] }

\newcommand{\conditional }{\biggm| }

\newcommand{\Exp}{\mathbb{E}}

\newcommand{\wt}{\widetilde}

\newcommand\thickbar[1]{\accentset{\rule{.4em}{.8pt}}{#1}}

\def\threebars{|\!|\!|}
\def\thrb{\big|\!\big|\!\big|}
\def\thrB{\Big|\!\Big|\!\Big|}
\def\thrbb{\bigg|\!\bigg|\!\bigg|}
\def\thrBB{\Bigg|\!\Bigg|\!\Bigg|}
\newcommand\nn[3][ ]{%
\ifthenelse{\equal{#1}{ }}{\|{#3}\|_{#2}}{}%
\ifthenelse{\equal{#1}{b}}{\bigl\|{#3}\bigr\|_{#2}}{}%
\ifthenelse{\equal{#1}{B}}{\Bigl\|{#3}\Bigr\|_{#2}}{}%
\ifthenelse{\equal{#1}{bb}}{\biggl\|{#3}\biggr\|_{#2}}{}%
\ifthenelse{\equal{#1}{BB}}{\Biggl\|{#3}\Biggr\|_{#2}}{}}%
\newcommand\nnn[3][ ]{%
\ifthenelse{\equal{#1}{ }}{\threebars{#3}\threebars_{#2}}{}%
\ifthenelse{\equal{#1}{b}}{\mathopen\thrb{#3}\mathclose\thrb_{#2}}{}%
\ifthenelse{\equal{#1}{B}}{\mathopen\thrB{#3}\mathclose\thrB_{#2}}{}%
\ifthenelse{\equal{#1}{bb}}{\mathopen\thrbb{#3}\mathclose\thrbb_{#2}}{}%
\ifthenelse{\equal{#1}{BB}}{\mathopen\thrBB{#3}\mathclose\thrBB_{#2}}{}}%

\newcommand{\bigcapls}[1]{\bigcap\limits_{\substack{#1}}}
\newcommand{\bigcupls}[1]{\bigcup\limits_{\substack{ #1  }  }  }

\newcommand{\Holder}{H\"{o}lder }

\def\mop#1{\mathop{\hbox{\textrm{#1}}}\nolimits}
\def\moplim#1{\mathop{\hbox{\textrm{#1}}}}
\def\harr#1#2{\smash{\mathop{\hbox to .5in{\rightarrowfill}}
        \limits^{\scriptstyle#1}_{\scriptstyle#2}}}

\makeatletter
\newcommand*{\relrelbarsep}{.386ex}
\newcommand*{\relrelbar}{%
  \mathrel{%
    \mathpalette\@relrelbar\relrelbarsep
  }%
}
\newcommand*{\@relrelbar}[2]{%
  \raise#2\hbox to 0pt{$\m@th#1\relbar$\hss}%
  \lower#2\hbox{$\m@th#1\relbar$}%
}
\providecommand*{\rightrightarrowsfill@}{%
  \arrowfill@\relrelbar\relrelbar\rightrightarrows
}
\providecommand*{\leftleftarrowsfill@}{%
  \arrowfill@\leftleftarrows\relrelbar\relrelbar
}
\providecommand*{\xrightrightarrows}[2][]{%
  \ext@arrow 0359\rightrightarrowsfill@{#1}{#2}%
}
\providecommand*{\xleftleftarrows}[2][]{%
  \ext@arrow 3095\leftleftarrowsfill@{#1}{#2}%
}
\makeatother
\def\@boxsqr#1#2#3{\vbox to0pt{
\fboxsep0pt%
\definecolor{tempc}{rgb}{#3}%
\colorbox{tempc}{\vbox to #2{\vss\hbox to #1{\hss}}}%
}}
\long\def\fillbox{\@ifnextchar[{\@fillbox{}{}}{\@fillbox{}{}[0.5mm,0.5mm]}}
\long\def\mfillbox{\@ifnextchar[%
	{\@fillbox{$\displaystyle\bgroup}{\egroup$}}%
	{\@fillbox{$\displaystyle\bgroup}{\egroup$}[0.5mm,0.5mm]}%
}
\long\def\@fillbox#1#2[#3,#4]#5#6{%
	\leavevmode%
  \setbox\@tempboxa\hbox{#1#6#2}%
  \@tempdima\wd\@tempboxa%
  \@tempdimb\ht\@tempboxa%
  \@tempdimc\dp\@tempboxa%
  \advance\@tempdima#3\advance\@tempdima#3%
  \advance\@tempdimb#4\advance\@tempdimb#4\advance\@tempdimb\@tempdimc%
  \advance\@tempdimc#4%
  \setbox\@tempboxa\hbox{%
	\@boxsqr{\number\@tempdima}{\number\@tempdimb}{#5}%
	\kern#3\raise\@tempdimc\box\@tempboxa}%
  \kern-#3\raise-\@tempdimc\box\@tempboxa%
  }
\let\phi\varphi

\def\rho{\varrho}

\def\doncl{\quad\Leftrightarrow\quad}

\def\ie{{\it i.e.}\@ifnextchar,{}{\ }}
\def\eg{{\it e.g.}\@ifnextchar,{}{\ }}

\def\({\left(}
\def\){\right)}
\newcommand{\para}[1]{\left(#1\right) }
\newcommand{\spara}[1]{\left[ #1\right] }

\newcommand{\maxp}[1]{\max\left(#1\right) }
\newcommand{\minp}[1]{\min\left(#1\right) }

\newcommand\cinf[1][ ]{\ifthenelse{\equal{#1}{ }}{{\cal C}^\infty}{{\cal C}^\infty(#1)}}
\newcommand\cOinf[1][ ]{\ifthenelse{\equal{#1}{ }}{{\cal C}_0^\infty}{{\cal C}_0^\infty(#1)}}
\newcommand\symb[2][\bf]{{\mathchoice{\hbox{#1#2}}{\hbox{#1#2}}%
        {\hbox{\scriptsize#1#2}}{\hbox{\tiny#1#2}}}}

\def\L^#1{{\textrm L}^{\!#1}}

\def\newop#1{\expandafter\def\csname#1\endcsname{\mop{#1}}}
\def\newoplim#1{\expandafter\def\csname#1\endcsname{\moplim{#1}}}
\newop{deg}
\newop{div}

\newop{sin}
\newop{arcsin}
\newop{rank}
\newop{tr}
\newop{cos}
\newop{arccos}
\newop{tan}
\newop{sinh}
\newop{arsinh}
\newop{cosh}
\newop{arcosh}
\newop{tanh}
\newop{arctan}
\newop{artanh}
\newop{log}
\newop{exp}
\newop{sgn}
\newop{dim}
\newop{ker}
\newoplim{min}
\newoplim{max}
\newoplim{sup}
\newoplim{inf}
\newoplim{lim}
\def\liminf{\moplim{lim inf}}

\newcommand{\e}{\varepsilon}

\def\CF{\mathcal{F} }

\def\CU{{\mathcal{U} } }

 \def\T{\mathbf{T}}

 \def\E{\mathbf{E}}
 \def\P{\mathbf{P}}

 \newcommand{\tor}{\text{ \normalfont{or} }}
 \newcommand{\tand}{\text{ \normalfont{and} }}
 
 \newcommand{\tifc}{\text{ ~,\normalfont{if}~}}
 \newcommand{\tfor}{\text{ \normalfont{for} }}

  \newcommand{\twhen}{\text{ \normalfont{when} }}
 \newcommand{\tforsome}{\text{ for some }}
 \newcommand{\twith}{\text{ with }}
 \newcommand{\tas}{\text{ as }}

 \newcommand{\mb}[1]{\mathbb{#1}}

\newcommand{\tcwhen}{\text{~,\normalfont{when}~}}

\newcommand{\sectm}[1]{\texorpdfstring{#1}{OOES}}

\newcommand{\branchmat}[1]{ \left\{\begin{matrix}
 #1
\end{matrix}\right. }

 \newenvironment{eqalign}{\begin{equation}\begin{aligned}}{\end{aligned}\end{equation}}

\newtheorem{innercustomgeneric}{\customgenericname}
\providecommand{\customgenericname}{}
\newcommand{\newcustomtheorem}[2]{%
  \newenvironment{#1}[1]
  {%
   \renewcommand\customgenericname{#2}%
   \renewcommand\theinnercustomgeneric{##1}%
   \innercustomgeneric
  }
  {\endinnercustomgeneric}
}

\newcustomtheorem{customthm}{Theorem}
\newcustomtheorem{customlemma}{Lemma}

 \newcommand{\ba}{\[\begin{aligned}}
\newcommand{\ea}{\end{aligned}\]}

 \newcommand{\ben}{\begin{enumerate}}
 
 \newcommand{\een}{\end{enumerate}}

 \newcommand{\bena}{\begin{enumerate}[label={\bf (\alph*) } ]}
 
 \newcommand{\benr}{\begin{enumerate}[\textbf{i.}]}

\newcommand{\br}[1]{\boldsymbol{\mathrm{#1}} }

\newcommand{\ind}[1]{\chi\left\{#1 \right\}}

\newcommand{\liz}[2][0]{\lim\limits_{#2\to #1}}

 \newcommand{\supl}[1]{\sup\limits_{#1}}

\newcommand{\infl}[1]{\inf\limits_{#1}}

\newcommand{\infp}[1]{\inf\left\{#1\right\} }

  \newcommand{\maxl}[1]{\max\limits_{#1}}
  \newcommand{\maxls}[1]{\max\limits_{\substack{#1}}}

\newcommand{\minl}[1]{\min\limits_{#1}}
\newcommand{\minls}[1]{\min\limits_{\substack{#1}}}

\newcommand{\eqdis}{\stackrel{d}{=}}

\newcommand{\dx}{\mathrm{d}x  }
\newcommand{\dy}{\mathrm{d}y  }

\newcommand{\dt}{\mathrm{d}t  }

\newcommand{\ds}{\mathrm{d}s  }
\newcommand{\dr}{\mathrm{d}r  }

\newcommand{\deta}{\mathrm{d}\eta  }

\newcommand{\dlambda}{\mathrm{d}\lambda  }

\newcommand{\etam}[1]{\eta \left(\left[ #1\right]\right) }

\newcommand{\etamu}[2]{\eta^{#2} \left(\left[ #1\right]\right) }
\newcommand{\etaum}[2]{\eta^{#1} \left(\left[ #2\right]\right) }
\newcommand{\etauin}[2]{\eta^{#1}  \left( #2\right) }

  \newcommand{\prodnk}[1]{\prod_{k=0}^{n}}

\newcommand{\sumls}[1]{\mathlarger{\mathlarger{\sum \limits_{\substack{#1}  } }}}

\newcommand{\normthrb}[1]{{\left\vert\kern-0.25ex\left\vert\kern-0.25ex\left\vert #1 
    \right\vert\kern-0.25ex\right\vert\kern-0.25ex\right\vert}}

\makeatletter
\providecommand{\vrectangle}{\mkern1mu\mathpalette\v@rectangle\relax\mkern1mu}
\newcommand{\v@rectangle}[2]{%
  \hbox{
  \fboxrule=0.5\fontdimen 8
    \ifx#1\displaystyle\textfont\else
    \ifx#1\textstyle\textfont\else
    \ifx#1\scriptstyle\scriptfont\else
    \scriptscriptfont\fi\fi\fi 3
  \fboxsep=-\fboxrule
  \fbox{$\m@th#1\phantom{(}$}%
  }
}
\makeatother

\newcommand{\expo}[1]{\exp\left\{ #1\right\}}
\newcommand{\floor}[1]{\lfloor#1 \rfloor}

\newcommand{\ceil}[1]{\lceil#1 \rceil }

\newcommand\tqed[1]{%
  \leavevmode\unskip\penalty9999 \hbox{}\nobreak\hfill
  \quad\hbox{#1}}

\def \P {{\bf P}}
\def\md{\mid}
\def\Bb#1#2{{\def\md{\bigm| }#1\bigl[#2\bigr]}}
\def\BB#1#2{{\def\md{\Bigm| }#1\Bigl[#2\Bigr]}}
\def\Bs#1#2{{\def\md{\mid}#1[#2]}}

\makeatletter
\DeclareRobustCommand{\cev}[1]{%
  \mathpalette\do@cev{#1}%
}
\newcommand{\do@cev}[2]{%
  \fix@cev{#1}{+}%
  \reflectbox{$\m@th#1\vec{\reflectbox{$\fix@cev{#1}{-}\m@th#1#2\fix@cev{#1}{+}$}}$}%
  \fix@cev{#1}{-}%
}
\newcommand{\fix@cev}[2]{%
  \ifx#1\displaystyle
    \mkern#23mu
  \else
    \ifx#1\textstyle
      \mkern#23mu
    \else
      \ifx#1\scriptstyle
        \mkern#22mu
      \else
        \mkern#22mu
      \fi
    \fi
  \fi
}

\makeatother

\newcommand{\EE}{\ensuremath{\mathbb{E}}}

\makeatletter
\newcommand{\setword}[2]{%
  \phantomsection
  #1\def\@currentlabel{\unexpanded{#1}}\label{#2}%
}
\makeatother

\makeatletter
\let\old@rule\@rule
\def\@rule[#1]#2#3{\textcolor{blue}{\old@rule[#1]{#2}{#3}}}
\makeatother

\begin{document}

\title[Article Title]{Inverse of the Gaussian multiplicative chaos: Decoupling and Multipoint moments}


\author*[1]{\fnm{Ilia} \sur{Binder}}\email{ilia@math.utoronto.ca}

\author*[1]{\fnm{Tomas} \sur{Kojar}}\email{tomas.kojar@mail.utoronto.ca}

\affil*[1]{\orgdiv{Math Department}, \orgname{University of Toronto}, \orgaddress{\street{40 St. George Street}, \city{Toronto}, \postcode{M5S 2E4}, \state{Ontario}, \country{Canada}}}


\abstract{ In this article we study the decoupling structure and multipoint moment of the inverse of the Gaussian multiplicative chaos. It is also the second part of preliminary work for extending the work in "Random conformal weldings" \cite{AJKS} to the existence of Lehto welding for the inverse. In particular, we prove that the dilatation of the inverse homeomorphism on the positive real line is in $L^{1}([0,1]\times[0,2])$.}




\maketitle
\tableofcontents

\par\rule{\textwidth}{0.4pt} 
 
\newpage\part{Introduction}
In this article we continue our systematic study of the inverse of 1d-GMC from \cite{binder2023inverse}. We consider a Gaussian field $U(t)$ on the real line $[0,\infty)$ with logarithmic covariance $\Expe{U(t)U(s)}=-\ln\abs{t-s}$, the corresponding GMC measure $\eta$
\begin{eqalign}
\eta(I):=\liz{\e}\int_{I}e^{\gamma U_{\e}(x)-\frac{\gamma^{2}}{2}\Expe{\para{U_{\e}(x)}^{2}}}\dx, I\subset [0,\infty)
\end{eqalign}
and its inverse $Q=Q_{\eta}$ i.e. $Q(\eta(x))=x=\eta(Q(x))$ (see precise definitions below) and let $Q(a,b):=Q(b)-Q(a)$. In \cite{binder2023inverse}, we studied the moments $\Expe{(Q(a,b))^{p}}$ for $p\in\para{-\frac{(1+\frac{\gamma^{2}}{2})^{2}}{2\gamma^{2}}, \infty}$. In this article, we study the behaviour of the multipoint moments 
\begin{equation}
\Expe{\prod_{k}\para{Q^{k}(a_{k},b_{k})}^{p_{k}} } \tand \Expe{\prod_{k}\para{\frac{Q^{k}(a_{k},b_{k})}{Q^{k}(c_{k},d_{k})}}^{p_{k}} } ,  \end{equation}
with $d_{k+1}<a_{k}<b_{k}<c_{k}<d_{k}$. This requires an in-depth understanding of the correlation between two intervals $Q([a,b])$ and $Q([c,d])$. This article is also the second part of the project extending the work in \cite{AJKS} to the case of inverse conformal welding, see \cite{AJKS,binder2023inverse} for the precise statement and motivation of the problem.

\subsection{Outline}
Here is an outline of the paper.
\begin{itemize}  
    \item In \cref{partdecoupling} we study the decoupling structure for the inverse. 
    \item In \cref{part:inverse_ratio_moments} we study the moments of the ratio $\frac{Q(a,a+x)}{Q(b,b+x)}$ because of its relevance to the dilatation in the \cite{AJKS}-framework.
    \item In \cref{partmultipointestimates} we use the decoupling structure to study the multipoint moments
    \begin{equation}
\Expe{\prod_{k}\para{Q^{k}(a_{k},b_{k})}^{p_{k}} } \tand \Expe{\prod_{k}\para{\frac{Q^{k}(a_{k},b_{k})}{Q^{k}(c_{k},d_{k})}}^{p_{k}} },   
\end{equation}
with $d_{k+1}<a_{k}<b_{k}<c_{k}<d_{k}$.

    \item In the appendix, we prove GMC estimates that we need for this work. 

\end{itemize}
\section{Main results}
The main results are the moments for the ratio $\frac{Q(a,b)}{Q(c,d)}$ and the product of the ratios.
\subsection{Single-point Ratio}
For the moments of the ratio, we go over the cases of decreasing numerator, equal-length intervals for $Q$ and equal-length intervals for $Q_{H}$, the inverse corresponding to the field $H$ in \cref{eq:covarianceunitcircle}. 
\begin{proposition}
We go over the cases of decreasing numerator, equal-length intervals for $Q$ and the same for $Q_{H}$.
\begin{itemize}
    \item Fix  $\beta\in  (0,0.152)$ and $\delta\leq 1$ . Assume $\abs{I}=r\delta$ for some $r\in (0,1)$ and $\abs{J}\to 0$. Then there exists $\e_{1},\e_{2}>0$, such that for each $p=1+\e_{1}$, one can find $q=1+\e_{2}\tand \tilde{q}>0$ such that
\begin{eqalign}
  \Expe{\para{\frac{Q^{\delta}\para{J}}{Q^{\delta}\para{I}}}^{p}}\lessapprox \para{\frac{\abs{J}}{\delta}}^{q} r^{-\tilde{q}}.
\end{eqalign}

 \item Fix $\gamma<\frac{2}{\sqrt{3}}$ and $\delta\leq 1$. Assume we have two equal-length intervals $J=(a,a+x),I=(b,b+x)\subset [0,1]$ with $b-a=c_{b-a}x$ for $c_{b-a}>1$  and $a=c_{1}\delta,b=c_{2}\delta$ and $x< \delta$.Then we have for all $p\in [1,1+\e_{1}]$ with small enough $\e_{1}>0$, a bound of the form
\begin{eqalign}
\maxp{\Expe{\para{\frac{Q(a,a+x)}{Q(b,b+x)}}^{p} },\Expe{\para{\frac{Q(b,b+x)}{Q(a,a+x)}}^{p} }}\leq c_{a,b,p,\delta }\para{\frac{x}{\delta}}^{-\e_{ratio}(p)},
\end{eqalign}
where $\e_{ratio}(p)\in (0,1)$ can be made arbitrarily small at the cost of larger comparison constant $c_{a,b,p,\delta }$. The constant $c_{a,b,p,\delta }$ is uniformly bounded in $\delta$. 

\item Fix $\gamma<\frac{2}{\sqrt{3}}$.  For the inverse $Q_{H}$ corresponding to the field $H$ in \Cref{eq:covarianceunitcircle}, we have the same singular estimate as above for $\delta=1$
\begin{eqalign}
\maxp{\Expe{\para{\frac{Q_{H}(a,a+x)}{Q_{H}(b,b+x)}}^{p} },\Expe{\para{\frac{Q_{H}(b,b+x)}{Q_{H}(a,a+x)}}^{p} }}\leq c_{a,b,p,\delta }x^{-\e_{ratio}(p)}.
\end{eqalign}

\end{itemize}
\end{proposition}
One major corollary of this proposition is that the dilatation of the inverse homeomorphism on the positive real line is in $L^{1}_{loc}$. This is the analogous result to \cite[Lemma 4.5]{AJKS}. There they studied the dilation $K$ for the Ahlfors-Beurling extension of the homeomorphism $h(x)=\frac{\tau(0,x)}{\tau(0,1)}$. Here we study the dilatation $K_{Q}$ for the Ahlfors-Beurling extension of the inverse map $Q^{1}(0,x)$ mapping $Q^{1}:\Rplus\to \Rplus$.
\begin{corollary}
We have finiteness
\begin{equation}
\int_{[0,1]\times [0,2]}K_{Q}(x+iy)\dx\dy\leq \frac{c}{2(1-\epsilon_{ratio}(1))}<\infty.    
\end{equation}
\end{corollary}

\subsection{Decoupling}
In the decoupling the goal is to estimate the multipoint estimates 
\begin{equation}
\Expe{\prod_{k}\para{\frac{Q^{k}(a_{k},b_{k})}{Q^{k}(c_{k},d_{k})}}^{q_{k}} } ,   
\end{equation}
with $d_{k+1}<a_{k}<b_{k}<c_{k}<d_{k}$, using a conditional independence result for the inverse. The ratios $\frac{Q^{1}(a_{1},b_{1})}{Q^{1}(c_{1},d_{1})},\frac{Q^{2}(a_{2},b_{2})}{Q^{2}(c_{2},d_{2})}$ can be decoupled into separate conditional factors  given the event that they are well-separated $G_{1,2}:=\set{Q^{2}(d_{2})+\delta_{2} <Q^{1}(a_{1})}$ 
\begin{equation}
\Expe{\frac{Q^{1}(a_{1},b_{1})}{Q^{1}(c_{1},d_{1})}\frac{Q^{2}(a_{2},b_{2})}{Q^{2}(c_{2},d_{2})}\ind{G_{1,2}}  }=\Expe{\Expe{\frac{Q^{1}(a_{1},b_{1})}{Q^{1}(c_{1},d_{1})}\ind{G_{1,2}} \mid \mathcal{F}}\frac{Q^{2}(a_{2},b_{2})}{Q^{2}(c_{2},d_{2})} },    
\end{equation}
see notations for detail on the filtration $\mathcal{F}$. These estimates are needed because they are the analogous of the decoupling
\begin{eqalign}
\Expe{\prod_{k}\para{\frac{\eta^{k}(a_{k},b_{k})}{\eta^{k}(c_{k},d_{k})}}^{q_{k}} }=\prod_{k}\Expe{\para{\frac{\eta^{k}(a_{k},b_{k})}{\eta^{k}(c_{k},d_{k})}}^{q_{k}} },    
\end{eqalign}
showing up in estimating the dilatation in \cite[equation (89)]{AJKS}. Starting with a collection of increments $\set{Q^{k}(a_{k},b_{k})}_{k\geq 1}$,we aim to extract a subsequence of them that will satisfy the gap events $Q_{a}^{k}-Q_{b}^{k+m}\geq \delta_{k+m}$ that we need in order to do decoupling i.e. obtaining with high probability a subsequence $S=\set{i_{1},\cdots, i_{\abs{S}}}\subset \set{1,\cdots,N}$ of size $\abs{S}\geq c N$, $c\in (0,1)$, such that for $i_{k}\in S$
\begin{eqalign}
&Q_{a_{i_{k}}}^{i_{k}}-Q_{b_{i_{k+1}}}^{i_{k+1}}\geq \delta_{i_{k+1}}.
\end{eqalign}
We turn this existence into a deviation estimate using graph theory and the concept of maximally independent subset. For $k,m\in [N]$ we consider the complement gap event i.e. the \textit{overlap} event
\begin{equation}
O_{k,m}:=G_{k,m}^{c}:=\ind{Q^{k\wedge m}(a_{k\wedge m})-Q^{m\vee k}(b_{m\vee k}) \leq \delta_{k\vee m}   }.    
\end{equation}
We then consider the random graph $G:=G_{Q}(N)$ with $N$ vertices labeled by $v_{k}:=\set{(Q^{k}(a_{k}),Q^{k}(b_{k})}$ and we say that $v_{k},v_{k+m}$  are connected by edges only when this event happens $G_{k,k+m}^{c}=1$. So the question of having a subsequence of increments that don't intersect is the same as obtaining the largest clique of vertices $S\subset G$ that do not connect to each other, this is called the \textit{independence number $\alpha(G)$ of $G$}.
\begin{theorem}
We show that the independence number being less than $c_{gap}N$ decays exponentially in $N$
\begin{equation}
\Proba{\alpha(G)< c_{gap}N}\lessapprox c \rho_{*}^{(1+\e_{*})N},
\end{equation}
for some $\e_{*}>0$.
\end{theorem}
This theorem shows that for each $N$, we can obtain such a subset $S\subset G$  with probability at least $1- c \rho_{*}^{(1+\e_{*})N}$.
\subsection{Multipoint moments and ratios}
From the above theorem for the independence number, for each $N\geq 1$ we are given a subset $S:=\set{i_{1},\cdots,i_{M}}\subset \set{1,\cdots,N}$ with $M\geq \alpha N$ such that we have large enough gaps
\begin{equation}
G_{S}:=G_{i_{1},\cdots,i_{M}}:=\bigcap_{k\in S}\set{Q_{a_{i_{k}}}^{i_{k}}-Q_{b_{i_{k+1}}}^{i_{k+1}}\geq \delta_{i_{k+1}}}.
\end{equation}
If we intersect with this event, we get the following estimate on the multipoint-moments.
\begin{proposition}
We fix $p_{k},k\in S$ then we have
\begin{eqalign}
\Expe{\prod\limits_{k\in S} \para{Q^{k}(J_{k})}^{p_{k}} \ind{G_{S}}}\leq c^{\abs{S}} \prod_{k\in S} x_{k}^{q_{k}},
\end{eqalign}
for $q_{k}>0$ satisfying
\begin{eqalign}
p_{k}>-\zeta(-q_{k})+1, k\in S.
\end{eqalign}
\end{proposition}
In the following proposition we study the first term to get moment estimates as in the single-ratio case.
\begin{proposition}[Product of ratios]
We fix set $S\subset [N]$.
\begin{itemize}
 \item (Equal length)  Fix $\beta<\frac{2}{3}$. Suppose we have equal length intervals $x_{k}=\abs{J_k}=\abs{I_k}<\delta_k$ for each $k\in S$. Fix $\delta_k\leq 1$ and intervals $J=(c_k,c_k+x_k),I=(d_k,d_k+x_k)\subset [0,1]$ with $d_k-c_k=c_{d-c,k}x_k$ for $c_{d-c,k}>1$  and $c_k=\lambda_{1,k}\delta_k,d_k=\lambda_{2,k}\delta_k$. Then we have for all $p_k\in [1,1+\e_{k}]$ with small enough $\e_{k}>0$, a bound of the form
\begin{eqalign}
\Expe{\prod\limits_{k\in S} \para{\frac{Q^{k}(J_{k})}{Q^{k}(I_{k})}}^{p_{k}} \ind{G_{S}}}\leq c^{\abs{S}}\prod\limits_{k\in S} \para{\frac{x_{k} }{\delta_{k}}}^{-\e_{ratio}(p_k)},  
\end{eqalign}
where $\e_{ratio}(p_k)>0$ can be made arbitrarily small at the cost of a larger proportionality constant $c$. The constants are uniformly bounded in $x_k$ and also in $\delta_k$.

    \item (Decreasing numerator) Fix $\beta\in \para{0,0.152}.$
    Suppose that $\abs{I_{k}}=r_{k}\delta_{k}$ for $r_{k}>0$ and $\abs{J_{k}}\to 0$ for each $k\in S$. Then we have for each $p_{k}\in [1,1+\e_{ratio}]$ some  $q_{k}>1$ such that
\begin{eqalign}
\Expe{\prod\limits_{k\in S} \para{\frac{Q^{k}(J_{k})}{Q^{k}(I_{k})}}^{p_{k}} \ind{G_{S}}}\leq \prod\limits_{k\in S} \para{\frac{\abs{J_{k}}}{\delta_{k}}}^{q_{k}}.
\end{eqalign}

\end{itemize}

\end{proposition}

\subsection{Acknowledgements}
First and foremost, we thank Eero Saksman and Antti Kupiainen. We had numerous useful discussions over many years. Eero asked us to compute the moments of the inverse and that got the ball rolling seven years ago. We also thank J.Aru, J.Junnila and V.Vargas for discussions and comments.

\section{Notations and GMC results }\label{notations}
\subsection{White noise expansion on the unit circle and real line }
Gaussian fields with logarithmic covariance indexed over time or even functions $f\in \mathbb{H}$, for some Hilbert space $\mathbb{H}$, have appeared in many places in the literature: for the \textit{Gaussian free field} which is a Gaussian field indexed over the $L^{2}$ Hilbert space equipped with the Dirichlet inner product \cite{sheffield2007gaussian,lodhia2016fractional,duplantier2017log} and for the \textit{Gaussian multiplicative chaos} which is about Gaussian fields with logarithmic covariance plus some continuous bounded function $\ln\frac{1}{\abs{x-y}}+g(x,y)$ \cite{robert2010gaussian,bacry2003log,rhodes2014gaussian,aru2020gaussian}. In this work we will work with the Gaussian field indexed over sets in the upper half-plane found in the works \cite{barral2002multifractal,bacry2003log} and used in the random welding context in \cite[section 3.2]{AJKS}. For the hyperbolic measure $\dlambda(x,y):=\frac{1}{y^{2}}\dx \dy$ in the upper half-plane $\uhp$ we consider a Gaussian process $\set{W(A)}_{A\in B_{f}(\uhp)}$ indexed by Borel sets of finite hyperbolic area:
\begin{eqalign}
    B_{f}(\uhp):=\{A\subset \uhp: \lambda(A)<\infty; \supl{(x,y),(x',y')\in A}|x-x'|<\infty\}
\end{eqalign}
with covariance
\begin{eqalign}
    \Expe{W(A_{1})W(A_{2})}:=\lambda(A_{1}\cap A_{2}).
\end{eqalign}
Its periodic version $W_{per}$ has covariance
\begin{eqalign}
\Expe{W_{per}(A_{1})W_{per}(A_{2})}:=\lambda(A_{1}\cap \bigcup_{n\in \Z} (A_{2}+n)).
\end{eqalign}
\subsubsection*{Logarithmic field on the real line}Consider the triangular region
\begin{eqalign}
\mathcal{S}:=\{(x,y)\in \uhp: x\in [-1/2,1/2]\tand y>2\abs{x}\}
\end{eqalign}
and define the field on the real line
\begin{eqalign}
U(x):=W(\mathcal{S}+x)\tfor x\in \mb{R}.
\end{eqalign}
It has logarithmic covariance (\cite{bacry2003log})
\begin{eqalign}
\Exp[U(x)U(y)]=\ln\para{\frac{1}{\min(\abs{x-y},1)}}.    
\end{eqalign}
This is the main field we will work with throughout this article. Because of the divergence along the diagonal, we upper and lower truncate by considering the field evaluated over shifts of the region
\begin{eqalign}
\mathcal{S}_{\e}^{\delta}(x_{0}):=\{(x,y)\in \uhp: x\in [x_{0}-\frac{\delta}{2},x_{0}+\frac{\delta}{2}]\tand \para{2\abs{x_{0}-x}}\vee \e<y\leq \delta\}.
\end{eqalign}
and $\mathcal{S}_{\e}^{\delta}=\mathcal{S}_{\e}^{\delta}(0)$. The covariance of $U_{\e}^{\delta}(x)=W(\mathcal{S}_{\e}^{\delta}+x)$ is the following (\cite{bacry2003log}).
\begin{lemma}\label{linearU}
The truncated covariance satisfies for $\delta\geq \e$ and all $x_{1},x_{2}\in \mathbb{R}$:
\begin{eqalign}\label{eq:Ucovariance}
\Exp[U_{ \varepsilon}^{  \delta }(x_{1} )U_{ \varepsilon}^{  \delta }(x_{2} )  ]=R_{ \varepsilon}^{  \delta }(\abs{x_{1}-x_{2}}):=\left\{\begin{matrix}
\ln(\frac{\delta }{\varepsilon} )-\para{\frac{1}{\e}-\frac{1}{\delta}}\abs{x_{2}-x_{1}} &\tifc \abs{x_{2}-x_{1}}\leq \varepsilon\\
 \ln(\frac{\delta}{\abs{x_{2}-x_{1}}}) +\frac{\abs{x_{2}-x_{1}}}{\delta}-1&\tifc \e\leq \abs{x_{2}-x_{1}}\leq \delta\\
 0&\tifc  \delta\leq \abs{x_{2}-x_{1}}.
\end{matrix}\right.    .
\end{eqalign}
In the case of $\e=0$, we shorthand write $R^{  \delta }(\abs{x_{1}-x_{2}})=R_{0}^{  \delta }(\abs{x_{1}-x_{2}})$.
\end{lemma}
\begin{remark}
This implies that we have a bound for its difference:
\begin{eqalign}
\Expe{\abs{U_{\varepsilon_{1}}^{\delta}(x_{1})-U_{\varepsilon_{2}}^{\delta}(x_{2})}^{2}}\leq 2 \frac{|x_{2}-x_{1}|+|\varepsilon_{1}-\varepsilon_{2}|}{\varepsilon_{2}\wedge \e_{1}}.
\end{eqalign}
Therefore, being Gaussian this difference bound is true for any $a>1$
\begin{eqalign}
\Expe{|U_{\varepsilon_{1}}^{\delta}(x_{1})-U_{\varepsilon_{2}}^{\delta}(x_{2})|^{a}}\leq c\para{ \frac{|x_{2}-x_{1}|+|\varepsilon_{1}-\varepsilon_{2}|}{\varepsilon_{2}\wedge \e_{1}}     }^{a/2}.
\end{eqalign}
and so we can apply the Kolmogorov-Centsov lemma from \cite[Lemma C.1]{hu2010thick}:
\begin{eqalign}
|U_{\varepsilon_{1}}(x_{1})-U_{\varepsilon_{2}}(x_{2})|\leq M~ (\ln(\frac{1}{\varepsilon_{2}}))^{\zeta}\frac{(|x_{2}-x_{1}|+|\varepsilon_{1}-\varepsilon_{2}|)^{\gamma}}{ \varepsilon_{2}^{\gamma+\varepsilon} },
\end{eqalign}
where $\frac{1}{2}< \frac{\e_{1}}{\e_{2}}<2$, $0<\gamma<\frac{1}{2}$, $\varepsilon,\zeta>0 $ and $M=M(\varepsilon,\gamma,\zeta)$.
\end{remark}
Also, we can consider strictly increasing larger heights i.e. $U^{h}_{0}$ denotes the field of height $h\geq \delta_{1}$. The infinitely large field  $U^{\infty}_{0}$ does not exist because it has infinite variance.\\

This field is needed because it allows us to work on the entire real line and not worry about the boundary effects of working on the unit circle $[0,1]/ 0\sim 1$.

\subsubsection*{Infinite cone field}
For the region
\begin{eqalign}
\mathcal{A}_{\e}^{\delta}:=\{(x,y)\in \uhp: x\in [-\frac{\delta}{2},\frac{\delta}{2}]\tand \para{2\abs{x}}\vee \e<y\}
\end{eqalign}
we let $U_{\omega,\e}^{\delta}:=\omega_{\e}^{\delta}(x):=W(\mathcal{A}+x)$. This field was considered in \cite{bacry2003log}. It has the covariance
\begin{eqalign}\label{eq:exactscalinglogafieldcova}
\Exp[\omega_{ \varepsilon}^{  \delta }(x_{1} )\omega_{ \varepsilon}^{  \delta }(x_{2} )  ]=\left\{\begin{matrix}
\ln(\frac{\delta }{\varepsilon} )+1-\frac{\abs{x_{2}-x_{1}} }{\e}&\tifc \abs{x_{2}-x_{1}}\leq \varepsilon\\
 \ln(\frac{\delta}{\abs{x_{2}-x_{1}}}) &\tifc \e\leq \abs{x_{2}-x_{1}}\leq \delta\\
  0&\tifc  \delta\leq \abs{x_{2}-x_{1}}
\end{matrix}\right.    .
\end{eqalign}
This field is interesting because it has "exact scaling laws" as we will see below. And so it helps in estimating the moments for the GMC corresponding to $U$ from above.
\subsubsection*{Trace of the Gaussian free field on the unit circle}\label{sec:fieldonunitcircle}
For the wedge shaped region
\begin{eqalign}
H:=\{(x,y)\in \uhp: x\in [-1/2,1/2]\text{ and }y>\frac{2}{\pi}tan(\pi x)\},
\end{eqalign}
we define the \textit{trace GFF} $H$ to be
\begin{eqalign}
H(x):=W_{per}(H+x), x\in \R/\Z.
\end{eqalign}
A regularized version of $H$ comes from truncating all the lower scales. Let $A_{r,\varepsilon}:=\{(x,y)\in \uhp: r>y>\varepsilon\}$ and $H_{\varepsilon}^{r}:=H\cap A_{r,\varepsilon}$ then we set:
\begin{eqalign}
H_{\varepsilon}^{r}(x):=W(H_{\varepsilon}^{r}+x)
\end{eqalign}
and $H_{\varepsilon}:=h^{\infty}_{\varepsilon}$. Its covariance is as follows: For points $x,\xi\in [0,1)$, with 0 and 1 identified,  and $\varepsilon\leq r$ we find for $y:=|x-\xi|$:
\begin{eqalign}\label{eq:covarianceunitcircle}
\Exp[H_{\varepsilon}(x)H_{r}(\xi)]:=&\branchmat{2\log(2)+\log\frac{1}{2sin(\pi y)}  &\tifc y> \frac{2}{\pi} arctan(\frac{\pi}{2}\e)\\ \log(1/\e)+(1/2)\log(\pi^{2}\varepsilon^{2}+4)+\frac{2}{\pi}\frac{arctan(\frac{\pi}{2}\e)}{\varepsilon}&\tifc y\leq \frac{2}{\pi} arctan(\frac{\pi}{2}\e)\\
-\log(\pi)- y \varepsilon^{-1}-\log(cos(\frac{\pi}{2}y))&}.
\end{eqalign}
From the above covariance computations, we find the difference bound for $1/2\leq \frac{\e_{1}}{\e_{2}}\leq 2$:
\begin{eqalign}
\Exp[|H_{\varepsilon_{1}}(x_{1})-H_{\varepsilon_{2}}(x_{2})|^{2}]\lessapprox~\frac{|x_{2}-x_{1}|+|\varepsilon_{1}-\varepsilon_{2}|}{\varepsilon_{2}\wedge \e_{1} }
\end{eqalign}
Therefore, by Kolmogorov-Centsov \cite[Lemma C.1]{hu2010thick}, there is a continuous modification with the modulus:
\begin{eqalign}
|H_{\varepsilon_{1}}(x_{1})-H_{\varepsilon_{2}}(x_{2})|\leq M (\ln(\frac{1}{\varepsilon_{1}}))^{\zeta}\frac{(|x_{2}-x_{1}|+|\varepsilon_{1}-\varepsilon_{2}|)^{\gamma}}{ \varepsilon_{1}^{\gamma+\varepsilon} },
\end{eqalign}
where $0<\gamma<\frac{1}{2}$, $\varepsilon,\zeta>0 $ and $M=M(\varepsilon,\gamma,\zeta)$. A similar modulus is true for the field $U(x)$ on the real line constructed above.\\
This field is important because of the original problem being the conformal welding on the unit circle. So all our estimates need to return to the GMC with the field $H$.

\subsubsection*{Measure and Inverse notations}
 For a field $X$, let us define the corresponding Gaussian multiplicative chaos (GMC) as
\begin{eqalign}
\eta_{X}(I):=\liz{\e}\int_{I}e^{\overline{X}_{\e}(x)}\dx,\quad \overline{X}_{\e}:=\gamma X_{\e}-\frac{\gamma^{2}}{2}E[X_{\e}^{2}],\quad I\subset [0,1) ,
\end{eqalign}
provided the limit exists.

For the real-line field $U^{\delta}$ we use the following notation to match \cite{AJKS}:
\begin{eqalign}
\eta^{\delta}(A):=\liz{\e}\eta_{\e}^{\delta}(A):=\liz{\e}\int_{A}e^{\gamma U_{\e}^{\delta}(x)-\frac{\gamma^{2}}{2}\ln\frac{1}{\e}}\dx, \quad A\subset \mathbb{R}.
\end{eqalign}
Its inverse $Q^{\delta}:\Rplus\to \Rplus$ is defined as
\begin{eqalign}
Q^{\delta}(x)=Q^{\delta}_{x}=:=\infp{t\geq 0 : \eta^{\delta}\spara{0,t}\geq x}
\end{eqalign}
and we will also consider increments of the inverse over intervals $I:=(y,x)$
\begin{eqalign}
Q^{\delta}(I)=Q^{\delta}(y,x):=Q^{\delta}(x)-Q^{\delta}(y).
\end{eqalign}
For the inverse of the lower-truncated GMC $\eta_{\e}^{\delta}$, we will write
$Q^{\delta}_{\e,x}:=Q^{\delta}_{\e}(x).$

In the ensuing articles we will work with a sequence of truncations $\delta_{n}\to 0$, and so we will shorthand denote $U^{n}:=U^{\delta_{n}}$, $\eta^{n}$and $Q^{n}$. This is because as in \cite{AJKS}, we need to decouple nearby measures: to have the following two objects  
\begin{eqalign}
\eta^{\delta_{n+1}}(a_{n+1},b_{n+1})\tand \eta^{\delta_{n}}(a_{n},b_{n})    
\end{eqalign}
be independent with $a_{n+1}<b_{n+1}<a_{n}<b_{n}$ decreasing to zero, we simply require $b_{n+1}+\delta_{n+1}<a_{n}$. So we see that as $\delta_{n}\to 0$, this becomes easier to achieve. Whereas if we kept $\delta_{n}=\delta$, this would be false as $a_{n}\to 0$.\\
Similarly, when lower truncating $U_{n}:=U_{\delta_{n}}$, we will write $\eta_{n}$ and $Q_{n}$. By slight abuse of notation, we denote
\begin{eqalign}
U(s)\cap U(t):=W(\para{U+s}\cap \para{U+t})    \tand U(s)\setminus U(t):=W(\para{U+s}\setminus \para{U+t}).
\end{eqalign}
\subsubsection*{Semigroup formula}\label{not:semigroupformula}
In the spirit of viewing $Q_{a}=Q(a)$ as a hitting time we have the following semigroup formula and notation for increments
\begin{eqalign}
Q^{\delta}(y,x)=Q^{\delta}(x)-Q^{\delta}(y)=\infp{t\geq 0: \etamu{Q^{\delta}(y),Q^{\delta}(y)+t}{\delta}\geq x-y }=:Q_{x-y}^{\delta }\bullet Q_{y}^{\delta}.
\end{eqalign}
Generally for any nonnegative $T\geq 0$ we use the notation
\begin{eqalign}
Q_{x}^{\delta} \bullet T:=\infp{t\geq 0: \etamu{T,T+t}{\delta}\geq x }.
\end{eqalign}
To be clear the notation $Q_{x}^{\delta }\bullet Q_{y}^{\delta}$ is \textit{not} equal to the composition $Q^{\delta }\circ Q_{y}^{\delta}=Q^{\delta }(Q_{y}^{\delta})=Q^{\delta }(a),$ which is the first time that the GMC $\eta$ hits the level $a:=Q_{y}^{\delta}$.

\subsubsection*{Scaling laws}
Let $\omega^{\delta,\lambda}_{\e}(x)$ be the field with the following covariance \begin{eqalign}\label{eq:truncatedscaledomega}
\Expe{\omega^{\delta,\lambda}_{\e}(x_{1})\omega^{\delta,\lambda}_{\e}(x_{2})}=\left\{\begin{matrix}
\ln(\frac{\delta }{\varepsilon} )+1-\frac{1}{\e}\abs{x_{2}-x_{1}}&\tifc \abs{x_{2}-x_{1}}\leq \varepsilon\\ ~\\
 \ln(\frac{\delta}{\abs{x_{2}-x_{1}}}) &\tifc \e\leq \abs{x_{2}-x_{1}}\leq \frac{\delta}{\lambda}\\~\\
  0&\tifc \frac{\delta}{\lambda}\leq \abs{x_{2}-x_{1}}
\end{matrix}\right.    .
\end{eqalign}
In our study of the scaling laws we will use the following statement (\cite[theorem 4]{bacry2003log}, \cite[proposition 3.3]{robert2010gaussian}.)
\begin{proposition}\label{exactscaling}\label{prop:GMClogscalinglaw}
For $\lambda\in (0,1)$ and fixed $x_{0}$ and all Borel sets $A\subset B_{\delta/2}(x_0)$ we have
\begin{eqalign}
 \set{\eta^{\delta}_{\omega}(\lambda A)}_{A\subset B_{\delta/2}(x_0)}\eqdis \set{\lambda e^{\overline{\Omega_{\lambda}}}\eta^{\delta,\lambda}_{\omega}(A)}_{A\subset B_{\delta/2}(x_0) },
\end{eqalign}
where  $\overline{\Omega_{\lambda}}:=\gamma\sqrt{ \ln(\frac{1}{\lambda})}N(0,1)-\frac{\gamma^{2}}{2}\ln(\frac{1}{\lambda})$ and the measure $\eta_{\omega_{\e}}^{\delta,\lambda}$ has the underlying field $\omega^{\delta,\lambda}_{\e}(x)$
\end{proposition}
 In the rest of the article, we let
\begin{eqalign}\label{logfield}
&G_{\lambda}:=\frac{1}{\lambda}\expo{-\overline{\Omega_{\lambda}}}.
\end{eqalign}
We will also use the field $U_{ \varepsilon}^{\delta, \lambda}$ with covariance $R_{\e}^{\delta,\lambda}(\abs{x_{1}-x_{2}}):=\Expe{U_{ \varepsilon}^{  \delta,\lambda }(x_{1} )U_{ \varepsilon}^{  \delta,\lambda }(x_{2} )  }$
\begin{eqalign}\label{eq:truncatedscaled}
R_{\e}^{\delta,\lambda}(\abs{x_{1}-x_{2}})=\left\{\begin{matrix}
\ln(\frac{\delta }{\varepsilon} )-\para{\frac{1}{\e}-\frac{1}{\delta}}\abs{x_{2}-x_{1}}+(1-\lambda)(1-\frac{\abs{x_{2}-x_{1}}}{\delta})&\tifc \abs{x_{2}-x_{1}}\leq \varepsilon\\ ~\\
 \ln(\frac{\delta}{\abs{x_{2}-x_{1}}})-1+\frac{\abs{x_{2}-x_{1}}}{\delta}+(1-\lambda)(1-\frac{\abs{x_{2}-x_{1}}}{\delta}) &\tifc \e\leq \abs{x_{2}-x_{1}}\leq \frac{\delta}{\lambda}\\~\\
  0&\tifc \frac{\delta}{\lambda}\leq \abs{x_{2}-x_{1}}
\end{matrix}\right.    .
\end{eqalign}
\begin{remark}\label{rem:negativecov}
We note that after $\abs{x_{2}-x_{1}}\geq \delta$ this covariance is negative and thus discontinuous at $\abs{x_{2}-x_{1}}=\frac{\delta}{\lambda}$ and so the GMC $\eta^{\delta,\lambda}$ cannot be defined (so far as we know in the literature, GMC has been built for positive definite covariances \cite{allez2013lognormal}). So we are forced to evaluate those fields only over sets $A$ with length $\abs{A}\leq \delta$ (see \cite[lemma 2]{bacry2003log} where they also need this restriction to define the scaling law.).
\end{remark}
 The $U_{ \varepsilon}^{\delta, \lambda}$ is related to $U^{\delta}_{\e}$ in the following way
\begin{eqalign}\label{eq:truncatedscalinglawGMC}
U^{\delta}_{\lambda\e}(\lambda x)\eqdis Z_{\lambda}+  U_{ \varepsilon}^{  \delta,\lambda }(x),
\end{eqalign}
where $Z_{\lambda}:=N(0,\ln\frac{1}{\lambda}-1+\lambda)$ is a Gaussian independent of $ U_{ \varepsilon}^{  \delta,\lambda }$. The corresponding lognormal is denoted by
\begin{eqalign}\label{logfieldshifted}
c_{\lambda}:=&\frac{1}{\lambda}\expo{-\overline{Z_{\lambda}}}.
\end{eqalign}
We can now restate \eqref{eq:truncatedscalinglawGMC}:
\begin{lemma}\label{lem:scalinglawudelta}[Scaling transformation for $U^{\delta}_{\e}$]
For $\lambda\in (0,1)$, $x_{0}\in\mathbb{R}$ and all Borel sets $A\subset B_{\delta/2}(x_0)$ we have
\begin{eqalign}
 \set{\eta^{\delta}_{U}(\lambda A)}_{A\subset B_{\delta/2}(x_0)}\eqdis \set{\lambda e^{\overline{Z_{\lambda}}}\eta^{\delta,\lambda}_{U}(A)}_{A\subset B_{\delta/2}(x_0) },  \quad \eta_{U}^{\delta,\lambda}:=\eta_{U^{\delta,\lambda}_{\e}}.\end{eqalign}
\end{lemma}
 For any $p\in \R$ we have the moment formula
\begin{eqalign}
\Expe{\expo{p\overline{Z_{\lambda}}}  }=\expo{p\beta r_{\lambda} (p-1)}.
\end{eqalign}
More generally, for positive continuous bounded function $g_{\delta}:\Rplus\to \Rplus$ with $g_{\delta}(x)=0$ for all $x\geq \delta$, we let
\begin{eqalign}
\Expe{U_{ \varepsilon}^{\delta,g }(x_{1} )U_{ \varepsilon}^{  \delta,g }(x_{2} )  }=\left\{\begin{matrix}
\ln(\frac{\delta }{\varepsilon} )-\para{\frac{1 }{\e}-\frac{1}{\delta}}\abs{x_{2}-x_{1}}+g_{\delta}(\abs{x_{2}-x_{1}})&\tifc \abs{x_{2}-x_{1}}\leq \varepsilon\\
 \ln(\frac{\delta}{\abs{x_{2}-x_{1}}})-1+\frac{\abs{x_{2}-x_{1}}}{\delta}+g_{\delta}(\abs{x_{2}-x_{1}}) &\tifc \e\leq \abs{x_{2}-x_{1}}\leq \delta\\
  0&\tifc \delta\leq \abs{x_{2}-x_{1}}
\end{matrix}\right.  .
\end{eqalign}

\subsubsection*{Singular integral}
Let the singular/fusion integrals for GMC measure be defined by (\cite[lemma A.1]{david2016liouville})
\begin{eqalign}\label{eq:singularintegralfusion}
\eta_{R(a)}(A):=\int_{A}e^{\gamma^{2}\Expe{U(s)U(a)}}e^{\gamma U(s)}\ds,\hfil A\subset \Rplus,\hfil a\in\Rplus.\end{eqalign}
As before, let $Q_{R}(x)$ denote its inverse.

Using the Girsanov/tilting-lemma (\cite[Lemma 2.5]{berestycki2021gaussian}) we have
\begin{eqalign}
\Expe{\ind{\eta(x)\geq t}e^{\thickbar{U}(a)}}=\Proba{\eta_{R(a)}(x)\geq t}.
\end{eqalign}

\subsubsection*{Filtration notations}
Following \cite[section 4.2]{AJKS}, for a Borel set $S\subset \mathbb{H}$ we define $\mathcal{U}_{S}=\mathcal{U}(S)$ to be the $\sigma$-algebra generated by the randoms variables $U(A)$, where $A$ runs over Borel subsets $A\subseteq S$. For short if $X\in \CU_{S}$ we will call such a variable \textit{measurable}. Moreover, since we will be using the triangular region $\mathcal{S}$ over intervals $[a,b]$, we define
\begin{eqalign}
\mathcal{U}_{\e}^{\delta}([a,b]):=\mathcal{U}(\bigcup_{x_{0}\in [a,b]}\mathcal{S}_{\e}^{\delta}(x_{0})).
\end{eqalign}
We fix decreasing sequence $(\delta_{n}=\rho_{*}^{n})_{n\geq 1}$ for some $\rho_{*}<1$. We denote the $\sigma-$algebra of upper truncated fields
\begin{eqalign}
\CU^{n}=\CU^{n}_{\infty}:=\CU_{\mathcal{S}^{\delta_{n}}_{0}}.
\end{eqalign}
All the lower truncations are measurable with respect to the filtration of the larger scales
\begin{eqalign}
U^{n}(s)=U^{\delta_{n}}_{0}(s)\in \CU^{k}([a,b]) \tfor n\geq k,s\in [a,b],
\end{eqalign}
and the same is true for the measure and its inverse
\begin{eqalign}
\set{\eta^{n}(a,b)\geq t}\in \CU^{k}([a,b])\tand \set{Q^{n}(0,x)\geq t}\in \CU^{k}([0,t]).  \end{eqalign}
We define measurability with respect to random times in a standard fashion. Namely, for a single time $Q^{k}(b)$ we have the usual notion of the stopping time filtration
\begin{eqalign}\label{eq:filtrationnotation}
\mathcal{F}\para{[0,Q^{k}(b)]}:=\left\{A\in \bigcup_{t\geq 0}  \CU^{k}([0,t]): A\cap\{Q^{k}(b) \leq t\}\in  \CU^{k}([0,t]), \forall t\geq 0\right\}.
\end{eqalign}
For example, $\CU^{n}\para{[0,Q^{k}(b)-s]}\subset \mathcal{F}\para{[0,Q^{k}(b)]}$ for all $n\geq k$ and $s\in [0,Q^{k}(b)]$. We also have, for $n\geq k$,
\begin{eqalign}
\set{Q^{k}(b)\geq Q^{n}(a)+s} \in  \mathcal{F}\para{[0,Q^{k}(b)]}.
\end{eqalign}

This perspective helps us achieve a decoupling. For two interval increments $Q^{k}(a_{k},b_{k}),Q^{n}(a_{n},b_{n})$ from possibly different scales $n\geq k$ , we we will need to consider the \textit{gap event}
\begin{eqalign}
G_{k,n}:=\set{Q^{k}(a_{k})-Q^{n}(b_{n})\geq \delta_{n}} \in  \mathcal{F}\para{[0,Q^{k}(a_{k})]}.
\end{eqalign}

 \part{Decoupling }\label{partdecoupling}\label{corrstruinver} 
 In this part we study the conditional decoupling for the increments of the inverse $Q^{\delta}(x)$ (for either field $U^{\delta}$ or $\omega^{\delta}$ because we only use their short-range correlations). We start by briefly discussing in a heuristic manner the main idea of decoupling the inverse.  This was inspired by the Brownian motion setting where increments of hitting times $T_{d}-T_{c}$ and $T_{b}-T_{a}$ are independent iff $c>b$ or $a>d$ \cite[theorem 2.35]{mortersperes2010brownian}. It was also inspired by the following conditional independence for Markov chains $S_{n}=\sum_{k=1}^{n} X_{k}$ (\cite[section 7.3]{chow2003probability} )
\begin{equation}\label{eq:randomwalkcondinde}
\Proba{S_{1}=i, S_{3}=j\conditional S_{2}=k}=\Proba{S_{1}=i\conditional S_{2}=k}\Proba{S_{3}=j\conditional S_{2}=k}.    
\end{equation}
In words the past $S_{1}$ and the future $S_{3}$ are conditionally independent given the present $S_{2}$.

We start with two inverse increments $Q^{k}(a_{k},b_{k}),Q^{n}(a_{n},b_{n})$ from possibly different scales $n\geq k$ and an event about their gap
\begin{equation}\label{gapeventkndelta}
G_{k,n}:=\set{Q^{k}(a_{k})-Q^{n}(b_{n})\geq \delta_{n}}.    
\end{equation}
In analogy with the random walk, this gap is thought of as the current step: $\set{S_{2}=k}$ and so conditional on the the gap $G_{k,n}$ we study the decoupling of $S_{1}:=Q^{n}(a_{n},b_{n})$ and $S_{3}:=Q^{k}(a_{k},b_{k})$. Here is a way we will use this idea. We start with applying the tower property with respect to $ \CF\para{[0,Q^{n}(b_{n})]}$:
\begin{eqalign}
 &\Proba{Q^{k}(a_{k},b_{k})\in A_{1},Q^{n}(a_{n},b_{n})\in A_{2},G_{k,n} }\\
 =&\Expe{\Proba{Q^{k}(a_{k},b_{k})\in A_{1},G_{k,n}\mid  \CF\para{[0,Q^{n}(b_{n})]}} \ind{Q^{n}(a_{n},b_{n})\in A_{2}} },
\end{eqalign}
where $A_{1},A_{2}$ are Borel subsets of $\Rplus$. Next we further clarify the source of the correlation between $S_{1}$ and $S_{3}$ via the gap $G_{k,n}$. We take discrete dyadic approximations $T_{N}(x)\downarrow Q^{k}(x)$:
\begin{equation}
T_{N}(x):=\frac{m+1}{2^{N}}\twhen \frac{m}{2^{N}}\leq Q^{k}(x) < \frac{m+1}{2^{N}}, m\geq 0.
\end{equation}
By fixing $T_{N}(a_{k})=\ell$, we have the gap event $G_{k,n}^{N}:=\set{\ell-Q^{n}(b_{n})\geq \delta_{n}}\in \CF\para{[0,Q^{n}(b_{n})]}$ and so we are left with
\begin{eqalign}
 &\sum_{\ell\in D_{N}}\Expe{\Proba{Q^{k}_{b_{k}-a_{k}}\bullet \ell\in A_{1},T_{N}(a_{k})=\ell\mid  \CF\para{[0,Q^{n}(b_{n})]}} \ind{G_{k,n}^{N},Q^{n}(a_{n},b_{n})\in A_{2}} },
\end{eqalign}
where  we used the semigroup notation
\begin{equation}   
Q^{k}(x)-Q^{k}(y)=\infp{t\geq 0: \etamu{Q^{k}(y),Q^{k}(y)+t}{k}\geq x-y }=:Q_{x-y}^{k }\bullet Q_{y}^{k}.
\end{equation}   
But here because of the gap event $\set{\ell\geq Q^{n}(b_{n})+ \delta_{n}}$, we get that 
\begin{equation}
Q^{k}_{b_{k}-a_{k}}\bullet \ell=\infp{t\geq 0: \etamu{\ell,\ell+t}{k}\geq b_{k}-a_{k} }
\end{equation}
is \textit{independent} of the filtration $ \CF\para{[0,Q^{n}(b_{n})]}$. So the correlation got localized through the decomposed event $\set{T_{N}(a_{k})=\ell}$. This is the essential decoupling that we will use. It comes at the cost of obtaining the gap events with high probability and so we next study the error of not having such gaps.
\section{Decay of overlap }
 In this section we estimate the probability that the gap between $Q^{k+m}(b_{k+m})$ and $Q^{k}(a_{k})$ is small:
\begin{equation}
\Proba{Q^{k}(a_{k})-Q^{k+m}(b_{k+m})< \delta_{k+m}}\to 0\tas m\to +\infty.     
\end{equation}
This says that the decoupling from above becomes asymptotically true in smaller scales. It shows that the complement of the gap event from \cref{gapeventkndelta}
\begin{equation}
G_{k,k+m}^{c}=\set{Q^{k}(a_{k})-Q^{k+m}(b_{k+m})\leq \delta_{k+m}}    
\end{equation}
decays fast in probability and so it is reasonable that the intervals $Q^{k}(a_{k},b_{k})$ and $Q^{k+m}(a_{k+m},b_{k+m})$ decorrelate for large $m$. This kind of \textit{asymptotic independence} was crucial in \cite{AJKS}, where even though the annuli were correlated, after removing the upper-scales, the lower scales exhibited decoupling.
\begin{proposition}\label{prop:decayoverlap}\label{rem:exponentialchoices}
We fix $k,m\geq 1$. We let $\delta_{k}:=\rho_{\delta}^{n}$ for some $\rho_{\delta}\leq \rho_{*}$. Then the probability for a small gap can be bounded by:
\begin{eqalign}\label{eq:decayofoverlap}
\Proba{Q^{k}(a_{k})-Q^{k+m}(b_{k+m})<\delta_{k+m}}\lessapprox&  \para{a_k-g_{k,m}-b_{k+m}}^{-2}  \delta_{k}^{p_{1}}b_{k+m}^{2-p_{1}}\\
&+\para{a_{k}-g_{k,m}-b_{k+m}}^{-p_{2}}\delta_{k}^{p_{2}} \para{\frac{\delta_{k+m}}{\delta_{k}}}^{\zeta(p_{2})-1}\\
=:&e^{gap}_{k,k+m}, 
\end{eqalign}
where $p_{1}\in (0,1),p_{2}\in [1,\beta^{-1})$, and the $\e$'s are arbitrarily small and sequence $g=g_{k,m}:=\rho_{g}\rho_{1}^{k}\rho_{2}^{m}$ for some $\rho_{1}\in [\rho_{\delta},\rho_{*}],\rho_{2}\in (0,1)$. Using the \nameref{def:exponentialchoiceparam} for general $\rho_{\delta}\leq \rho_{*}$, we bound
\begin{eqalign}
&\Proba{Q^{k}(a_{k})-Q^{k+m}(b_{k+m})<\delta_{k+m}}\\
\lessapprox &   \para{\rho_{a}- \para{\frac{\rho_{1}}{\rho_{*}}}^{k}\rho_{2}^{m}-\rho_{b}\rho_{*}^{m}}^{-2} \para{\frac{\rho_{\delta}}{\rho_{*}}}^{p_{1}k}\rho_{*}^{(2-p_{1})m}+\frac{1}{\rho_{g}^{p_{2}}} \para{\frac{\rho_{\delta}}{\rho_{1}}}^{p_{2}k}\para{\frac{\rho_{\delta}^{}}{\rho_{2}^{p_{2}}}}^{m}.
\end{eqalign}
\end{proposition}
\begin{remark}
 The main choice in this article is $\rho_{\delta}=\rho_{*}$. And the above estimate cannot be used directly to obtain a subset of decoupled intervals. So we just include it to showcase the phenomenon of the exponential decay of having overlaps.
\end{remark}
\begin{remark}
For simplicity we let
\begin{eqalign}\label{eq:f1f2def}
f_{k,m}:=\maxp{\para{\frac{\rho_{\delta}}{\rho_{*}}}^{p_{1}k}\rho_{*}^{(2-p_{1})m},\para{\frac{\rho_{\delta}}{\rho_{1}}}^{p_{2}k} \para{\frac{\rho_{\delta}^{\zeta(p_{2})-1}}{\rho_{2}^{p_{2}}}}^{m}}.    
\end{eqalign}
We see that $f_{k,m}$ can be naturally represented as $f_{1}^{k} f_{2}^{m}$ for some $f_{1}\in (0,1],f_{2}\in (0,1)$.    
\end{remark}
\begin{proof} 
A recurring way to write the overlap event in terms of shifted GMCs is as follows
\begin{eqalign}\label{eq:maingapprogGMCinvert}
&\Proba{Q^{k}(a_{k})-Q^{k+m}(b_{k+m})< \delta_{k+m}}=\Proba{a_k-\eta^{k}(Q^{k+m}(b_{k+m})\leq \int_{0}^{\delta_{k+m}}e^{U^{k}(s+Q^{k+m}(b_{k+m})}\ds}.    
\end{eqalign}
Here we remove the upper bound 
\begin{eqalign}
\eqref{eq:maingapprogGMCinvert}\leq &\Proba{a_k-g_{k,m}-b_{k+m}\leq \eta^{k}(Q^{k+m}(b_{k+m})-b_{k+m}}\\
&+\Proba{g_{k,m}\leq \int_{0}^{\delta_{k+m}}e^{U^{k}(s+Q^{k+m}(b_{k+m})}\ds},    
\end{eqalign}
for our choice of $g_{k,m}$. For the first term we use \cref{lemm:precomposedetainvuppertrun} to bound it by
\begin{eqalign}\label{eq:l2estimatepreco}
\frac{1}{\para{a_k-g_{k,m}-b_{k+m}}^{2}}\Expe{\para{\eta^{k}(Q^{k+m}(b_{k+m})-b_{k+m}}^{2}}\lessapprox& \frac{1}{\para{a_k-g_{k,m}-b_{k+m}}^{2}}\delta_{k}^{p_{1}}b_{k+m}^{2-p_{1}},
\end{eqalign}
for $p_{1}\in (0,1)$. For the second term we use scaling by $\delta_{k}$ and \cref{cor:shiftedGMCmoments} for some $p_{2}\in [1,\beta^{-1})$ to upper bound
\begin{equation}
\Proba{g_{k,m}\leq \int_{0}^{\delta_{k+m}}e^{U^{k}(s+Q^{k+m}(b_{k+m})}\ds}\lessapprox g_{k,m}^{-p_{2}} \delta_{k}^{p_{2}} \para{\frac{\delta_{k+m}}{\delta_{k}}}^{\zeta(p_{2})-1}. 
\end{equation}
Using the \nameref{def:exponentialchoiceparam}
the first term is
\begin{eqalign}
\frac{1}{\para{\rho_{a}- \para{\frac{\rho_{1}}{\rho_{*}}}^{k}\rho_{2}^{m}-\rho_{b}\rho_{*}^{m}} ^{2}} \para{\frac{\rho_{\delta}}{\rho_{*}}}^{p_{1}k}\rho_{*}^{(2-p_{1})m},
\end{eqalign}
and the second term is
\begin{eqalign}
\frac{1}{\rho_{g}^{p_{2}}} \para{\frac{\rho_{\delta}}{\rho_{1}}}^{p_{2}k}\para{\frac{\rho_{\delta}^{\zeta(p_{2})-1}}{\rho_{2}^{p_{2}}}}^{m}.    
\end{eqalign}
\end{proof}
\begin{remark}
One possible room for sharpening this estimate is to write
\begin{eqalign}
&\Proba{g\leq \int_{0}^{\delta_{k+m}}e^{U^{k}(s+Q^{k+m}(b_{k+m})}\ds}    \\
=&\Proba{g-\delta_{k+m}\leq \int_{0}^{\delta_{k+m}}e^{U^{k}(s+Q^{k+m}(b_{k+m})}\ds-\delta_{k+m}}\\
\leq&\frac{1}{\para{g-\delta_{k+m}}^{2}} \Expe{\para{\int_{0}^{\delta_{k+m}}e^{U^{k}(s+Q^{k+m}(b_{k+m})}\ds-\delta_{k+m}}^{2}}.
\end{eqalign}
However, it is unclear to us how to estimate this difference  because $U^{k}(s+Q^{k+m}(b_{k+m})$ is not necessarily Gaussian anymore. If one uses the supremum, then we can do a chaining argument as in maximum estimates in \cite{binder2023inverse} but including the indicator
\begin{equation}
X_{T}:=\sup_{T}\ind{\eta(T,T+x)-x\geq g}\para{\eta(T,T+x)-x}    
\end{equation}
in order to preserve positivity.
\end{remark}

\section{The independence number of graphs   }\label{indnumber}
In this section we study the family of increments $\set{Q^{k}(a_{k},b_{k})}_{k\geq 1}$ and aim to extract a subsequence of them that will satisfy the gap events $Q_{a}^{k}-Q_{b}^{k+m}\geq \delta_{k+m}$ that we need in order to do decoupling i.e. obtaining with high probability a subsequence $S=\set{i_{1},\cdots, i_{\abs{S}}}\subset \set{1,\cdots,N}$ of size $\abs{S}\geq c N$, $c\in (0,1)$, such that for $i_{k}\in S$
\begin{eqalign}
&Q_{a_{i_{k}}}^{i_{k}}-Q_{b_{i_{k+1}}}^{i_{k+1}}\geq \delta_{i_{k+1}}.
\end{eqalign}
We turn this existence into a deviation estimate using graph theory and the concept of maximally independent subset.
For $k,m\in [N]$ we consider the complement gap event i.e. the \textit{overlap} event
\begin{equation}
O_{k,m}:=G_{k,m}^{c}:=\ind{Q^{k\wedge m}(a_{k\wedge m})-Q^{m\vee k}(b_{m\vee k}) \leq \delta_{k\vee m}   }.    
\end{equation}
We then consider the random graph $G:=G_{Q}(N)$ with $N$ vertices labeled by $v_{k}:=\set{(Q^{k}(a_{k}),Q^{k}(b_{k})}$ and we say that $v_{k},v_{k+m}$  are connected by edges only when this event happens $G_{k,k+m}^{c}=1$. So the question of having a subsequence of increments that don't intersect is the same as obtaining the largest clique of vertices $S\subset G$ that do not connect to each other, this is called the \textit{independence number $\alpha(G)$ of $G$}. The independence number has the following lower bounds \cite{larson2017independence}
\begin{equation}
\alpha(G)\geq \frac{N}{1+\thickbar{d}}, \alpha(G)\geq \frac{N}{\Delta_{N}+1}\tand    \alpha(G)\geq\sum_{k=1}^{N}\frac{1}{d(v_{k})+1},
\end{equation}
where the kth-degree is $d(v_k):=\sum_{m\in[N]\setminus \set{k} } O_{k,m}$, the average degree is $\thickbar{d}:=\frac{\sum_{k=1}^{N}d(v_{k})}{N}$ and the maximum degree is $\Delta_{N}:=\maxl{k\in[N] } d(v_{k})$. The sharpest bound is the third one and is called Caro-Wei bound. So returning to our decoupling, by showing that the event $\alpha(G)\geq c N$, for some constant $c\in (0,1)$, has high probability as $N\to +\infty$,  we obtain $cN$ large enough gaps $\set{Q_{a_{i_{k}}}^{i_{k}}-Q_{b_{i_{k+m}}}^{i_{k+m}}\geq \delta_{i_{k+m}}}$ for $i_{j}\in S$, in particular
\begin{equation}\label{eq:decayofsmallindependennumber}
\Proba{\alpha(G)\leq c N}\leq c_{1} e^{-c_{2}N},    
\end{equation}
where $c_{2}$ can be arbitrarily large by adjusting the parameters in \nameref{def:exponentialchoiceparamannul}.
\subsection{Maximum degree }
For independent interest, here we will study the tail of the maximum $\Delta_{N}:=\maxl{k\in[N] } d(v_{k})$ using a tail for sum of indicators \cite{linial2014chernoff}.
\begin{theorem} \label{lp_theorem}
Let $X_1,\ldots,X_n$ be indicator random variables, $0 < \beta < 1$ and $ 0 < k < \beta n $.
Then
\begin{equation}\label{main}
\Pr \left( \sum_{i=1}^n{X_i}\geq \beta n\right) \leq
\frac{1}{\binom{\beta n}{k}}\sum_{|S|=k}{\Pr\left(\wedge_{i \in S} (X_i = 1) \right)}.
\end{equation}
In particular, if $\Pr\left(\wedge_{i \in S} (X_i = 1) \right) \leq \alpha^k$ for every $S$ of size $k = \left(\frac{\beta-\alpha}{1-\alpha}\right)n$, where $0 < \alpha < \beta$ then
\begin{equation}
\Pr \left( \sum_{i=1}^n{X_i}\geq \beta n\right)  \leq e^{-D(\beta||\alpha) n},    
\end{equation}
where $D(\beta||\alpha):=\beta\ln\frac{\beta}{\alpha}+(1-\beta)\ln\frac{1-\beta}{1-\alpha}$ is the Kullback-Leibler distance between $\alpha,\beta$.
\end{theorem}

 By the union bound we study
\begin{equation}
 \Proba{\maxl{k\in[N] } d(v_{k})\geq \beta N}\leq \sum_m\Proba{d(v_{m})\geq \beta N}.
\end{equation}
For every set $S\subset [N]$ with large enough cardinality $\abs{S}=k>\frac{N}{2}$, there exists an index $m_{*}\in S$ that is the most distant scale $m_{*}(m,S):=\arg\max\limits_{i\in S}\abs{m-i}$ such that $\abs{m-m_{*}}\geq r \abs{S}=rk$ for some uniform in $S$ fixed constant $r<\frac{1}{2}$. Moreover, as explained in \cref{rem:exponentialchoices}, if we use exponential choices we get uniform exponential bound $e^{Q}_{k,k+m}\leq cf_{1}^{k}f_{2}^{m}$ and so we further bound 
\begin{equation}
\Pr\left(\wedge_{i \in S} (G_{m,i}^{c} = 1) \right)\leq c f_{2}^{\abs{m-m_{*}}} \leq cf_{2}^{rk}.    
\end{equation}
Therefore, the \cref{lp_theorem} implies
\begin{equation}
\Proba{\max_{m}d(v_{m})\geq \beta N}\leq cNe^{-rD(\beta||f_{2}) N}.    
\end{equation}
However, we will need a stronger estimate to get \cref{eq:decayofsmallindependennumber}. This is because to get a linearly-growing independent-clique we need for the maximum degree to behave almost as a constant
\begin{equation}
   \Proba{\alpha(G)\leq c_{N}N}\leq \Proba{\maxl{k\in[N] } d(v_{k})\geq \frac{1}{c_{N}}-1}, 
\end{equation}
for some sequence $ c_{N}$. So for this probability to decay exponentially, the maximum needs to behave almost as a constant or the sequence $c_{N}$ needs to go to zero and hence we lose the linear growth for $\alpha(G)$.
\subsection{Average degree and mean }
Another bound could come from the average degree. In the event 
\begin{equation}
\thickbar{d}=\frac{\sum_{k=1}^{N}d(v_{k})}{N}\geq \alpha \doncl 2\sum_{k=1}^{N}\sum_{m>k}^{N}\ind{Q^{k}(a_{k})-Q^{m}(b_{m})\leq \delta_{m}}\geq \alpha N    
\end{equation}
the lower bound is linear in $N$ and the sum has order $N^{2}$, so we would need a large number $N^{2}-\alpha N=N(N-\alpha)$ of indicator terms to be zero. This is reasonable because its mean is uniformly bounded.
\begin{lemma}\label{lem:unifboundedmean}
For $L\leq N$ we have the formula
\begin{eqalign}
\sum_{k=L}^{N}\sum_{m>k}^{N}\Proba{Q^{k}(a_{k})-Q^{m}(b_{m})\leq \delta_{m}}&\leq C \frac{f_{1}^{L}f_{2}}{(1-f_{1})(1-f_{2})},
\end{eqalign}
for the $f_{i}\in (0,1)$ from \cref{rem:exponentialchoices}. So for $L=1$ we have \begin{eqalign}
\sum_{k=1}^{N}\sum_{m>k}^{N}\Proba{Q^{k}(a_{k})-Q^{m}(b_{m})\leq \delta_{m}}&\leq C \frac{f_{1}f_{2}}{(1-f_{1})(1-f_{2})}.
\end{eqalign}
\end{lemma}
\begin{proof}
Once again we use the exponential choice in \cref{rem:exponentialchoices} that bound the gap
\begin{equation}
\Proba{Q^{k}(a_{k})-Q^{k+m}(b_{k+m})\leq \delta_{k+m}}\lessapprox f_{1}^{k}f_{2}^{m}.
\end{equation}
So we get
\begin{equation}
\sum_{k=L}^{N}\sum_{m=k+1}^{N}f_{1}^{k}f_{2}^{m-k}= f_{2}f_{1}^{L}\frac{-f_{2}\para{1-f_{2}^{N-L}}+f_{1}\para{1-f_{2}^{N+1-L}}-\para{1-f_{2}}f_{1}^{N+1-L}}{(1-f_{1})(1-f_{2})(f_{1}-f_{2})}.
\end{equation}
We take $N$ to infinity to be left with
\begin{equation}
\frac{f_{1}^{L}f_{2}}{(1-f_{1})(1-f_{2})}\tand \frac{f_{1}f_{2}}{(1-f_{1})(1-f_{2})}.
\end{equation}
\end{proof}
\begin{remark}\label{rem:averagedegreedeviation}
So we see here that if we restricted to the graph $G_{L,N}$ with vertices in $[L,N]$ and set $L:=\alpha_{0}N$ for some $\alpha_{0}\in (0,1)$, then by Markov we get
\begin{equation}
\Proba{\frac{\sum_{k=\alpha_{0}N}^{N}d(v_{k})}{N}\geq \alpha } \leq \frac{1}{\alpha N} \frac{f_{1}^{\alpha_{0}N}f_{2}}{(1-f_{1})(1-f_{2})},   
\end{equation}
which indeed it decays exponentially in $N$. This means that we get with high probability
\begin{equation}
\alpha(G_{L,N})\geq \frac{N-L}{1+\alpha}=\frac{(1-\alpha_{0})N}{1+\alpha} .   
\end{equation}
We tested getting a large deviation result
\begin{equation}
\Proba{2\sum_{k=1}^{N}\sum_{m>k}^{N}\ind{Q^{k}(a_{k})-Q^{m}(b_{m})\leq \delta_{m}}\geq \alpha N }\lessapprox e^{-cN}   
\end{equation}
using projective operators \cite[theorem 3.20]{merlevede2019functional} but we got negative results because those estimates often assume a $L_{\infty}$-bounded object whereas here if we use filtrations we end up with semi-random GMC objects and thus lognormals which are unbounded. 
\end{remark}
\section{Existence of Gap event Estimate}
In this section we obtain the exponential decay of the independent number being small.
\begin{theorem}\label{indnumbertheorem}\label{thm:gapeventexistence}
We show that the independence number being less than $c_{gap}N$ decays exponentially in $N$
\begin{equation}
\Proba{\alpha(G)< c_{gap}N}\leq c \rho_{*}^{(1+\e_{*})N},
\end{equation}
for some $\e_{*}>0$, given the following constraint
\begin{eqalign}\label{eq:constraintdecouplingbeta}
\para{\frac{(\beta+1)^{2}}{4\beta}-\frac{(1+\e_{*})}{1-c_{gap}}}-\frac{\beta^{-1}+1}{2}r_{a}> \frac{(1+\e_{*})}{1-c_{gap}}.
\end{eqalign}
\end{theorem}
\begin{remark}
The constraint \cref{eq:constraintdecouplingbeta} is a quadratic inequality in $\beta$ that has a solution for $\beta<0.171452$  and some small enough $\e_{*},c_{gap}$.
\end{remark}

\begin{proof}
\proofparagraph{Case $\rho_{\delta}<\rho_{*}$}
Recall $\delta_{m}:=\rho_{\delta}^{m}$. Even though we are not using $\rho_{\delta}<\rho_{*}$ in this article, we do this case because the proof is very short and so it gives an idea of the underlying phenomenon being the decay of the overlap. We study the following event and its complement going over all totally isolated subsets of the graph
\begin{eqalign}
E_{\alpha(G)}:=\bigcupls{S\subset [1,N]\\\abs{S}\geq c_{gap}N}\bigcapls{k,m\in S\\ m>k}\set{Q^{k}(a_{k})-Q^{m}(b_{m})\geq \delta_{m}}.    
\end{eqalign}
Since this event gives the existence of an independent set with cardinality larger than $c_{gap}N$, we have
\begin{eqalign}\label{eq:intersectioneventoflargesets}
 \Proba{\alpha(G)< c_{gap}N}=& \Proba{E_{\alpha(G)}^{c}}   \\
 =&\Proba{\bigcapls{S\subset [1,N]\\\abs{S}\geq c_{gap}N}\bigcupls{k,m\in S\\ m>k}\set{Q^{k}(a_{k})-Q^{m}(b_{m})\leq \delta_{m}}}.
\end{eqalign}
Here we only keep the leftmost subset
\begin{equation}\label{eq:sstarset}
S_{*}:=\set{N,...,(1-c_{gap})N-1}    
\end{equation}
and then use union bound and \cref{lem:unifboundedmean}
\begin{eqalign}
 \eqref{eq:intersectioneventoflargesets} \leq\sumls{k,k+m\in S_{*}}\Proba{Q^{k}(a_{k})-Q^{k+m}(b_{k+m})\leq \delta_{k+m}}
 &\leq \sum_{k=(1-c_{gap})N-1}^{N}\sumls{m=1}^{N-k}f_{1}^{k}f_{2}^{m}\\
 \lessapprox&\frac{f_{1}^{(1-c_{gap})N}f_{2}}{(1-f_{1})(1-f_{2})}.
\end{eqalign}
\proofparagraph{Case $\rho_{\delta}=\rho_{*}$}
We start with splitting over all possibilities of positive/zero for the sums of edge-weights with $m>k$
\begin{eqalign}\label{eq:avch10}
\Proba{\alpha(G)< c_{gap}N   }=\sum_{S\subset [N]}    \Proba{\alpha(G)< c_{gap}N, \bigcap_{i\in S}\set{Y_{i}\geq 1} , \bigcap_{i\in S^{c}}\set{Y_{i}= 0}   },
\end{eqalign}
for 
\begin{eqalign}
Y_{k}:=\sum_{m>k}^{N}\ind{O_{k,m}}:=\sum_{m>k}^{N}\ind{Q^{k}(a_{k})-Q^{m}(b_{m})\leq \delta_{m}}.    
\end{eqalign}
If we have $\abs{S^{c}}\geq c_{gap}N$ , then we obtain that the vertices corresponding to $S^{c}:=\set{j_{1},...,j_{R}}$ do not overlap with each other. This is because the zero sum $Y_{j_{1}}=0$ implies that the node $v_{j_{1}}$ doesn't have any edges with $v_{j_{2}},...,v_{j_{R}}$; the zero sum $Y_{j_{2}}=0$ similarly implies that $v_{j_{2}}$ doesn't have any edges with $v_{j_{3}},...,v_{j_{R}}$ and so on and so forth. But this not allowed due to the event $\alpha(G)< c_{gap}N$. Therefore, we reduce the sum over only subsets that satisfy $\abs{S^{c}}\leq  c_{gap}N\doncl \abs{S}\geq (1-c_{gap}) N$
\begin{eqalign}
\eqref{eq:avch10}=\sumls{S\subset [N]\\\abs{S}\geq (1-c_{gap}) N}    \Proba{\alpha(G)<c_{gap}N, \bigcap_{i\in S}Y_{i}\geq 1 , \bigcap_{i\in S^{c}}Y_{i}= 0   }.
\end{eqalign}
So it remains to study the probability term
\begin{equation}\label{eq:Ykdecompseventprob}
\Proba{\bigcap_{i\in S} \set{\sum_{m>i, m\in S }\ind{O_{i,m}} \geq 1}}=\Proba{\bigcap_{i\in S} \bigcup_{m_{i}>i, m\in S } O_{i,m}},    
\end{equation}
for  $S:=\set{s_{1},...,s_{M}}\subset [N]$ with $s_{i}<s_{i+1}$ and $M=  \ceil{(1-c_{gap}) N}$.
\proofparagraph{Step 1: Independent and Shrinking intervals}
Here we use a similar strategy as done for the terms $L_{n,m}$ in \cite[section 4.3]{AJKS}. Namely we invert back to GMC and split into shrinking intervals. We insert the deviation estimate
\begin{equation}
\set{Q^{m}(b_{m})\leq g_{m,k}},    
\end{equation}
for sequence $g_{m,k}:=\rho_{g}^{m-k}\rho_{*}^{k}$ with $\rho_{g}>\rho_{*}$, to get the inclusion
\begin{eqalign}
\set{Q^{k}(a_{k})-Q^{m}(b_{m})\leq \delta_{m}}\subseteq &\set{Q^{k}(a_{k})\leq g_{m,k}+\delta_{m}}\cup\set{Q^{m}(b_{m})\geq g_{m,k}}\\
=&\set{a_{k}\leq \etaum{k}{0,g_{m,k}+ \delta_{m}}}\cup\set{b_{m}\geq \etaum{m}{0,g_{m,k}}}.
\end{eqalign}
So we reduced to studying the tail and small-ball of GMC respectively. We decompose 
\begin{equation}
\spara{0,g_{m,k}+ \delta_{m}}=\bigcup_{\ell\geq 1}I_{m,k,\ell}\tand   \spara{0,g_{m,k}}=\bigcup_{\ell\geq 1}J_{m,k,\ell},  
\end{equation}
for sequences of partition intervals $I_{m,k,\ell}=\spara{x_{m,k,\ell},x_{m,k,\ell}+y_{m,k,\ell}}$ and $J_{m,k,\ell}=\spara{z_{m,k,\ell},z_{m,k,\ell}+w_{m,k,\ell}}$. In particular, we let
\begin{eqalign}
x_{m,k,\ell}:=&\rho_{*}(g_{m,k}+\delta_{m})\rho_{*}^{\ell}=\rho_{*}\para{\rho_{g}^{m-k}\rho_{*}^{k}+\rho_{*}^{m}}\rho_{*}^{\ell},   \\
y_{m,k,\ell}:=&(1-\rho_{*}) \para{\rho_{g}^{m-k}\rho_{*}^{k}+\rho_{*}^{m}}\rho_{*}^{\ell},   \\
z_{m,k,\ell}:=&\rho_{*}\rho_{g}^{m-k}\rho_{*}^{k}\rho_{*}^{\ell},   \\
w_{m,k,\ell}:=&(1-\rho_{*}) \rho_{g}^{m-k}\rho_{*}^{k+\ell} 
\end{eqalign}
to obtain the partition constraints
\begin{eqalign}
x_{m,k,\ell+1}+y_{m,k,\ell+1}=x_{m,k,\ell}\text{ since } \rho_{*}^2+(1-\rho_{*})\rho_{*} =\rho_{*}, 
\end{eqalign}
and
\begin{eqalign}
\sum_{\ell\geq 0}y_{m,k,\ell}=  g_{m,k}+\delta_{m},  
\end{eqalign}
and same for $J_{m,k,\ell}$
\begin{eqalign}
z_{m,k,\ell+1}+w_{m,k,\ell+1}=z_{m,k,\ell}\text{ since } \rho_{*}^2+(1-\rho_{*})\rho_{*} =\rho_{*}, 
\end{eqalign}
and
\begin{eqalign}
\sum_{\ell\geq 0}w_{m,k,\ell}=  g_{m,k}.  
\end{eqalign}
Let
\begin{eqalign}
u_{k,\ell}:=&\frac{(1-\rho_{u})}{2}a_{k}\rho_{u}^{\ell}    \tand\tilde{u}_{m,\ell}:=\frac{1}{\rho_{b}}b_{m}\rho_{\tilde{u}}^{\ell}=\delta_{m}\rho_{\tilde{u}}^{\ell},
\end{eqalign}
for $\rho_{u}:=\rho_{*}^{r_{u}},\rho_{\tilde{u}}:=\rho_{*}^{r_{\tilde{u}}}$, with $r_{u},r_{\tilde{u}}>0$, to get
\begin{eqalign}
 \sum_{\ell\geq 0}u_{k,\ell}=\frac{a_{k}}{2}<a_{k} \tand \sum_{\ell\geq 0}\tilde{u}_{m,\ell}=\frac{b_{m}}{\rho_{b}(1-\rho_{\tilde{u}})}>b_{m}.   
\end{eqalign}
In summary,  we upper bound by the following union over $\ell$
\begin{eqalign}\label{eq:mainprobabilityintersectionsnonzeror}
\eqref{eq:Ykdecompseventprob}\leq &\Proba{\bigcapls{j\in S}\bigcup_{m_{j}=j+1,m_{j}\in S}E_{j,m_j}}=\Proba{\bigcapls{j\in S}E_{j}},  
\end{eqalign}
 for
\begin{eqalign}
E_{a,b}:=&\bigcup_{\ell\geq 0}\set{u_{a,\ell}\leq \etauin{a}{I_{b,a,\ell}}}\cup\bigcup_{\tilde{\ell}\geq 0}\set{\tilde{u}_{b,\tilde{\ell}}\geq \etauin{b}{J_{b,a,\tilde{\ell} }}},\\
E_{j}:=&\bigcup_{m_{j}=j+1,m_{j}\in S}E_{j,m_j}.
\end{eqalign}
\proofparagraph{Step 2: Decoupling}
As done at step \cite[eq. (98)]{AJKS}
we pull the first scale $i_{1}$ for which $r_{i_{1}}>0$ using the union bound
\begin{eqalign}
\eqref{eq:mainprobabilityintersectionsnonzeror}\leq&\sum_{m_{i_{1}}=i_{1}+1,m_{i_{1}}\in S }\sum_{\ell_{1}\geq 0}\Proba{\set{u_{i_1,\ell_{1}}\leq \etauin{i_1}{I_{m_{i_{1}},i_{1},\ell_{1}}}}\cap \bigcapls{j\in S\setminus\{i_1\}}E_j}\\
+&\sum_{m_{i_{1}}=i_{1}+1,m_{i_{1}}\in S }\sum_{\tilde{\ell}_{1}\geq 0}\Proba{\set{\tilde{u}_{m_{i_{1}},\tilde{\ell}_{1}}\geq \etauin{m_{i_{1}}}{J_{m_{i_{1}},i_{1},\tilde{\ell}_{1}}}}\cap \bigcapls{j\in S\setminus\{i_1\}}E_j}.
\end{eqalign}
For the first sum we pick the smallest scale $i_{2}\geq i_{1}$ such that 
\begin{eqalign}
 &g_{i_{2},1}+\delta_{i_{2}+1}+\delta_{i_{2}}\leq x_{m_{i_{1}},i_{1},\ell_{1}},
\end{eqalign}
and if none exists then we have $x_{m_{i_{1}},\ell_{1}}\leq g_{s_{M}}+\delta_{s_{M}+1}+\delta_{s_{M}}$. We have the analogous choice for the second sum, namely the smallest scale  $i_{2}\geq m_{i_{1}}$ with 
\begin{eqalign}
 g_{i_{2},1}+\delta_{i_{2}+1}+\delta_{i_{2}}\leq w_{m_{i_{1}},i_{1},\tilde{\ell}_{1}},  
\end{eqalign}
and if none exists then we have $w_{m_{i_{1}},i_{1},\tilde{\ell}_{1}}\leq g_{s_{M}}+\delta_{s_{M}+1}+\delta_{s_{M}}$. By repeating the process, for each $L\in [N]$,  we get a subsequence $S_{L}:=\set{i_{1},...,i_{L}}\subset S$ such that we have the bound
\begin{eqalign}\label{eq:avch3}
\eqref{eq:mainprobabilityintersectionsnonzeror}\leq &\sumls{L=1}^{M}~~\sumls{m_{i}=i+1,m_{i}\in S,\\i\in S_{L}}^{N}\sumls{E\subset S_{L}}~~\sumls{\ell_{j}\geq 0\\ j\in E}~~\sumls{\tilde{\ell}_{j}\geq 0\\ j\in E^{c}}\\
&\quad\prod_{j\in E}\Proba{u_{i_{j},\ell_{j}}\leq \etauin{i_{j}}{I_{m_{i_{j}},i_{j},\ell_{j}}}}\prod_{j\in E^{c}}\Proba{\tilde{u}_{m_{i_{j}},\tilde{\ell}_{j}}\geq \etauin{m_{i_{j}}}{J_{m_{i_{j}},i_{j},\tilde{\ell}_{j}}}},
\end{eqalign}
where we sum over all subsets $E\subset S_{L}$ and the $m_{i_{j}},\ell_{j},\tilde{\ell}_{j}$ and $i_{i_{j+1}}$ satisfy 
\begin{eqalign}
 &g_{i_{j+1},1}+\delta_{i_{j+1}+1}+\delta_{i_{j+1}}\leq x_{m_{i_{j}},i_{j},\ell_{j}}\\
\tor &g_{i_{j+1},1}+\delta_{i_{j+1}+1}+\delta_{i_{j+1}}\leq w_{m_{i_{j}},i_{j},\tilde{\ell}_{j}},  
\end{eqalign}
depending on whether we study tail or small-ball respectively. Next we translate this condition into concrete inequalities for $m,\ell,\tilde{\ell}$.
\proofparagraph{Step 3: Covering condition}
First we study the $x-$terms. Using the \nameref{def:exponentialchoiceparam}, we write
\begin{eqalign}
 g_{i_{2},1}+\delta_{i_{2}+1}+\delta_{i_{2}}\leq x_{m_{i_{1}},i_{1},\ell_{1}}\doncl& \rho_{g}\rho_{*}^{i_{2}}+\rho_{*}^{i_{2}+1}+\rho_{*}^{i_{2}}\leq \rho_{*} \para{\rho_{g}^{m_{i_{1}}-i_{1}}\rho_{*}^{i_{1}}+\rho_{*}^{m_{i_{1}}}}\rho_{*}^{\ell_{1}}\\ \doncl&i_{1}+\ell_{1}+(m_{i_{1}}-i_{1})c_{1,\rho}+c_{2,\rho}+1\leq i_{2}, 
\end{eqalign}
for 
\begin{eqalign}
c_{1,\rho}:=&\frac{\ln\frac{1}{\rho_{g}}}{\ln\frac{1}{\rho_{*}}}=r_{g}\tand c_{2,\rho}:=\para{\ln\frac{1}{\rho_{*}}}^{-1}\para{\ln\para{1+\para{\frac{\rho_{*}}{\rho_{g}}}^{m_{i_{1}}-i_{1}}}^{-1} +\ln\para{\rho_{g}+\rho_{*}+1}^{-1}  }.    
\end{eqalign}
Conversely, if $r\in S$ and $r\leq i_{2}$, since $i_{2}$ is the smallest-choice,  we have the reverse inequality
\begin{eqalign}
 g_{r,1}+\delta_{r+1}+\delta_{r}> x_{m_{i_{1}},i_{1},\ell_{1}}\doncl&i_{1}+\ell_{1}+(m_{i_{1}}-i_{1})r_{g}+c_{2,\rho}\geq r. 
\end{eqalign}
For the $w-$terms we similarly have
\begin{eqalign}
 g_{i_{2},1}+\delta_{i_{2}+1}+\delta_{i_{2}}\leq w_{m_{i_{1}},i_{1},\tilde{\ell}_{1}}\doncl i_{1}+\tilde{\ell}_{1}+(m_{i_{1}}-i_{1})r_{g}+c_{3,\rho}+1\leq i_{2},
\end{eqalign}
for 
\begin{eqalign}
c_{3,\rho}:=\para{\ln\frac{1}{\rho_{*}}}^{-1}\ln\para{\rho_{g}+\rho_{*}+1}^{-1}. 
\end{eqalign}
By taking $\rho_{*}$ small enough, we get 
\begin{equation}
c_{2,\rho},c_{3,\rho}\in (0,\frac{1}{2})    
\end{equation}
and so since $i_{2},r$ are integers, we obtain 
\begin{eqalign}
r\leq i_{1}+\ell_{1}+(m_{i_{1}}-i_{1})r_{g}< i_{2}.    
\end{eqalign}
By repeating this process , we get a subsequence $S_{L}:=\set{i_{1},...,i_{L}}\subset S$ that satisfy covering condition for $S$
\begin{equation}
S\subset [i_{1},C_{1}]\cup [i_{2},C_{2}]\cup...\cup      [i_{L},C_{L}],
\end{equation}
for 
\begin{equation}
 C_{j}:=i_{j}+\ell_{j}+(m_{i_{j}}-i_{j})r_{g}.
\end{equation}
Therefore, due to the covering condition we obtain
\begin{eqalign}\label{eq:sumofellsgreaterthanM}
\sum_{j=1}^{L}(C_{j}-i_{j}+1)\geq M.    
\end{eqalign}
\proofparagraph{Step 4: Estimates} For the product over $E$ in \cref{eq:avch3} we simply use Markov for $p>1$
\begin{eqalign}\label{eq:Markovinequalityboundgap}
\Proba{u_{i_{j},\ell_{j}}\leq \etauin{i_{j}}{I_{m_{i_{j}},i_{j},\ell_{j}}}}\leq &c \frac{\para{\abs{I_{m_{i_{j}},i_{j},\ell_{j}}}/\delta_{i_{j}}}^{\zeta(p)}}{\para{u_{i_{j} ,\ell_{j}}/\delta_{i_{j}}}^{p}}\\
\leq &c 
\para{\frac{1}{\rho_{a}}}^{p}\para{1+\para{\frac{\rho_{*}}{\rho_{g}}}^{m_{i_{j}}-i_{j}}}^{\zeta(p)}\rho_{g}^{(m_{i_{j}}-i_{j})\zeta(p)}\para{\frac{\rho_{*}^{\zeta(p)}}{\rho_{u}^{p}}}^{\ell_{j}}.
\end{eqalign}
Here we need to extract a bound $\rho_{*}^{\frac{1}{1-\alpha}(1+\e_{*})(\ell_{j}+(m_{i_{j}}-i_{j})r_{g}+1)}$. As mentioned in \cref{eq:avch3}, the $\ell_{j}\geq 0$ and $m_{i_{j}}-i_{j}\geq 1$, so in order to extract the term $+1$ we need to use the second inequality $m_{i_{j}}-i_{j}\geq 1$. Assuming $\ell_{j}=0$, we are left with
\begin{eqalign}
\eqref{eq:Markovinequalityboundgap}\leq c2^{\zeta(p)}\rho_{a}^{-p}\rho_{g}^{(m_{i_{j}}-i_{j})\zeta(p)}=c2^{\zeta(p)}\rho_{*}^{r_{g}(m_{i_{j}}-i_{j})\zeta(p)-pr_{a}},    
\end{eqalign}
and so we require
\begin{eqalign}\label{eq:exponentgapconstraintoneplusepsilon}
  r_{g}(m_{i_{j}}-i_{j})\zeta(p)-pr_{a}> \frac{1}{1-c_{gap}} (1+\e_{*})(r_{g}(m_{i_{j}}-i_{j})+1).
\end{eqalign}
Secondly, for $\ell_{j}\geq 1$ we similarly require
\begin{eqalign}\label{eq:exponentgapconstraintoneplusepsilon2}
\zeta(p)-p r_{u}>\frac{1}{1-c_{gap}}(1+\e_{*}).    
\end{eqalign}
We get
\begin{eqalign}
\eqref{eq:Markovinequalityboundgap}\leq   c\rho_{*}^{\frac{1}{1-\alpha}(1+\e_{*})(\ell_{j}+(m_{i_{j}}-i_{j})r_{g}+1)}.  
\end{eqalign}
We discuss those constraints at the compatibility section below.\\
For the product over $E^{c}$ in \cref{eq:avch3} we apply the small ball estimate \cref{cor:smallballestimateexpocho} for the case \textit{$n>k$ and $\ell\geq 0$}, we need to take the parameters as follows:
\begin{eqalign}
&r=\tilde{u}_{m_{i_{j}},\tilde{\ell}_{j}}=\rho_{2}^{-n}\rho_{1}^{-\ell},\tand t=\abs{J_{m_{i_{j}},i_{j},\tilde{\ell}_{j}}}=\rho_{2}^{k+\ell+\e}\rho_{3}^{n-k}\\
&\rho_{1}=\rho_{\tilde{u}}=\rho_{*}^{r_{\tilde{u}}}, \rho_{2}=\rho_{*},\rho_{3}=\rho_{g}=\rho_{*}^{r_{g}}\\
&a_{1}=r_{\tilde{u}},a_{2}=1, a_{3}=r_{g},\\
&\tand n=m_{i_{j}}, k=i_{j},\ell=\tilde{\ell}_{j},\e:=\frac{\ln(1 - \rho_{*})}{\ln \frac{1}{\rho_*}}.
\end{eqalign}
To apply this corollary, we also need
\begin{eqalign}
r_{\tilde{u}}>1>r_{g} \tand r_{\tilde{u}}>(1+\e)(1+\beta).   
\end{eqalign}
We study these conditions later in \cref{eq:constraintscompatibilitygap}. Therefore, we get the bound 
\begin{eqalign}
\Proba{\tilde{u}_{m_{i_{j}},\tilde{\ell}_{j}}\geq \etauin{m_{i_{j}}}{J_{m_{i_{j}},i_{j},\tilde{\ell}_{j}}}}\leq c\branchmat{\rho_{*}^{\ell_{j} c_{2} } & \tcwhen\ell_{j}>c_{g}(m_{i_{j}}-i_{j})\\
\rho_{*}^{\ell_{j} c_{3,1}+(m_{i_{j}}-i_{j}) c_{3,2} } & \tcwhen 0\leq\ell_{j}\leq c_{g}(m_{i_{j}}-i_{j})},    
\end{eqalign}
for $c_{g}:=1-r_{g}$ and
\begin{eqalign}
c_{2}:=&\frac{1}{32}\para{\frac{r_{\tilde{u}}-(1+\e)-\beta (1+\e)}{\sqrt{\beta\para{1+\e}  }   }   }^{2},\\
c_{3,1}:=&(r_{\tilde{u}}-(1+\e))\para{\ln\frac{1}{\rho_{*}}}^{\alpha_{0}-1}\tand c_{3,2}:=(1-r_{g})\para{\ln\frac{1}{\rho_{*}}}^{\alpha_{0}-1},
\end{eqalign}
for $\e_{0}\approx 1$ and $\alpha_{0}>1$. To use the covering condition \cref{eq:sumofellsgreaterthanM} we again need to extract a bound $\rho_{*}^{(1+\e_{*})(\ell_{j}+(m_{i_{j}}-i_{j})r_{g}+1)}$. For the second range $0\leq\ell_{j}\leq c_{g}(m_{i_{j}}-i_{j})$  we simply require
\begin{eqalign}\label{eq:smalldeviexponentgaprequi}
c_{3,1}>\frac{1}{1-c_{gap}}(1+\e_{*})\tand c_{3,2}>\frac{1}{1-c_{gap}}2(1+\e_{*}),    
\end{eqalign}
where the $2$ is needed as above in order to give the term $+1$. For the first range $\ell_{j}>c_{g}(m_{i_{j}}-i_{j})$, we need to use the summation in order to extract the exponent $c_{g}(m_{i_{j}}-i_{j})$. By requiring
\begin{eqalign}\label{eq:smalldeviexponentgaprequi2}
c_{2}>\frac{1}{1-c_{gap}}2(1+\e_{*})\maxp{\frac{r_{g}}{1-r_{g}},1},
\end{eqalign}
we split this exponent to bound
\begin{equation}
\rho_{*}^{\ell_{j} c_{2}  }\leq \rho_{*}^{\ell_{j} \frac{c_{2}}{2}+\ell_{j} \frac{c_{2}}{2}  }\leq \rho_{*}^{ \frac{1}{1-c_{gap}}(1+\e_{*})\para{\ell_{j}+r_{g} (m_{i_{j}}-i_{j})+1} }.    
\end{equation}
\proofparagraph{Step 5: Conclusion} 
Here we can finally use the covering condition \cref{eq:sumofellsgreaterthanM}. We take out an $\e_{*}(1-\lambda)$-part, for small $\lambda\in (0,1)$, in order to extract exponent $M$
\begin{eqalign}\label{eq:avch4}
\eqref{eq:avch3}\leq&\rho_{*}^{\frac{1}{1-c_{gap}}(1+\e_{**}(1-\lambda))M}\sumls{L=1}^{M}c^{L}~~\sumls{m_{i}=i+1,m_{i}\in S\\i\in S_{L}}^{N}\sumls{E\subset S_{L}}~~\sumls{\ell_{j}\geq 0\\ j\in E}~~\sumls{\tilde{\ell}_{j}\geq 0\\ j\in E^{c}}\prod_{j\in E\cup E^{c}}\rho_{*}^{(\e_{*}\lambda)(\ell_{j}+(m_{i_{j}}-i_{j})r_{g}+1)}.
\end{eqalign}
The $\ell$-sums give a factor $\para{\frac{1}{1-\rho_{*}^{\e_{*}\lambda}}}^{L}$. The  $m_{i}$ sums give a factor $\para{\frac{\rho_{*}^{r_{g}\e_{*}\lambda}}{1-\rho_{*}^{r_{g}\e_{*}\lambda}}}^{L}$, and so we are left with
\begin{eqalign}\label{eq:avch5}
\eqref{eq:avch4}\leq&\rho_{*}^{(1+\e_{**}(1-\lambda))N}\sumls{L=1}^{M}c^{L}\rho_{*}^{(r_{g}\e_{*}(1-\lambda))L}\leq \tilde{c} \rho_{*}^{(1+\e_{**}(1-\lambda))M}.
\end{eqalign}
\proofparagraph{Step 6: Compatibility of constraints}
In summary we have the constraints
\begin{eqalign}\label{eq:constraintscompatibilitygap}
1)~~&r_{g}(m_{i_{j}}-i_{j})\zeta(p)-pr_{a}> \frac{1}{1-c_{gap}} (1+\e_{*})(r_{g}(m_{i_{j}}-i_{j})+1),    \\
2)~~&\zeta(p)-p r_{u}>\frac{1}{1-c_{gap}}(1+\e_{*}),\\ 
3)~~&r_{\tilde{u}}>1>r_{g} \tand r_{\tilde{u}}>(1+\e)(1+\beta),   \\
4)~~&c_{3,1}=(r_{\tilde{u}}-(1+\e))\para{\ln\frac{1}{\rho_{*}}}^{\alpha_{0}-1}>\frac{1}{1-c_{gap}}(1+\e_{*}),\\
5)~~&c_{3,2}=(1-r_{g})\para{\ln\frac{1}{\rho_{*}}}^{\alpha_{0}-1}>\frac{1}{1-c_{gap}}2(1+\e_{*}),    \\
6)~~&c_{2}=\frac{1}{32}\para{\frac{r_{\tilde{u}}-(1+\e)-\beta (1+\e)}{\sqrt{\beta\para{1+\e}  }   }   }^{2}>\frac{1}{1-c_{gap}}2(1+\e_{*})\maxp{\frac{r_{g}}{1-r_{g}},1}.
\end{eqalign}
The (4,5) constraints are immediate by taking small enough $\rho_{*}$. The constraints (3,6) also follow by taking large $r_{\tilde{u}}$. So it remains to study the first two constraints. We maximize $\zeta(p)$ by setting $p=\frac{\beta^{-1}+1}{2}$ (allowed since $\beta<1$) to get $\zeta(p)=\frac{(\beta+1)^{2}}{4\beta}$. For the second constraint we simply take small enough $r_{u}$. Since $m_{i_{j}}-i_{j}\geq 1$, the first constraint simplifies to
\begin{eqalign}
r_{g}\para{\frac{(\beta+1)^{2}}{4\beta}-\frac{(1+\e_{*})}{1-c_{gap}}}-\frac{\beta^{-1}+1}{2}r_{a}> \frac{(1+\e_{*})}{1-c_{gap}}. 
\end{eqalign}
To further simplify we take $r_{g}=1-\e$ for small $\e>0$ (at the cost of larger $r_{\tilde{u}}$ in (6))
\begin{eqalign}\label{eq:couplingconstraintforra}
(1-\e)\para{\frac{(\beta+1)^{2}}{4\beta}-\frac{(1+\e_{*})}{1-c_{gap}}}-\frac{\beta^{-1}+1}{2}r_{a}> \frac{(1+\e_{*})}{1-c_{gap}}.
\end{eqalign}

\end{proof}
\newpage \part{Ratios of increments }\label{part:inverse_ratio_moments}
 In this section we study the moments of the ratios of $Q_{\eta}$. There is an analogous result for $\eta$ in \cite[lemma 4.4]{AJKS}, where using the lognormal-scaling law one can bound
\begin{equation}
\Expe{\para{\frac{\eta^{\delta}(I)}{\eta^{\delta}((J)}}^{p}  }\leq c\para{\frac{\abs{I}}{\abs{J}}}^{\zeta(p)},    
\end{equation}
for all small enough intervals $I,J$ ie. $\abs{I},\abs{J}<\delta$ that are close and for $1<p<\beta^{-1}$.
\begin{remark}
    In the conformal welding part for the inverse, the dilatation $K$ is controlled by ratios $\frac{Q^{\delta}(J)}{Q^{\delta}(I)}$. The case $\abs{I}=\abs{J}$ represents the most singular behaviour of the dilatation as it approaches the real axis. In that case we return to the scaling law for GMC.
    \end{remark}
 In the following proposition we will study moments of the analogous ratio of the inverse increments.
As in \cite[lemma 4.4]{AJKS} we have two separate statements.
\begin{proposition}
\label{prop:inverse_ratio_moments}
We go over the cases of decreasing numerator, equal-length intervals for $Q$ and the same for $Q_{H}$.
\begin{itemize}
    \item Fix  $\beta\in  (0,0.152)$ and $\delta\leq 1$ . Assume $\abs{I}=r\delta$ for some $r\in (0,1)$ and $\abs{J}\to 0$. Then there exists $\e_{1},\e_{2}>0$, such that for each $p=1+\e_{1}$, one can find $q=1+\e_{2}\tand \tilde{q}>0$ such that
\begin{eqalign}\label{eq:ratiolpestimates}
  \Expe{\para{\frac{Q^{\delta}\para{J}}{Q^{\delta}\para{I}}}^{p}}\lessapprox \para{\frac{\abs{J}}{\delta}}^{q} r^{-\tilde{q}}.
\end{eqalign}

 \item Fix $\gamma<\frac{2}{\sqrt{3}}$ and $\delta\leq 1$. Assume we have two equal-length intervals $J=(a,a+x),I=(b,b+x)\subset [0,1]$ with $b-a=c_{b-a}x$ for $c_{b-a}>1$  and $a=c_{1}\delta,b=c_{2}\delta$ and $x< \delta$.Then we have for all $p\in [1,1+\e_{1}]$ with small enough $\e_{1}>0$, a bound of the form
\begin{eqalign}\label{eq:ratiobound}
\maxp{\Expe{\para{\frac{Q(a,a+x)}{Q(b,b+x)}}^{p} },\Expe{\para{\frac{Q(b,b+x)}{Q(a,a+x)}}^{p} }}\leq c_{a,b,p,\delta }\para{\frac{x}{\delta}}^{-\e_{ratio}(p)},
\end{eqalign}
where $\e_{ratio}(p)\in (0,1)$ can be made arbitrarily small at the cost of larger comparison constant $c_{a,b,p,\delta }$. The constant $c_{a,b,p,\delta }$ is uniformly bounded in $\delta$. 

\item Fix $\gamma<\frac{2}{\sqrt{3}}$.  For the inverse $Q_{H}$ corresponding to the field $H$ in \Cref{eq:covarianceunitcircle}, we have the same singular estimate as above for $\delta=1$
\begin{eqalign}\label{eq:ratioboundcircle}
\maxp{\Expe{\para{\frac{Q_{H}(a,a+x)}{Q_{H}(b,b+x)}}^{p} },\Expe{\para{\frac{Q_{H}(b,b+x)}{Q_{H}(a,a+x)}}^{p} }}\leq c_{a,b,p,\delta }x^{-\e_{ratio}(p)}.
\end{eqalign}

\end{itemize}
\end{proposition}
\section{Dilatation is in \sectm{$L^{1}_{loc}$}}
 One major corollary of this proposition is that the dilatation of the inverse homeomorphism on the positive real line is in $L^{1}_{loc}$. This is the analogous result to \cite[Lemma 4.5]{AJKS}. In particular, here we study the dilatation $K_{Q}$ corresponding to the inverse of scale one $Q^{1}:\Rplus\to \Rplus$.
\begin{corollary}\label{cor:dilatation}
We have finiteness
\begin{equation}
\int_{[0,1]\times [0,2]}K_{Q}(x+iy)\dx\dy\leq \frac{c}{2(1-\epsilon_{ratio}(1))}<\infty.    
\end{equation}
\end{corollary}
\vspace{-0.5cm}
\begin{proof}[proof of \Cref{cor:dilatation}]
We use the dilatation bound that depends linearly on the homeomorphism \cite[Theorem 2.6]{AJKS}: for each $I\in \mathcal{D}:=\bigcup_{n\geq 0}\mathcal{D}_{n}$ and Whitney square $C_{I}:=\set{(x,y)\in I: x\in I\tand y\in [2^{-n-1},2^{-n}]}$, we have
\begin{equation}
\sup_{z\in C_{I}}K(z)\leq \sum_{\mathbf{J}\in j_{\ell}(I) } \frac{Q(J_{1})}{Q(J_{2})}+ \frac{Q(J_{2})}{Q(J_{1})},
\end{equation}
for \begin{equation}
j_{\ell}(I):=\set{ \mathbf{J}=(J_{1},J_{2}): J_{1},J_{2}\in \mathcal{D}_{n+5}\tand  J_{1},J_{2}\subset j_{0}(I)}.    
\end{equation}
and $j_{0}(I)$ be the union of $I$ and its neighbors in $\mathcal{D}_n$. So by covering the rectangle $[0,1]\times [0,2]$ by Whitney squares we get similarly to \cite[Lemma 4.5]{AJKS} the bound
\begin{equation}
\int_{[0,1]\times [0,2]}K_{Q}(x+iy) \dx\dy \leq c\sum_{n=1}^{\infty}\sum_{k=1}^{2^{n}}\int_{\frac{k-1}{2^{n}}}^{\frac{k}{2^{n}}}\frac{1}{x^{\e_{ratio}(1)}}\frac{1}{2^{n+1}}\dx =\frac{c}{2(1-\epsilon_{ratio}(1))}<\infty.
\end{equation}
\end{proof}
 Since we have the bound for the unit circle case too in \Cref{eq:ratioboundcircle}, we have the above corollary there as well. For the inverse homeomorphism $Q_{H}(x\eta_{H}(1))$ is a bit more tricky because one has to first scale out the total mass. 
\begin{remark}\label{gammaconstraint}
One seemingly major research problem are extending the constraints on $\gamma$. It is unclear how to remove these given the current techniques. One source of them are in \Cref{eq:constraintongamma} which are because in the the maximum modulus estimates we get singular factors. So it seems that one would have  to avoid the modulus techniques all together. We talk about various routes in the further directions of \cite{binder2023inverse}.   
\end{remark}

\section{Moments of the ratio: proof of \sectm{\Cref{prop:inverse_ratio_moments}}}
\begin{remark}\label{rem:tricky}
The proof is tricky for the following three reasons. 
\begin{enumerate}
    
    \item \textbf{lack of prefactor scaling law}  The lognormal scaling law shows up as a prefactor in \Cref{exactscaling} and so there is no cancelation of these lognormals as in the GMC-case \cite[lemma 4.4]{AJKS}. And even then one still needs the constraint $Q(a,a+x)\leq x$ which an event that goes to zero in probability. Even for $a=0$ using the log-normal scaling law we actually have 
\begin{eqalign}\label{eq:lognormalscalinglawgoingtoone}
\Proba{Q_{x}\geq \delta x }=\Proba{ x\geq \eta^{\delta}\para{x\delta}}=\Proba{\sqrt{\beta \ln\frac{1}{x}}+\frac{\ln\frac{1}{\eta^{\delta}(\delta)}}{\sqrt{\beta \ln\frac{1}{x}}}\geq N(0,1)  }\to 1 \tas x\to 0,
\end{eqalign}
for independent Gaussian $N(0,1)$. This means that $Q_{x}$ decays a lot slower than linear.

    \item \textbf{$Q_{a}$ is random} In the GMC-case $\frac{\eta(a,a+x))}{\eta(b,b+x)}$, one can apply \Cref{exactscaling} at $x_{0}=a$. However, here $Q_{a}$ is random. In fact, the distribution of $U(s+Q_{a})$ will likely be assymmetric and non-Gaussian.  In \cite[lemma 2.2]{cui2017first} for the case of Brownian motion $U_{s}=B_{s}$ , they show that $\expo{\sigma B_{T_{a}}-(\frac{\sigma^{2}}{2}-\mu)T_{a}}\eqdis\frac{\sigma^{2}}{4} X_{a}^{2}$ for a Bessel process $X_{a}$ at location $a$ and index $\nu=\frac{2\mu}{\sigma^{2}}-1$. So in this work we decompose $Q_{a}$ and then extract a factor $x^{-p}$. If there was a way to deal with this issue, then perhaps it would also remove that singular factor.

\item \textbf{$Q_{b-a}\bullet Q_{a}$ is random} In the GMC-case, one can restrict $b-a\leq c x$ so that the scaling law will go through. However, the analogous object to $b-a$ here is $Q_{b-a}\bullet Q_{a}$ and that can be arbitrarily large. 
We turn the ratio into
\begin{eqalign}
&\Expe{\para{\frac{Q(a,a+x)}{Q(b,b+x)}}^{p}}\leq \para{\Expe{\para{\frac{Q(a,a+x)}{Q_{b-a}\bullet Q_{a}}}^{p_{1}}}}^{1/p_{1}} \para{\Expe{\para{\frac{Q_{b-a}\bullet Q_{a}}{Q(b,b+x)}}^{p_{2}}} }^{1/p_{2}} ,  
\end{eqalign}
in order to match scales and \textit{heuristically} apply the scaling law wrt $\tilde{x}\approx Q_{b-a}\bullet Q_{a}$. This fits the perspective of truly working in the \textit{image} and thus needing to do all the operations as if $Q$ is the domain i.e. keep reverting to shifted-GMC.
    
\end{enumerate}
\end{remark}



\begin{remark}
To understand the overall strategy, we start with a heuristic argument. We first approximate over the dyadics: $Q^{\delta,n}(a)=\ell\in D_{n}(0,\infty)$. Using the semigroup formula
\begin{eqalign}
Q^{\delta}(x)-Q^{\delta}(y)=&\infp{t\geq 0: \etamu{Q^{\delta}(y),Q^{\delta}(y)+t}{\delta}\geq x-y }\\
=&:Q_{x-y}^{\delta }\bullet Q_{y}^{\delta},
\end{eqalign}
and the approximation $Q^{\delta,n}(a)$ for $ Q^{\delta}(a)$ we replace $Q^{\delta}(J)=Q^{\delta}(a,a+x)=Q^{\delta}_{x}\bullet Q^{\delta}_{a}$ by
\begin{eqalign}
& Q_{x}\bullet Q^{\delta,n}(a)=  Q_{x}\bullet\ell=\infp{t\geq 0: \etamu{ \ell,\ell+t}{\delta}\geq x }
\end{eqalign}
and we replace $Q^{\delta}(I)=Q^{\delta}(b,b+x)$ by
\begin{eqalign}
&Q_{x}\bullet \para{ Q_{b-a}\bullet \ell     +\ell}=\infp{t\geq 0: \etamu{Q_{b-a}\bullet \ell     +\ell,Q_{b-a}\bullet \ell     +\ell+t}{\delta}\geq x }. 
\end{eqalign}
So by applying \Holder we get
\begin{eqalign}
&\sum_{\ell} \Proba{\frac{Q^{\delta}(J)}{Q^{\delta}(I)}>R,Q^{\delta,n}(a)=\ell ,Q_{b}-Q_{a}\leq \rho\delta }\\
&\leq \sum_{\ell}\para{\Proba{\frac{Q_{x}\bullet\ell}{Q_{x}\bullet \para{Q_{b-a}\bullet\ell+\ell} }>R,Q_{b-a}\bullet\ell \leq \rho\delta  }}^{\frac{1}{b}_{1}} \para{\Proba{Q^{\delta,n}(a)=\ell}}^{\frac{1}{b}_{2}},
\end{eqalign}
for $b_{i}$ \Holder conjugates. The first factor is perfectly set up for the log-scaling law for $\lambda=x$.  Now the sum is proportional to $2^{n}$ and thus diverges when we take the limit to approximate the original $Q^{\delta}(a)$. So instead we will decompose as in the proof of \Cref{cor:shiftedGMCmoments}.    
\end{remark}

\subsection{Proof of \sectm{$\abs{I}=r\delta$} for \sectm{$r>0$} and \sectm{$\abs{J}\to 0$}}
Here we simply scale by $\delta$ and apply \Holder:
\begin{eqalign}
    &\Expe{\para{Q^{1} \para{\frac{J}{\delta}}}^{p}\para{Q^{1} \para{\frac{I}{\delta}}}^{-p}}\\
    \leq& \para{\Expe{\para{Q^{1} \para{\frac{J}{\delta}}}^{pp_{1}}}}^{1/p_{1}}\para{\Expe{\para{Q^{1} \para{\frac{I}{\delta}}}^{-pp_{2}}}}^{1/p_{2}}
\end{eqalign}
for $\frac{1}{p_{1}}+\frac{1}{p_{2}}=1$ and then just apply the inverse increment moments \Cref{prop:momentsofshiftedinverse}
\begin{equation}
\leq \para{c_{11}\para{\frac{J}{\delta}}^{q_{11}}+c_{12}\para{\frac{J}{\delta}}^{q_{12}}}^{1/p_{1}}\para{c_{21}r^{-q_{21}}+c_{22}r^{-q_{22}}}^{1/p_{2}}.    
\end{equation}
Next we further explore the constraints for getting $\frac{q_{11}}{p_{1}},\frac{q_{12}}{p_{1}}\geq 1+\e_{ratio}$ for some $\e_{ratio}>0$. The constraint from \Cref{prop:momentsofshiftedinverse} is
\begin{equation}
 \zeta(-q_{11})+pp_{1}>1  \tand q_{12}>pp_{1},   
\end{equation}
where we simplified and relabelled to match the current setting. The second inequality is cleared by taking large $q_{12}>0$. We set $\frac{q_{11}}{p_{1}}=1+\e_{ratio}$ and obtain a constrain for $\e_{ratio}$ in terms of $\beta$. This choice translates second inequality to
\begin{eqalign}\label{eq:eratioconstraint}
&pp_{1}>\zeta(-q_{11})-1\\
\doncl &p>(1+\e_{ratio})\para{1+\beta\para{p_{1}(1+\e_{ratio})+1}}+\frac{1}{p_{1}}.
\end{eqalign}
Since $pp_{2}\in [1,\beta^{-1})$, we further take 
\begin{eqalign}
p=\beta^{-1}\frac{p_{1}-1}{p_{1}}-o(\e)>1    
\end{eqalign}
and so we study
\begin{eqalign}\label{eq:eratioconstraint2}
\eqref{eq:eratioconstraint}\doncl &\beta^{-1}\frac{p_{1}-1}{p_{1}}>(1+\e_{ratio})\para{1+\beta\spara{p_{1}(1+\e_{ratio})+1}}+\frac{1}{p_{1}}.
\end{eqalign}
Since the inverse has all its positive moments we are free to take arbitrarily large $p_{1}>1$. This has a solution $p_{1}>1, \e_{ratio}\in (0,1)$ when we further require $\beta\in (0,0.152)$.


\subsection{Proof of \sectm{$\abs{J}=\abs{I}<\delta$ and ratio of the form $\frac{Q(a,a+x)}{Q(b,b+x)}$ for $a<b$}}
Here we will use a similar argument as in the computation of moment increments in \cite{binder2023inverse}. 
 Now we proceed to the proof. For simplicity we assume $\delta=1$ by scaling
\begin{equation}
 \Expe{\para{\frac{Q^{\delta}(a,a+x)}{Q^{\delta}(b,b+x)}}^{p}} = \Expe{\para{\frac{Q(\frac{a}{\delta},\frac{a+x}{\delta})}{Q(\frac{b}{\delta},\frac{b+x}{\delta})}}^{p}}  
\end{equation}
and instead just abuse notation to write $\Expe{\para{\frac{Q(a,a+x)}{Q(b,b+x)}}^{p}}$ whereas in the proposition we write the fractions. We fix $\rho_{b-a}\in (0,1)$ that will be taken to be small enough below given the various constraints at the cost of large comparison constants. We will go over all the possibilities of $Q\para{b,b+x},Q\para{a,a+x}\tand Q_{b-a}\bullet Q_{a}$ depending on them being less or greater than one or $\rho_{b-a}$ respectively.
\subsubsection{Case \sectm{$Q(b,b+x)\in [1,\infty)$}}
This is the easiest case, since we can simply lower bound $Q(b, b+x)$ by $1$
\begin{equation}
\Expe{\para{\frac{Q(a,a+x)}{Q(b,b+x)}}^{p}\ind{Q\para{b,b+x}\geq 1}}\leq \Expe{\para{Q\para{a,a+x}}^{p}}.     
\end{equation}
So below we assume $Q(b,b+x)\in [0,1]$.
\subsubsection{Case \sectm{$Q(b,b+x)\in [0,1]$ and $Q_{b-a}\bullet Q_{a}\geq \rho_{b-a}$}}
Here we use the upper bound $x\leq b-a$ and just study
\begin{eqalign}\label{eq:caseoflargenumerator}
&\Expe{\para{\frac{Q(a,a+x)}{Q(b,b+x)}}^{p}\ind{E_{b,b+x}^{[0,1]}}}
\leq \Expe{\para{\frac{Q_{b-a}\bullet Q_{a}}{Q(b,b+x)}}^{p}\ind{E_{b,b+x}^{[0,1]}}}. 
\end{eqalign}
\proofparagraph{Step 1: Bounding the ratio $\para{\frac{Q_{b-a}\bullet Q_{a}}{Q(b,b+x)}}^{p}$ }
Here we will decompose the numerator event $E_{b-a}^{[\rho_{b-a},\infty)}$ with $E_{b-a}^{k}:=\set{Q_{b-a}\bullet Q_{a}\in \spara{b_{k},b_{k+1}}=:I_{b-a}^{k}}$ for $k\geq 1$ and the sequence  $b_{k}:=\rho_{b-a}k^{\lambda_{2}}$ for some $\lambda_{2}\in (0,1)$ and $k\geq 1$
\begin{eqalign}\label{eq:ratioQb-aoverQbbx}
\Expe{\para{\frac{Q_{b-a}\bullet Q_{a}}{Q(b,b+x)}}^{p}\ind{E_{b-a}^{[\rho_{b-a},\infty)},E_{b,b+x}^{[0,1]} }}&=\sum_{k}\Expe{\para{\frac{Q_{b-a}\bullet Q_{a}}{1\wedge Q(b,b+x)}}^{p}\ind{E_{b-a}^{k}}}.
\end{eqalign}
Then we use the upper bound and apply layercake 
\begin{eqalign}\label{eq:mainprobabilityeventsums}
&\sum_{k}b_{k+1}^{p}\Expe{\para{\frac{1}{1\wedge Q(b,b+x)}}^{p}\ind{E_{b-a}^{k},E_{b,b+x}^{[0,1]}}}\\
=&\sum_{k}b_{k+1}^{p}\int_{1}^{\infty}\Proba{\eta(Q_{b},Q_{b}+t^{-1/p})\geq x ,E_{b-a}^{k}}\dt.
\end{eqalign}
Here we need to further decompose $Q_{a}$
\begin{eqalign}
D_{\ell}^{a}:=&\set{  Q(a)\in [a_{\ell},a_{\ell+1}]=:I_{\ell}},
\end{eqalign}
 for the $a_{\ell}:=\rho_{1}\ell^{\lambda_{1}},\ell\geq 0$ and $\rho_{1},\lambda_{1}\in (0,1)$. Using the semigroup formula $Q(b)= Q_{b-a}\bullet Q(a)    +Q(a) $, we further upper bound by
\begin{eqalign}\label{eq:mainprobabilityeventsumsonlylower}
\eqref{eq:mainprobabilityeventsums}\leq &\sum_{k}b_{k+1}^{p}\int_{1}^{\infty}\sum_{\ell} \\
&\Proba{\max_{u\in I_{\ell},\tilde{u}\in I_{b-a}^{k}}\eta(\tilde{u}+u,\tilde{u}+u+t^{-1/p})\geq x, c_{b-a}x\geq \min_{u\in I_{\ell}}\eta(u,u+b_{k}) ,D_{\ell}^{a}},      
\end{eqalign}
where we only kept the upper bound in $E_{b-a}^{k}$. Now that we used the decomposition events $D_{\ell}^{a}$ that we will study later, we bound by the minimum
\begin{equation}
\minp{\Proba{\max_{u\in I_{\ell},\tilde{u}\in I_{b-a}^{k}}\eta(\tilde{u}+u,\tilde{u}+u+t^{-1/p})\geq x, c_{b-a}x\geq \min_{u\in I_{\ell}}\eta(u,u+b_{k})},\Proba{D_{\ell}^{a}}}.    
\end{equation}
For the first two events, we further apply the FKG inequality \Cref{FKGineq}
\begin{eqalign}
&\Proba{\max_{u\in I_{\ell},\tilde{u}\in I_{b-a}^{k}}\eta(\tilde{u}+u,\tilde{u}+u+t^{-1/p})\geq x, c_{b-a}x\geq \min_{u\in I_{\ell}}\eta(u,u+b_{k})}\\
\leq &\Proba{\max_{u\in I_{\ell},\tilde{u}\in I_{b-a}^{k}}\eta(\tilde{u}+u,\tilde{u}+u+t^{-1/p})\geq x}\Proba{c_{b-a}x\geq \min_{u\in I_{\ell}}\eta(u,u+b_{k})}.
\end{eqalign}
For the first event depending on $t$, we simply use the positive moments in \Cref{cor:shiftedGMCmoments} for $t^{-1}\leq 1$ and $p_{1}\in [1,\beta^{-1})$
\begin{equation}\label{eq:reallinelinemaxbound}
\Proba{\max_{u\in I_{\ell},\tilde{u}\in I_{b-a}^{k}}\eta(\tilde{u}+u,\tilde{u}+u+t^{-1/p})\geq x}\lessapprox x^{-p_{1}} t^{-\frac{\zeta(p_{1})-1}{p}} ,
\end{equation}
where the constant is uniformly bounded since $\abs{I_{\ell}}+\abs{I_{b-a}^{k}}\leq 10$ \setword{(*)}{conditionone}, which will be shown below. So for the t-integral to converge we require $\zeta(p_{1})>p+1$. In summary returning to \Cref{eq:mainprobabilityeventsums} we will study
\begin{eqalign}\label{eq:maineventoboundratioone}
\eqref{eq:mainprobabilityeventsums}\leq &\int_{[1,\infty)}\sum_{k}\sum_{\ell}b_{k+1}^{p}\minp{\para{x^{-p_{1}} t^{-\frac{\zeta(p_{1})-1}{p}}}\Proba{c_{b-a}x\geq \min_{u\in I_{\ell}}\eta(u,u+b_{k})}, \Proba{D_{\ell}^{a}}},
\end{eqalign}
 where we also applied translation invariance. We let $\alpha_{1}:=\frac{\zeta(p_{1})-1}{p}$ and 
\begin{eqalign}
R_{1,k}:=&\Proba{c_{b-a}x\geq \min_{u\in I_{\ell}}\eta(u,u+b_{k})} \tand R_{2,\ell}:=\Proba{D_{\ell}^{a}}.
\end{eqalign}
\proofparagraph{Step 2: The t-integral}
When $t\geq 1$, we have
\begin{eqalign}
   \para{x^{-p_{1}} t^{-\alpha_{1}}}R_{1,k} \leq R_{2,\ell}\doncl t\geq\para{x^{-p_{1}} R_{1,k}\para{R_{2,\ell}}^{-1}}^{1/\alpha_{1}} =:N_{1}.
\end{eqalign}
Therefore, all together, by integrating in $t$ we have 
\begin{eqalign}\label{eq:maineventoboundratiotwo}
\eqref{eq:maineventoboundratioone}\leq &\sum_{k}b_{k+1}^{p}\para{\sum_{\ell}N_{1}  R_{2,\ell}+N_{1}^{1-\alpha_{1}} x^{p_{1}} R_{1,k}}\\
=&\sum_{k}b_{k+1}^{p}\Big(\sum_{\ell}\para{x^{-p_{1}} R_{1,k}\para{R_{2,\ell}}^{-1}}^{1/\alpha_{1}}   R_{2,\ell}\\
&+\para{x^{-p_{1}} R_{1,k}\para{R_{2,\ell}}^{-1}}^{(1-\alpha_{1})/\alpha_{1}}  x^{-p_{1}}R_{1,k}\Big)\\
\leq&2\sum_{k}b_{k+1}^{p}\para{x^{-p_{1}} R_{1,k}}^{1/\alpha_{1}} \sum_{\ell}\para{R_{2,\ell}}^{1-1/\alpha_{1}}.
\end{eqalign}
\proofparagraph{Step 3: The $\ell$-sum for $R_{2,\ell}$}
For the event $D_{\ell}^{a}:=\set{  Q(a)\in [a_{\ell},a_{\ell+1}]=:I_{\ell}}$, we choose $a_{\ell}:=\rho_{1}\ell^{\lambda_{1}},\ell\geq 0$ and $\rho_{1},\lambda_{1}\in (0,1)$. For $\ell=0$, we bound by 1 and for $\ell\geq 1$ we apply Markov for $r_{1}>0$
\begin{equation}
R^{2}_{\ell}= \Proba{D_{\ell}^{a}}\leq  \Proba{a\geq \eta(a_{\ell})}\lessapprox (a/\rho_{1})^{r_{1}}\ell^{-\lambda_{1}r_{1} }=:c\ell^{-\lambda_{1}r_{1} },
\end{equation}
where for finiteness we require $\lambda_{1}r_{1}>1$. To get summability in the  $\ell$-sum \Cref{eq:maineventoboundratiotwo} we require
\begin{equation}
\lambda_{1}r_{1}\para{1-\frac{1}{\alpha_{1}}}  >1.
\end{equation}
Here we can take large $r_{1}$ to get this as long as again $\alpha_{1}>1\doncl \zeta(p_{1})>p+1$. Since $p=1+\e_{1}$, we require
\begin{eqalign}
\zeta(p_{1})>2,    
\end{eqalign}
which forces $\beta<\frac{2}{3}$.
\proofparagraph{Step 4: The $k$-sum and $x$}
We choose $b_{k}:=\rho_{b-a}k^{\lambda_{2}}$ for some $\lambda_{2}\in (0,1)$ and $k\geq 1$ in order to cover the $[\rho_{b-a},\infty)$ image. Together with above constraint on $\lambda_{1}$, we also get the constraint \ref{conditionone}. We apply the negative moments \Cref{cor:shiftedGMCmoments} for some $p_{2}>0$: when $b_{k}\leq 1$ we bound by
\begin{equation}\label{eq:sumkevent1}
R_{1,k}=\Proba{c_{b-a}x\geq \min_{u\in I_{\ell}}\eta(u,u+b_{k})}\lessapprox x^{p_2} b_{k}^{\zeta(-p_{2})-1}=x^{p_2}\rho_{b-a}^{\zeta(-p_{2})-1}k^{(\zeta(-p_{2})-1)\lambda_{2}}.   
\end{equation}
when $b_{k}\geq 1$ we bound by
\begin{equation}\label{eq:sumkevent2}
R_{1,k}=\Proba{c_{b-a}x\geq \min_{u\in I_{\ell}}\eta(u,u+b_{k})}\lessapprox x^{p_2} b_{k}^{-p_{2}-1}=x^{p_2}\rho_{b-a}^{-p_{2}-1}k^{-(p_{2}-1)\lambda_{2}}. 
\end{equation}
So to get finiteness in the $k$-sum we require 
\begin{eqalign}\label{eq:lambdaone}
&\lambda_{2}\para{\frac{p_{2}-1}{\alpha_{1}}-p}=\lambda_{2}p\para{\frac{p_{2}-1}{\zeta(p_{1})}-1} >1, 
\end{eqalign}
which is possible for large enough $p_{2}$.
For $x$ we have the exponent 
\begin{eqalign}\label{eq:exponenntofx}
p_{2}-\frac{p_{1}}{\alpha_{1}}=p_{2}-p\frac{p_{1}}{\zeta(p_{1})},
\end{eqalign}
where again we take large enough $p_{2}$ to get a positive exponent.
\begin{remark}
This case was simple because we truncated away from zero for $Q_{b-a}\bullet Q_{a}\in [\rho_{b-a},\infty)$ which then came at the price of a very large constant in $\rho_{b-a}^{-p_{2}}$ that doesn't affect overall finiteness.    
\end{remark}

\newpage
\subsection{Case \sectm{$Q_{b-a}\bullet Q_{a}\in [0,\rho_{b-a}]$ and $Q(b,b+x)\in [0,1]$}}
Here we finally study the singular case of both $Q_{b-a}\bullet Q_{a}$ and $Q(b,b+x)$ going to zero and thus we need to use the lognormal scaling law. As mentioned in \Cref{rem:tricky}, a key part is the difference $b-a$ which is required to decay proportionally to $x$ i.e. $b-a\leq c_{b-a}x$ for the scaling law to go through. Here we deal with a random variable $Q_{b-a}\bullet Q_{a}$ and as mentioned in \cite{binder2023inverse} even for $a=0$ using the log-normal scaling law we actually have 
\begin{eqalign}
\Proba{Q_{x}\geq \delta x }=\Proba{ x\geq \eta^{\delta}\para{x\delta}}=\Proba{\sqrt{\beta \ln\frac{1}{x}}+\frac{\ln\frac{1}{\eta^{\delta}(\delta)}}{\sqrt{\beta \ln\frac{1}{x}}}\geq N(0,1)  }\to 1 \tas x\to 0,
\end{eqalign}
for independent Gaussian $N(0,1)$. This means that $Q_{x}$ decays a lot slower than linear. So here the scaling law needs to be done with respect to a different variable, not $x$. Heuristically by turning to the shifted GMC, we will do the scaling law with respect to $\tilde{x}\approx Q_{b-a}\bullet Q_{a}$.\\\\
\proofparagraph{Step 1: Decomposing $Q_{b-a}\bullet Q_{a}$ and $Q_{a}$} We decompose the numerator 
\begin{equation}
E_{b-a}^{k}(Q_{a}):=\set{Q_{b-a}\bullet Q_{a}\in I_{k}:=\spara{a_{k+1},a_{k}}\subset [0,\rho_{b-a}]},    
\end{equation}
for the fixed small constant $\rho_{b-a}\in (0,1)$ and the sequence $a_{k}:=\rho_{b-a}k^{-\lambda_{b-a}}$ for $k\geq 1$ and $\lambda_{b-a}>0$, to obtain as above
\begin{eqalign}\label{eq:maineventosumsingularcase}
&\Expe{\para{\frac{Q_{b-a}\bullet Q_{a}}{Q(b,b+x)}}^{p}\ind{Q_{b-a}\bullet Q_{a}\in [0,\rho_{b-a}]}}\\
\leq&1+\sum_{k}\int_{1}^{\infty}\Proba{\eta\para{Q_{b},Q_{b}+ a_{k}t^{-1/p} }\geq x,  \eta\para{Q_{a},Q_{a}+ a_{k+1}}\leq b-a\leq \eta\para{Q_{a},Q_{a}+ a_{k}}  }\dt,
\end{eqalign}
where for convenience we removed the small ratio part i.e. $\para{\frac{Q_{b-a}\bullet Q_{a}}{Q(b,b+x)}}^{p}\leq 1$. So we again write $b-a=c_{b-a}x$ for $c_{b-a}>1$. In preparation of scaling law, we will use the lower bound in $E_{b-a}^{k}(Q_{a})$ to form a ratio and use the semigroup formula $Q(b)= Q_{b-a}\bullet Q(a)    +Q(a) $ to bound
\begin{eqalign}
&\Proba{\frac{\eta\para{Q_{b},Q_{b}+ a_{k}t^{-1/p} }}{\eta\para{Q_{a},Q_{a}+ a_{k+1}}}\geq c_{b-a}^{-1},E_{b-a}^{k}(Q_{a}) }\\
\leq &\Proba{\max_{u\in I_{k}}\frac{\eta\para{Q_{a}+u,Q_{a}+u+ a_{k}t^{-1/p} }}{\eta\para{Q_{a},Q_{a}+ a_{k+1}}}\geq c_{b-a}^{-1},E_{b-a}^{k}(Q_{a}) }. 
\end{eqalign} 
For each $k\geq 1$, we also decompose the variable $Q(a)$ with events $D_{\ell,k}^{a}:=\set{  Q(a)\in [c_{\ell,k},c_{\ell+1,k}]=:I_{\ell,k}}$ to overall bound
\begin{eqalign}\label{eq:mainsingulareventboundwithallsums}
\eqref{eq:maineventosumsingularcase}\leq &\sum_{k}\int_{1}^{\infty}\sum_{\ell}\\
&\Expe{\maxls{\set{u\in I_{k},\tilde{u}\in I_{\ell,k}}}\ind{\para{\frac{\eta\para{\tilde{u}+u,\tilde{u}+u+ a_{k}t^{-1/p}}}{\eta\para{\tilde{u},\tilde{u}+ a_{k+1}}}}\geq c_{b-a}^{-1} ,E_{b-a}^{k}(\tilde{u})}\ind{D_{\ell,k}^{a}} }\dt.    
\end{eqalign}
Now that we used the event $D_{\ell,k}^{a}$, we bound by a \Holder inequality
\begin{eqalign}\label{eq:maineventrealline}
&\sum_{\ell}\sum_{k}\para{\Proba{D_{\ell,k}^{a}} }^{1/b_{12}}  \\
\cdot&\int_{1}^{\infty}\para{\Expe{\maxls{\set{u\in I_{k},\tilde{u}\in I_{\ell,k}}}\ind{\para{\frac{\eta\para{\tilde{u}+u,\tilde{u}+u+ a_{k}t^{-1/p}}}{\eta\para{\tilde{u},\tilde{u}+ a_{k+1}}}}\geq c_{b-a}^{-1} ,E_{b-a}^{k}(\tilde{u})} }}^{1/b_{11} }\dt  
\end{eqalign}
for $b_{11}^{-1}+b_{12}^{-1}=1$. So now we study
\begin{eqalign}
\Expe{\maxls{u\in I_{k},\tilde{u}\in[0, \abs{I_{\ell,k}}]}\ind{\para{\frac{\eta\para{\tilde{u}+u,\tilde{u}+u+ a_{k}t^{-1/p}}}{\eta\para{\tilde{u},\tilde{u}+ a_{k+1}}}}\geq c_{b-a}^{-1},E_{b-a}^{k}(\tilde{u})} },    
\end{eqalign}
where we also used translation invariance.
\proofparagraph{Step 2: Scaling law step} Here we are finally in the position to apply the lognormal-scaling law \cref{eq:truncatedscalinglawGMC} for $\lambda=\frac{a_{k+1}}{\rho_{b-a}}\in (0,1)$ by requiring the condition
\begin{eqalign}\label{eq:scalinglawconstraintratio}
\frac{\rho_{b-a}}{a_{k+1}}\para{\abs{ I_{\ell,k}}+a_{k}+a_{k} t^{-1/p}}\leq \frac{1}{2}.
\end{eqalign}
We let $L_{1}:=\frac{\rho_{b-a}\abs{I_{\ell,k}}}{a_{k+1}}, L_{2}:=\frac{\rho_{b-a} a_{k}}{a_{k+1}}$ and we study this condition in \Cref{eq:scalinglawconstraintratioagain}. Therefore, by the scaling law we get
\begin{eqalign}
\Expe{\maxls{S_{k,\ell}(u,\tilde{u}) }\ind{\para{\frac{\eta\para{\tilde{u}+u,\tilde{u}+u+ L_{2}t^{-1/p}}}{\eta\para{\tilde{u},\tilde{u}+ \rho_{b-a}}}}\geq c_{b-a}^{-1},E_{b-a}^{k}(\tilde{u},N(0,1))} },    
\end{eqalign}
where $S_{k,\ell}(u,\tilde{u}):=\set{u\in \frac{\rho_{b-a}}{a_{k+1}}I_{k},\tilde{u}\in\spara{0, \frac{\rho_{b-a}}{a_{k+1}}\abs{I_{\ell,k}}}}$ and 
\begin{eqalign}
E_{b-a}^{k}(\tilde{u},N(0,1))&:=\set{\eta\para{\tilde{u},\tilde{u}+ \rho_{b-a}}\leq \frac{\rho_{b-a}(b-a)}{a_{k+1}}e^{-\bar{Z}_{\lambda}}\leq \eta\para{\tilde{u},\tilde{u}+ L_{2}}}\\
&=\set{\ln\frac{1}{\eta\para{\tilde{u},\tilde{u}+ L_{2}}}\leq \sigma_{k}N(0,1)-\frac{\sigma_{k}^{2}}{2} +\ln{ \frac{a_{k+1}}{\rho_{b-a}(b-a)}}\leq  \ln\frac{1}{\eta\para{\tilde{u},\tilde{u}+ \rho_{b-a}}}},        
\end{eqalign}
and the independent $\bar{Z}_{\lambda}:=\sigma_{k}N(0,1)-\frac{\sigma_{k}^{2}}{2}$ (defined in \Cref{logfieldshifted}) and $\sigma_{k}^{2}:=2\beta\para{\ln\frac{1}{\lambda}-1+\lambda}$. Here we further integrate out this independent lognormal
\begin{eqalign}\
\int_{\R}\Expe{\maxls{S_{k,\ell}(u,\tilde{u}) }\ind{\para{\frac{\eta\para{\tilde{u}+u,\tilde{u}+u+ L_{2}t^{-1/p}}}{\eta\para{\tilde{u},\tilde{u}+ \rho_{b-a}}}}\geq c_{b-a}^{-1},E_{b-a}^{k}(\tilde{u},y)} }g_{N(0,1)}(y)\dy,   
\end{eqalign}
where $g_{N(0,1)}(y):=\frac{1}{\sqrt{2\pi}}\expo{-y^{2}/2}$. Now that we used the event $E_{b-a}^{k}$, we again bound by another \Holder for $b_{21}^{-1}+b_{22}^{-1}=1$. So in summary we are studying
\begin{eqalign}\label{eq:summarizedbound}
&\sum_{\ell}\sum_{k}\para{\Proba{D_{\ell,k}^{a}} }^{1/b_{12}}   \\
&\cdot \int_{1}^{\infty}\dt\bigg\{\int_{\R}\para{\Proba{\bigcup_{S_{k,\ell}(u,\tilde{u}) }\set{\para{\frac{\eta\para{\tilde{u}+u,\tilde{u}+u+ L_{2}t^{-1/p}}}{\eta\para{\tilde{u},\tilde{u}+ \rho_{b-a}}}}\geq c_{b-a}^{-1}}}}^{1/b_{21}}\\ &\cdot \para{\Proba{\bigcup_{S_{k,\ell}(u,\tilde{u}) }E_{b-a}^{k}(\tilde{u},y)}}^{1/b_{22}} g_{N(0,1)}(y)\dy \bigg\}^{1/b_{11}}.    
\end{eqalign}

\proofparagraph{Step 3: Ratio modulus}
Now, we apply Markov inequality for $q_{0}\in [1,\beta^{-1})$ and \Holder inequality for $q_{1}^{-1}+q_{2}^{-1}=1$ with $q_{0}q_{1}\in \para{0,\frac{2}{\gamma^{2}}}$ to separate numerator and denominator. We then use \Cref{prop:maxmoduluseta} to get
\begin{eqalign}\label{eq:scalingnumer}
\para{\Expe{\max_{S_{k,\ell}(u,\tilde{u})}\para{\eta\para{u+\tilde{u},u+\tilde{u}+\frac{L_{2}}{t^{1/p} } }}^{q_{0}q_{1}}}}^{1/q_{1}}\lessapprox& \para{\frac{L_{2}}{t^{1/p} }}^{\frac{\zeta(q_{0}q_{1})-1}{q_{1}}}\approx t^{-\frac{1}{q_{1}p}\para{\zeta(q_{0}q_{1})-1 }},
\end{eqalign}
and the \Cref{prop:minmodeta} to get
\begin{eqalign}\label{eq:scalingdenomer}
\para{\Expe{\max_{\tilde{u}\in\spara{0, L_{1}}}\para{\eta\para{\tilde{u},\tilde{u}+\rho_{b-a} }}^{-q_{0}q_{2}}}}^{1/q_{2}}\lessapprox~\rho_{b-a}^{(\zeta(-q_{0}q_{2})-1)/q_{2}},
\end{eqalign}
which is uniformly bounded since we fix $\rho_{b-a}$.

\proofparagraph{Step 4: Integral over $t$} We return to the t-integral in  \Cref{eq:summarizedbound}. Here we have
\begin{eqalign}\label{eq:boundexponentfort}
&\int_{1}^{\infty}\bigg\{\para{\Proba{\bigcup_{S_{k,\ell}(u,\tilde{u}) }\set{\para{\frac{\eta\para{\tilde{u}+u,\tilde{u}+u+ L_{2}t^{-1/p}}}{\eta\para{\tilde{u},\tilde{u}+ \rho_{b-a}}}}\geq c_{b-a}^{-1}}}}^{1/b_{21}} \bigg\}^{1/b_{11}}\dt \\
\leq &c_{\rho_{b-a}} \int_{1}^{\infty}t^{-\frac{1}{b_{11}b_{21}q_{1}p}\para{\zeta(q_{0}q_{1})-1 }}\dt.
\end{eqalign}
 So in order for the t-integral to be finite we require
\begin{equation}\label{eq:tintegralcosnstraint}
\frac{\zeta(q_{0}q_{1})-1}{b_{11}b_{21}}>  pq_{1}.
\end{equation}
We take $p=1+\e$. We also take $q_{0}=\beta^{-1}-\e$ by taking $q_{1}=1+\e_{4}$ and thus $q_{1}$'s \Holder conjugate $q_{2}$ to be arbitrarily large. Finally, we take $b_{11}=1+\e_{5}$  (see \Cref{eq:akfromtheQa} where we take $b_{12}$ arbitrarily large) and $b_{21}=1+\e_{5}$ (we take $b_{22}$ arbitrarily large). So all together we get the constraint
\begin{equation}\label{eq:constraintongamma}
\beta^{-1}>2+o(\e_{i})\doncl \gamma<\sqrt{\frac{2}{2+o(\e_{i})}}  \doncl \gamma<1.
\end{equation}

\proofparagraph{Step 5: Event $D_{\ell,k}^{a}$}
Here we study $D_{\ell,k}^{a}:=\set{  Q(a)\in [c_{\ell,k},c_{\ell+1,k}]=:I_{\ell,k}}$.  For $\ell=0$, we bound by $\Proba{D_{\ell,k}^{a}}\leq 1$ and for $\ell\geq 1$ we bound by a minimum as we will see below.  We set $c_{\ell,k}:=\rho_{a}a_{k}c_{\ell}$ for some small enough $\rho_{a}\in (0,1)$ and $c_{\ell}:=\ell^{\lambda_{a}}$ for $\lambda_{a}\in (0,1)$ so that indeed $c_{\ell+1}-c_{\ell}\to 0$ as $\ell\to +\infty$. Then we have the following bounds by applying Markov for $p_{1,a}>0$ and $p_{2,a}\in [1,\beta^{-1})$ by \Cref{momentseta}
\begin{equation}
 \Proba{ \eta^{1}(c_{\ell,k})\leq a}\leq \branchmat{\para{\rho_{a}a_{k}c_{\ell}}^{\zeta(-p_{1,a})}  & c_{\ell,k}\leq 1 \\\para{\rho_{a}a_{k}c_{\ell}}^{-p_{1,a}}  & c_{\ell,k}\geq 1    }   =:B_{1,\ell,k}
\end{equation}
and 
\begin{equation}
 \Proba{ a\leq \eta^{1}(c_{\ell+1})}\leq \branchmat{\para{\rho_{a}a_{k}c_{\ell+1}}^{\zeta(p_{2,a})}  & c_{\ell+1}\leq 1 \\ \para{\rho_{a}a_{k}c_{\ell+1}}^{p_{2,a}}  & c_{\ell+1}\geq 1    } =:B_{2,\ell,k}   
\end{equation}
and bound by their minimum $ \Proba{D_{\ell,k}^{a}}\leq\min_{i=1,2}B_{i,\ell,k}$. When $1\geq c_{\ell+1,k}\geq c_{\ell,k}$, we have
\begin{equation}
B_{1,\ell,k}\leq B_{2,\ell,k}    \doncl\ell \geq \para{\frac{c_{\ell}}{c_{\ell+1}}}^{\alpha(1)/\lambda_{a}}\para{\rho_{a}a_{k}}^{-1/\lambda_{a}},
\end{equation}
where $\alpha(1):=\frac{\zeta(p_{2,a})}{\zeta(p_{2,a})-\zeta(-p_{1,a})}$. When $c_{\ell+1,k}\geq 1\geq c_{\ell,k}$, we have 
\begin{equation}
B_{1,\ell,k}\leq B_{2,\ell,k}    \doncl \ell\geq \para{\frac{c_{\ell}}{c_{\ell+1}}}^{\alpha(2)/\lambda_{a}}\para{\rho_{a}a_{k}}^{-1/\lambda_{a}},
\end{equation}
where $\alpha(2):=\frac{p_{2,a}}{p_{2,a}-\zeta(-p_{1,a})}$. Finally, when $c_{\ell+1,k}\geq c_{\ell,k}\geq 1$, we have 
\begin{equation}
B_{1,\ell,k}\leq B_{2,\ell,k}    \doncl \ell\geq \para{\frac{c_{\ell}}{c_{\ell+1}}}^{\alpha(3)/\lambda_{a}}\para{\rho_{a}a_{k}}^{-1/\lambda_{a}},
\end{equation}
where $\alpha(3):=\frac{p_{2,a}}{p_{2,a}+p_{1,a}}$. The ratio is uniformly bounded $\frac{c_{\ell}}{c_{\ell+1}}\in [\frac{1}{2},1]$ and so we upper bound by terms of the form
\begin{eqalign}\label{eq:akfromtheQa}
\sum_{\ell\geq 1} \para{\Proba{D_{\ell,k}^{a}} }^{1/b_{12}}  \lessapprox&  \para{\rho_{a}a_{k}}^{-p_{1,a}/b_{12}}\para{(\rho_{a}a_{k})^{-1/\lambda_{a}}}^{\para{1-p_{1,a}\lambda_{a}}/b_{12}}\\
&+\para{\rho_{a}a_{k}}^{\zeta(p_{2,a})/b_{12}}\para{(\rho_{a}a_{k})^{-1/\lambda_{a}}}^{\para{1+\zeta(p_{2,a})\lambda_{a}}/b_{12}}\\
=&c 2\para{\rho_{a}a_{k}}^{-\frac{1}{b_{12}\lambda_{a} } }=:c_{a}a_{k}^{-\frac{1}{b_{12}\lambda_{a} } },
\end{eqalign}
where for finiteness of the $\ell$-sum we require $\frac{p_{1,a}\lambda_{a}}{b_{12}}>1$ since as $\ell\to +\infty$ we eventually get $c_{\ell}\geq 1$. Due to existence of negative moments of GMC, here we can take $p_{1,a}>0$ arbitrarily large so that we allow $\lambda_{a}=1-\e_{a}$ and $b_{12}$ arbitrarily large which in turn allows $b_{11}=1+\e_{5}$ in \Cref{eq:tintegralcosnstraint}.
In summary the \Cref{eq:summarizedbound} is upper bounded by
\begin{eqalign}\label{eq:summarizedboundtwo}
&\sum_{k}a_{k}^{-\frac{1}{b_{12}\lambda_{a} } }\bigg\{\int_{\R}\para{\Proba{\bigcup_{S_{k,\ell}(u,\tilde{u}) }E_{b-a}^{k}(\tilde{u},y)}}^{1/b_{22}} g_{N(0,1)}(y)\dy \bigg\}^{1/b_{11}}. \end{eqalign}
\proofparagraph{Step 6: Requirement for scaling}
We set $a_{k}:=\rho_{b-a}k^{-\lambda_{b-a}}$ for $k\geq 1$ and $\lambda_{b-a}>0$. We verify the constraints in \Cref{eq:scalinglawconstraintratio}. We required that $\frac{a_{k+1}}{\rho_{b-a}}\in (0,1)$ and 
\begin{eqalign}\label{eq:scalinglawconstraintratioagain}
&\frac{\rho_{b-a}}{a_{k+1}}\para{\abs{ I_{\ell,k}}+a_{k}+a_{k} t^{-1/p}}\leq \frac{1}{2},
\end{eqalign}
where $t\geq 1$. For $k\geq 1$ we have that $\frac{a_{k}}{a_{k+1}}\leq 2^{\lambda_{b-a}} $ and for $\ell\geq 0$ we have
\begin{equation}
 \abs{I_{\ell,k}}=\rho_{a}a_{k}((\ell+1)^{\lambda_{a}}-(\ell)^{\lambda_{a}})\leq    \rho_{a}a_{k}
\end{equation}
and so we ask for $\frac{a_{k+1}}{\rho_{b-a}}\leq  \frac{\rho_{b-a}}{ 2^{\lambda_{b-a}}\rho_{b-a}}<1$ and
\begin{eqalign}
&\frac{\rho_{b-a}}{a_{k+1}}\para{\abs{ I_{\ell,k}}+a_{k}+a_{k} t^{-1/p}}\leq \rho_{b-a} 2^{\lambda_{b-a}}\para{ \rho_{a}+2}<\frac{1}{2}.
\end{eqalign}
\proofparagraph{Step 7: The event $E^{k}_{a,a+x}$}
 Here we are studying
\begin{eqalign}\label{eq:mainkeventsum}
&\sum_{k}a_{k}^{-\frac{1}{b_{12}\lambda_{a} } }\\
&\cdot\Bigg(\int_{\R}\para{\Proba{\bigcup_{u\in [0,\rho_{b-a}]}\set{ \ln\frac{1}{\eta\para{u,u+ L_{2}}}\leq \sigma_{k}y-\frac{\sigma_{k}^{2}}{2} +\ln{ \frac{a_{k+1}}{\rho_{b-a}(b-a)}}\leq  \ln\frac{1}{\eta\para{u,u+ \rho_{b-a}}} }}}^{1/b_{22}}\\
&\cdot g_{N(0,1)}(y)\dy\Bigg)^{1/b_{11}},
\end{eqalign}
where we had set $a_{k}=\rho_{b-a}k^{-\lambda_{b-a}}$ for $k\geq 1$ and $\lambda_{b-a}>0$ and $g_{N(0,1)}$ is the standard normal density,
\begin{eqalign}
\sigma_{k}&:=\sqrt{2\beta\para{\ln \frac{\rho_{b-a}}{a_{k+1}}-1+\frac{a_{k+1}}{\rho_{b-a}}}}\\
\tand M_{k,x}&:=\frac{\sigma_{k}^{2}}{2}+\ln{ \frac{\rho_{b-a}(b-a)}{a_{k+1}}}=\ln\frac{x}{\para{a_{k+1}}^{1+\beta}}+ \beta \para{-1+\frac{a_{k+1}}{\rho_{b-a}}} +\ln{ \rho_{b-a} c_{b-a} }. 
\end{eqalign}
For the probability term, one upper bound follows by taking a minimum and using the \Cref{prop:minmodeta} and \Cref{prop:maxmoduluseta} for $z^{0}_{2}>0$ and $z_{1}^{0}\in (0,\beta^{-1})$
\begin{eqalign}
&\para{\Proba{\bigcup_{u\in [0,\rho_{b-a}]}\set{ \ln\frac{1}{\eta\para{u,u+ L_{2}}}\leq \sigma_{k}y-M_{k,x}\leq  \ln\frac{1}{\eta\para{u,u+ \rho_{b-a}}} }}}^{1/b_{22}}\\
\leq &\para{\minp{\Proba{e^{\sigma_{k}y-M_{k,x}}\leq \max_{u\in [0,\rho_{b-a}]}\frac{1}{\eta\para{u,u+ \rho_{b-a}}} },\Proba{e^{-(\sigma_{k}y-M_{k,x})}\leq \max_{u\in [0,\rho_{b-a}]}\eta\para{u,u+ L_{2}} },1}}^{1/b_{22}}\\
\leq& c_{\rho_{b-a},L_{2}} \minp{e^{-(\sigma_{k}y-M_{k,x})z_{2}^{0}/b_{22}},e^{(\sigma_{k}y-M_{k,x})z_{1}^{0}/b_{22}},1},  
\end{eqalign}
for $z_{2}:=\frac{z_{2}^{0}}{b_{22}}>0$ and $z_{1}:=\frac{z_{1}^{0}}{b_{22}}$. Therefore, we are left with
\begin{eqalign}\label{eq:mainsumintegralGaussiandensity}
\sum_{k\geq 1}a_{k}^{-\frac{1}{b_{12}\lambda_{a} } }\int_{\R} \para{\minp{e^{-(\sigma_{k}y-M_{k,x})z_{2}},e^{(\sigma_{k}y-M_{k,x})z_{1}},1}   g_{N(0,1)}(y)}^{1/b_{11}}\dy,
\end{eqalign}
where we also brought in the power $b_{11}$ because we take $b_{11}=1+\e$ for small $\e>0$. Since it is almost one, we will ignore $b_{11}$ below. We actually only need to study
\begin{eqalign}
\minp{e^{(\sigma_{k}y-M_{k,x})z_{1}},1}.    
\end{eqalign}
\proofparagraph{Step 8: Constraints from the minimum}
Next we identify how the minimum behaves
\begin{eqalign}\label{eq:maininequalityconstraintminimum}
&e^{(\sigma_{k}y-M_{k,x})z_{1}}<1\\
&\doncl M_{k,x}-\sigma_{k}y>0\\
&\doncl  \frac{\sigma_{k}^{2}}{2}+\ln{ \frac{\rho_{b-a}(b-a)}{a_{k+1}}}   -\sigma_{k}y>0\\
&\doncl \ln\frac{x}{\para{a_{k+1}}^{1+\beta}}+ \beta \para{-1+\frac{a_{k+1}}{\rho_{b-a}}} +\ln{ \rho_{b-a} c_{b-a} }-y\sqrt{2\beta\para{\ln \frac{\rho_{b-a}}{a_{k+1}}-1+\frac{a_{k+1}}{\rho_{b-a}}}}>0\\
&\doncl (1+\beta)\lambda_{b-a}\ln (k+1)+ \ln x+C_{1}(k)-y\sqrt{2\beta\lambda_{b-a}\ln (k+1)+C_{2}(k)}>0\\
&\doncl \ln (k+1)+ \frac{\ln x+C_{1}(k)}{(1+\beta)\lambda_{b-a}}-\frac{y}{(1+\beta)\lambda_{b-a}}\sqrt{2\beta\lambda_{b-a}\ln (k+1)+C_{2}(k)}>0,
\end{eqalign}
for 
\begin{eqalign}
C_{1}(k):=& \beta \para{-1+\frac{a_{k+1}}{\rho_{b-a}}}    +\ln{ \rho_{b-a} c_{b-a} }\tand C_{2}(k):=&2\beta\para{-1+\frac{a_{k+1}}{\rho_{b-a}}}.
\end{eqalign}
For all other variables fixed, as $k$ grows to infinity this inequality becomes true. Next we focus on getting concrete lower bounds for $k$. The complication happens as $x\to 0$, so to simplify analysis we just focus on the case
\begin{eqalign}\label{eq:assumptionx}
&\frac{(1+\beta)\para{1-\frac{a_{k+1}}{\rho_{b-a}}}}{\ln\frac{1}{x}-C_{1}(k)}<1.
\end{eqalign}
The reverse inequality of larger $x$ is simpler because $\ln x$ can be treated as a constant that can be ignored by taking large $\lambda_{b-a}$ and so we leave it. For $y\leq 0$, the inequality \Cref{eq:maininequalityconstraintminimum} holds for:
\begin{eqalign}
\ln(k+1)\geq&  \frac{\ln\frac{1}{x}-C_{1}(k)}{(1+\beta)\lambda_{b-a}}=:N_{x,k}
\end{eqalign}
and for $y> 0$ the inequality \Cref{eq:maininequalityconstraintminimum} holds for:
\begin{eqalign}
\ln(k+1)\geq&  \frac{\ln\frac{1}{x}-C_{1}(k)}{(1+\beta)\lambda_{b-a}}+\para{\frac{y}{(1+\beta)\lambda_{b-a}}}^{2}\beta\lambda_{b-a}\\
&+\frac{1}{2}\abs{\frac{y}{(1+\beta)\lambda_{b-a}}}\sqrt{4 \frac{\ln\frac{1}{x}-C_{1}(k)}{(1+\beta)}2\beta+\para{\frac{y}{(1+\beta)}2\beta}^{2}+4C_{2}(k)}\\
=:&N_{x,k}+g_{y}+p(x,y,k),\\
\tfor N_{x,k}:=&\frac{\ln\frac{1}{x}-C_{1}(k)}{(1+\beta)\lambda_{b-a}}, \qquad g_{y}:=\para{\frac{y}{(1+\beta)\lambda_{b-a}}}^{2}\beta\lambda_{b-a},\\
p(x,y,k):=&\frac{1}{2}\abs{\frac{y}{(1+\beta)\lambda_{b-a}}}\sqrt{4 \frac{\ln\frac{1}{x}-C_{1}(k)}{(1+\beta)}2\beta+\para{\frac{y}{(1+\beta)}2\beta}^{2}+4C_{2}(k)}.
\end{eqalign}
Whereas the opposite inequality $ M_{k,x}-\sigma_{k}y\leq 0$ holds  for 
\begin{eqalign}
 y\leq 0 \tand \ln(k+1)\leq&  N_{x,k}+g_{y}-p(x,y,k)
\end{eqalign}
and for
\begin{eqalign}
 y>0\tand \ln(k+1)\leq&  N_{x,k}+g_{y}+p(x,y,k).
\end{eqalign}
 We go over each the four cases of the signs of $ M_{k,x}-\sigma_{k}y\tand y$ to bound the $k-$sum and the $y-$integral and extract a singular factor for $x$.
\proofparagraph{Case $M_{k,x}-\sigma_{k}y>0$ and $y\leq 0$}
Here we just study the tail-sum $k\geq \expo{N_{x,k}}$ for $e^{(\sigma_{k}y-M_{k,x})z_{1}}$
\begin{eqalign}\label{eq:tailsumNxk}
\sum_{k\geq \expo{N_{x,k}}}a_{k}^{-\frac{1}{b_{12}\lambda_{a} } }\expo{-z_{1}M_{k,x}+\sigma_{k}yz_{1}}   \lessapprox& x^{-z_{1}}\sum_{k\geq \expo{N_{x,k}}}k^{-(z_{1}(1+\beta)-\frac{1}{b_{12}\lambda_{a} })\lambda_{b-a}}\expo{-C_{1}(k)} \\
 \lessapprox& x^{-z_{1}} x^{\frac{(z_{1}(1+\beta)-\frac{1}{b_{12}\lambda_{a} })\lambda_{b-a}-1}{(1+\beta)\lambda_{b-a}}}\\
 \lessapprox& x^{-\frac{1}{(1+\beta)\lambda_{b-a}}-\frac{1}{b_{12}\lambda_{a} }}
 \end{eqalign}
for $z_{1}(1+\beta)\lambda_{b-a}$. For summability we require $(z_{1}(1+\beta)-\frac{1}{b_{12}\lambda_{a} })\lambda_{b-a}>1$, which is possible by taking large enough $\lambda_{b-a},b_{12}$. So the exponent of $x$ is of the form $x^{-\e}$ for small $\e>0$.
\proofparagraph{Case $M_{k,x}-\sigma_{k}y>0$ and $y> 0$}
We start by extracting the $y$-part in the summand $a_{k}^{-\frac{1}{b_{12}\lambda_{a} }}e^{(\sigma_{k}y-M_{k,x})z_{1}}$. We ignore the first part $a_{k}^{-\frac{1}{b_{12}\lambda_{a} }}$ since we can take arbitrarily small power. For the second part we factor out the leading asymptotic
\begin{eqalign}\label{eq:mainsummandypositive}
&\sigma_{k}y-M_{k,x}\\
 =&y\sqrt{2\beta\lambda_{b-a}\ln (k+1)+C_{2}(k)}+\ln \frac{1}{x}-C_{1}(k)-(1+\beta)\lambda_{b-a}\ln (k+1)\\
= &\ln\frac{1}{x}-C_{1}(k)-\ln (k+1)\para{(1+\beta)\lambda_{b-a}-\frac{y}{\sqrt{\ln (k+1)}}\sqrt{2\beta\lambda_{b-a}+\frac{C_{2}(k)}{\ln (k+1)}}}. 
\end{eqalign}
Then we use the $k-$ lower bound $G:=\expo{N_{x,k}+g_{y}+p(x,y,k)}$ to bound by
\begin{eqalign}\label{eq:mainsummandypositivelowerboundG}
\eqref{eq:mainsummandypositive}\leq &\ln\frac{1}{x}-C_{1}(k)-\ln (k+1)\para{(1+\beta)\lambda_{b-a}-\frac{y}{\sqrt{\ln (G)}}\sqrt{2\beta\lambda_{b-a}+\frac{C_{2}(k)}{\ln (G)}}}. 
\end{eqalign}
Since the focus in the $y-$integral is for large $y$, we see that for large $y$ the $G$ term is asymptotically equal to
\begin{eqalign}
&G\approx \expo{N_{x,k}+ y^{2}\frac{2\beta}{(1+\beta)^{2}\lambda_{b-a}}}.
\end{eqalign}
Therefore, the $y-$part in \Cref{eq:mainsummandypositivelowerboundG} behaves like 
\begin{equation}
\frac{y}{\sqrt{\ln (G+1)}}\sqrt{2\beta\lambda_{b-a}+\frac{C_{2}(k)}{\ln (G+1)}}\approx    \frac{y}{\ln\frac{1}{x}+y^{2}}.
\end{equation}
However, outside we are multiplying by the standard normal density $e^{-y^{2}}$ and so we still have integrability in the $y-$integral in \Cref{eq:mainsumintegralGaussiandensity}. Also the effect of small $x$ only decreases this bound and so it doesn't contribute a singular factor to $x$. Finally, for simplicity we also drop the $y-$part from the lower bound and thus return to the same tail sum $k\geq N_{x,k}$ as in the previous case in \Cref{eq:tailsumNxk}.
\proofparagraph{Case $M_{k,x}-\sigma_{k}y\leq 0$ and $y\leq 0$}
Here we study the bulk-sum $ k\leq \expo{ N_{x,k}+g_{y}-p(x,y,k)}=:\tilde{G}$. The part $e^{-(\sigma_{k}y-M_{k,x})z_{2}}$ actually grows in $k$, so we simply keep the $1$ and thus get the upper bound
\begin{eqalign}
\sum_{1\leq k\leq \tilde{G} }1\lessapprox \expo{\frac{1}{(1+\beta)\lambda_{b-a}}\ln\frac{1}{x}+\para{\frac{y}{(1+\beta)\lambda_{b-a}}}^{2}\beta\lambda_{b-a}}.  
\end{eqalign}
Here again we get a singular power $x^{-\e}$. We also get integrability in the $y$-integral by taking large $\lambda_{b-a}$ and using the normal density $e^{-y^{2}/2}$.
\proofparagraph{Case $M_{k,x}-\sigma_{k}y\leq 0$ and $y>0$}
Here we study the bulk-sum $ k\leq \expo{ N_{x,k}+g_{y}+p(x,y,k)}=:\bar{G}$. For the $g(y)+p(x,y,k)$ we have an upper bound without $k$ since $C_{1}(k),C_{2}(k)$ are bounded in $k$
\begin{eqalign}
&g(y)+p(x,y,k)\\
\leq &\para{\frac{y}{(1+\beta)\lambda_{b-a}}}^{2}\beta\lambda_{b-a}+\frac{1}{2}\abs{\frac{y}{(1+\beta)\lambda_{b-a}}}\sqrt{4 \frac{\ln\frac{1}{x}+1}{(1+\beta)}2\beta+\para{\frac{y}{(1+\beta)}2\beta}^{2}+4}:=R_{x,y}.
\end{eqalign}
So we get
\begin{equation}
\sum_{1\leq k\leq \bar{G} } \expo{\frac{1}{(1+\beta)\lambda_{b-a}}\ln\frac{1}{x}+R_{x,y}}.    
\end{equation}
We are left with the $y-$integral
\begin{eqalign}
&\int_{0}^{\infty}\expo{R_{x,y}-y^{2}/2}\dy\\
=&\int_{0}^{\infty}\expo{\frac{1}{\lambda_{b-a}}\para{c_{1}y^{2}+c_{2}y\sqrt{c_{3}\ln\frac{1}{x}+c_{4}y^{2}+4} }-y^{2}/2}\dy,    
\end{eqalign}
for some positive constants $c_{i}$. Integrability is clear by taking large $\lambda_{b-a}$. Next we study the singular factor $x$. Using the triangle inequality $\sqrt{a+b}\leq \sqrt{a}+\sqrt{b}$ we are left with 
\begin{eqalign}
&\int_{0}^{\infty}\expo{-\tilde{c}_{1}\para{\frac{1}{2}-\frac{1}{\lambda_{b-a}}}y^{2}+\frac{\tilde{c}_{2}}{\lambda_{b-a}}y\sqrt{\ln\frac{1}{x}} }\dy,    
\end{eqalign}
for some positive $\tilde{c}_{i}$. Here we use the general formula
\begin{equation}\label{eq:gaussianintegralpolynomial}
\int_{0}^{\infty}e^{ay-y^{2}/2}dy=\sqrt{\frac{\pi}{2}}e^{a^{2}/2}Erf(\frac{a}{\sqrt{2}}+1),
\end{equation}
where $Erf(z)$ is the error function $\int_{0}^{z}e^{-t^{2}}\dt$ and so for $a:=\frac{\tilde{c}_{2}}{\lambda_{b-a}}\sqrt{\ln\frac{1}{x}}$ we get the singular factor
\begin{equation}
\expo{\frac{1}{2}\para{\frac{\tilde{c}_{2}}{\lambda_{b-a}}}^{2}\ln{\frac{1}{x}}}=x^{-\frac{1}{2}\para{\frac{\tilde{c}_{2}}{\lambda_{b-a}}}^{2}}.    
\end{equation}
\newpage
\subsection{Proof of \sectm{$\abs{J}=\abs{I}<\delta$ and ratio of the form $\frac{Q(b,b+x)}{Q(a,a+x)}$ for $a<b$}}
We will try to keep find the similarities with the previous case of $\frac{Q(a,a+x)}{Q(b,b+x)}$ as much as possible to shorten it.
\subsubsection{Case \sectm{$Q(a,a+x)\in [1,\infty)$}}
This is again the easiest case, since we can simply upper bound the denominator
\begin{equation}
\Expe{\para{\frac{Q(b,b+x)}{Q(a,a+x)}}^{p}\ind{Q\para{a,a+x}\geq 1}}\leq  1^{-p}\Expe{\para{Q\para{b,b+x}}^{p}}.     
\end{equation}
\subsubsection{Case \sectm{$Q(a,a+x)\in [0,1]$ and $E_{\rho,b-a}:=\set{Q_{b-a+x}\bullet Q_{a}\geq \rho_{b-a+x}}$}}
Because the ratio is flipped here we need to instead divide and multiply by $Q(a,b+x)$. That leaves us with studying
\begin{equation}\label{eq:mainsingularatio}
\Expe{\para{\frac{Q(a,b+x)}{Q(a,a+x)}}^{p}\ind{Q\para{a,a+x}\geq 1}}     
\end{equation}
since $Q(a,b+x)\geq Q(b,b+x)$. Here we repeat the steps up to \Cref{eq:mainprobabilityeventsumsonlylower} to get
\begin{eqalign}
\eqref{eq:mainsingularatio}\leq \sum_{k}b_{k+1}^{p}\int_{1}^{\infty}\sum_{\ell} \Proba{\max_{u\in I_{\ell}}\eta(u,u+t^{-1/p})\geq x, (c_{b-a}+1)x\geq \min_{u\in I_{\ell}}\eta(u,u+b_{k}) ,D_{\ell}^{a}},      
\end{eqalign}
for $D_{\ell}^{a}:=\set{  Q(a)\in [a_{\ell},a_{\ell+1}]=:I_{\ell}}$. From here we proceed the same way we did with \Cref{eq:mainprobabilityeventsumsonlylower} to get a non-singular power of $x$ as in \Cref{eq:exponenntofx}.
\subsubsection{Case \sectm{$Q_{b-a+x}\bullet Q_{a}\in [0,\rho_{b-a+x}]$ and $Q(a,a+x)\in [0,1]$}}
We decompose the numerator 
\begin{equation}
E_{b-a+x}^{k}(Q_{a}):=\set{Q_{b-a+x}\bullet Q_{a}\in I_{k}:=\spara{a_{k+1},a_{k}}\subset [0,\rho_{b-a+x}]},    
\end{equation}
for arbitrarily small $\rho_{b-a+x}\in (0,1)$ and $a_{k}:=\rho_{b-a+x}k^{-\lambda_{b-a}}$ as before, to obtain \begin{eqalign}
&\Expe{\para{\frac{Q_{b-a+x}\bullet Q_{a}}{Q(a,a+x)}}^{pp_{1}}\ind{Q_{b-a}\bullet Q_{a}\in [0,\rho_{b-a+x}],Q(a,a+x)\in [0,1]}}\\
\leq& 1+\sum_{k}\int_{1}^{\infty}\Proba{\eta\para{Q_{a},Q_{a}+ a_{k}t^{-1/pp_{1}} }\geq x,  \eta\para{Q_{a},Q_{a}+ a_{k+1}}\leq b-a+x\leq \eta\para{Q_{a},Q_{a}+ a_{k}}  }\dt.
\end{eqalign}
Here we again need to apply the scaling law and so we write it as a ratio
\begin{eqalign}
\Proba{\frac{\eta\para{Q_{a},Q_{a}+ a_{k}t^{-1/pp_{1}} }}{\eta\para{Q_{a},Q_{a}+ a_{k+1}}}\geq(c_{b-a}+1)^{-1},  \eta\para{Q_{a},Q_{a}+ a_{k+1}}\leq(c_{b-a}+1)x \leq \eta\para{Q_{a},Q_{a}+ a_{k}}  }.    
\end{eqalign}
This is a slightly easier version of the singular case that we already studied in \Cref{eq:maineventosumsingularcase}. It is a bit easier because here the ratio has the same shift $Q_{a}$. So from here we again get the same singular power of $x^{-\e}$.
\newpage
 \proofparagraph{Inverse on the unit circle}
Here we study the ratios
\begin{equation}
\Expe{\para{\frac{Q_{H}(a,a+x)}{Q_{H}(b,b+x)}}^{p}}   \tand \Expe{\para{\frac{Q_{H}(b,b+x)}{Q_{H}(a,a+x)}}^{p}}.     
\end{equation}
We will go over all the possibilities of $Q_{H}\para{b,b+x},Q_{H}\para{a,a+x}$ depending on them being less or greater than one. The strategy here is to switch back from $\eta_{H}$ to the measure $\eta=\eta_{U}$ as described in \Cref{it:doubleboundinv} via $\eta_{H}=G\eta$ and then use the results for $\eta$. 
\subsection{Case \sectm{$Q_{H}(b,b+x)\in [1,\infty)$}}
This is again the easiest case as before.
\subsection{Case \sectm{$Q_{H}(b,b+x)\in [0,1]$ and $Q_{H}(b-a)\bullet Q_{H}(a)\geq \rho_{b-a}$}}
Again we use $x\leq b-a$ to bound
\begin{eqalign}\label{eq:casecirculebiggerrho}
&\Expe{\para{\frac{Q_{H}(a,a+x)}{Q_{H}(b,b+x)}}^{p}\ind{E_{b,b+x}^{[0,1]}}}\leq \Expe{\para{\frac{Q_{H}(b-a)\bullet Q_{H}(a)}{Q_{H}(b,b+x)}}^{p}\ind{E_{b,b+x}^{[0,1]}}}. 
\end{eqalign}
Here we will decompose the numerator event and $Q_{H}(a)$
\begin{eqalign}
E_{b-a}^{k}:=&\set{Q_{H}(b-a)\bullet Q_{H}(a)\in \spara{b_{k},b_{k+1}}=:I_{b-a}^{k}}, k\geq 1\\
D_{\ell}^{a}:=&\set{  Q_{H}(a)\in [a_{\ell},a_{\ell+1}]=:I_{\ell}},\ell\geq 0,
\end{eqalign}
for the same sequence $b_{k}\to +\infty,a_{\ell}\to +\infty$ as before $b_{k}:=\rho_{b-a}k^{\lambda_{2}}$ for some $\lambda_{2}\in (0,1)$ and  $a_{\ell}:=\rho_{1}\ell^{\lambda_{1}},\ell\geq 0$ and $\rho_{1},\lambda_{1}\in (0,1)$. Thus as before, we obtain an analogous upper bound to \Cref{eq:maineventoboundratioone}
\begin{eqalign}\label{eq:maineventoboundratioonecircle}
&\eqref{eq:casecirculebiggerrho}\leq\int_{[1,\infty)}\sum_{k}\sum_{\ell}b_{k+1}^{p}\\
&\minp{\Proba{\max_{u\in [0,\abs{I_{\ell}}],\tilde{u}\in [0,\abs{I_{b-a}^{k}}]}\eta_{H}(\tilde{u}+u,\tilde{u}+u+t^{-1/p})\geq x}\Proba{c_{b-a}x\geq \min_{u\in [0,\abs{I_{\ell}}]}\eta_{H}(u,u+b_{k})}, \Proba{D_{\ell}^{a}}},
\end{eqalign}
where we also applied the periodicity and the translation invariance. Here the only difference with \Cref{eq:maineventoboundratioone} is that we switch back to the measure $\eta$ as described in \Cref{sec:transitionformuals} via $\eta_{H}=G\eta$ and then use the results for $\eta$. So for the first event depending on $t$, we first apply Markov for $p_{0}>0$
\begin{eqalign}\label{eq:unitcirclemaxxi}
&\Proba{\max_{u\in [0,\abs{I_{\ell}}],\tilde{u}\in [0,\abs{I_{b-a}^{k}}]}\eta_{H}(\tilde{u}+u,\tilde{u}+u+t^{-1/p})\geq x}\\
\leq & x^{-p_{0}}\Expe{\max_{T\in [0,\rho_{b-a}+\rho_{1}]}\para{\eta_{H}(T,T+t^{-1/p})}^{p_{0}}} \\
\leq & x^{-p_{0}}\Expe{e^{p_{0}\sup_{s\in [0,1]}\xi(s)}\max_{T\in [0,\rho_{b-a}+\rho_{1}]}\para{\eta(T,T+t^{-1/p})}^{p_{0}}} 
\end{eqalign}
then apply \Holder for pair $p_{0,1},p_{0,2}$ to separate the $\xi$-factor and finally use the positive moments in \Cref{cor:shiftedGMCmoments} 
\begin{equation}
\eqref{eq:unitcirclemaxxi}\leq  x^{-p_{0}}\para{\Expe{e^{p_{0}p_{0,1}\sup_{s\in [0,1]}\xi(s)}}}^{1/p_{0,1}} t^{-\frac{\zeta(p_{0}p_{0,2})-1}{p_{0,2}}}.
\end{equation}
We use that the supremum of the $\xi$-field from \Cref{rem:finiteexponentialmomentscomparisonfield} has finite exponential moments over any closed interval $I$
\begin{equation}
\Expe{\expo{\alpha\sup_{s\in I}\xi(s)}}<\infty \tfor \alpha>0.   
\end{equation}
So here we can take $p_{0,2}=1+\e$ by pushing its \Holder conjugate $p_{0,1}$ to be arbitrarily large. We do the same for the second factor with the $b_{k}$. So for the first two factors we return to the previous setting in \cref{eq:reallinelinemaxbound}. It remains to study the $\Proba{D_{\ell}^{a}}$.
\proofparagraph{Event $D_{\ell}^{a}$ for the $Q_{H}(a)$}
Here we have the same result as before because as shown in \Cref{prop:momentsunitcirclerealine}, the exponents for the q-moments of $\eta_{H}(t)$ and $t\in [0,\infty)$ are the same those of $\eta(t)$ as long as they are in $q\in (-\infty,0)\cup [1,\beta^{-1})$ with the only difference of bounding by $\floor{t}^{q}$ when $t\geq 1$ and $q<0$.
\subsection{Case \sectm{$Q_{H}(b-a)\bullet Q_{H}(a)\in [0,\rho_{b-a}]$ and $Q_{H}(b,b+x)\in [0,1]$}}
Here we again study the singular case. We follow the same overall strategy. We mainly switch back to $\eta$ at the cost of having to control a $\xi$-factor as above.
\proofparagraph{Step 1: Decomposing $Q_{H}(b-a)\bullet Q_{H}(a)$ and $Q_{H}(a)$} We repeat similar steps to get to the analogue of \Cref{eq:maineventrealline}
\begin{eqalign}\label{eq:maineventcircle}
&\sum_{\ell}\sum_{k}\para{\Proba{D_{\ell,k}^{a}} }^{1/b_{12}} \\ \cdot&\int_{1}^{\infty}\para{\Expe{\maxls{\set{u\in I_{k},\tilde{u}\in \spara{0,\abs{I_{\ell,k}}}}}\ind{\frac{\eta_{H}\para{\tilde{u}+u,\tilde{u}+u+ a_{k}t^{-1/p}}}{\eta_{H}\para{\tilde{u},\tilde{u}+ a_{k+1}}}\geq c_{b-a}^{-1} ,E_{b-a}^{k}(\tilde{u})} }}^{1/b_{11} }\dt,  
\end{eqalign}
where 
\begin{eqalign}
&D_{\ell,k}^{a}:=\set{  Q_{H}(a)\in [c_{\ell,k},c_{\ell+1,k}]=:I_{\ell,k}},  \\
\tand &E_{b-a}^{k}(\tilde{u}):=\set{Q_{H}(b-a)\bullet\tilde{u}\in I_{k}:=\spara{a_{k+1},a_{k}}\subset [0,\rho_{b-a}]}.    
\end{eqalign}
Now we make the transition to $\eta$. In the event $E_{b-a}^{k}(\tilde{u})$ we use \Cref{it:doubleboundinv} to write
\begin{eqalign}
&E_{(b-a)G_{\tilde{u}}^{1}}^{k}(\tilde{u}):=\set{Q_{U}\para{(b-a)e^{G_{\tilde{u}}^{1}}}\bullet\tilde{u}\in I_{k}:=\spara{a_{k+1},a_{k}}\subset [0,\rho_{b-a}]}.    
\end{eqalign}
where 
\begin{equation}
G_{\tilde{u}}^{1}= \xi(\tilde{u}+\theta_{[0,b-a]}) , 
 \end{equation}
 and $\theta_{[0,b-a]}\in [0,Q_{H}(b-a)\bullet T]\subseteq [0,\rho_{b-a}]$. For the ratio event we pull a $\xi$ field from numerator and denominator and combine them
\begin{eqalign}
\frac{\eta_{H}\para{\tilde{u}+u,\tilde{u}+u+ a_{k}t^{-1/p}}}{\eta_{H}\para{\tilde{u},\tilde{u}+ a_{k+1}}}\leq e^{G_{u,\tilde{u}}^{2}-G_{\tilde{u}}^{3}}    \frac{\eta\para{\tilde{u}+u,\tilde{u}+u+ a_{k}t^{-1/p}}}{\eta\para{\tilde{u},\tilde{u}+ a_{k+1}}},
\end{eqalign}
where 
\begin{eqalign}
G_{u,\tilde{u}}^{2}&:=\sup_{\theta\in [0,\rho_{b-a}]}\xi(\tilde{u}+u+\theta) ,\\
G_{\tilde{u}}^{3}&:=\sup_{\theta\in [0,\rho_{b-a}]}\xi(\tilde{u}+\theta).
\end{eqalign}
 In order to progress to the scaling law, we need to remove these two suprema. Since they have all their exponential moments, we decompose them and separate them out by \Holder. We use
\begin{eqalign}
D^{G^{1}}_{m_{1}}(\tilde{u}):&=  \set{  G_{\tilde{u}}^{1}\in [d_{m_{1}},d_{m_{1}+1}]=:I_{m}^{G^{1}}}\tfor m_{1}\in \Z_0,\\
D^{G^{2}}_{m_{2}}(u,\tilde{u}):&=  \set{  G_{u,\tilde{u}}^{2}\in [d_{m_{2}},d_{m_{2}+1}]=:I_{m_{2}}^{G^{2}}}\tfor m_{2}\in \Z_0,\\
D^{G^{3}}_{m_{3}}(\tilde{u}):&=  \set{  G_{\tilde{u}}^{3}\in [d_{m_{3}},d_{m_{3}+1}]=:I_{m_{3}}^{G^{3}}}\tfor m_{3}\in \Z_0,
\end{eqalign}
where $\Z_0:=\Z\setminus \set{0}$ and $d_{m}$ is a bi-infinite sequence going to minus and plus infinity to cover the image of those constants defined in \Cref{eq:summarizedbound}: we set
\begin{equation}\label{eq:sequencesdefinition}
d_{k}:=\branchmat{\ln k^{\lambda}& k\geq 1 \\ \ln \abs{k}^{-\lambda}& k\leq -1 \\},   
\end{equation}
 for $k\in \Z_0$ and $\lambda>0$. So finally, we take yet another supremum and apply triple \Holder to separate these events
\begin{eqalign}\label{eq:maineventcircletwo}
\eqref{eq:maineventcircle}&\leq \sum_{\ell,k,m_{i}}\para{\Proba{D_{\ell,k}^{a}} }^{1/b_{12}}\int_{1}^{\infty}\dt\\
&\para{\Expe{\maxls{S}\ind{e^{u_{2}-u_{3}}\frac{\eta\para{\tilde{u}+u,\tilde{u}+u+ a_{k}t^{-1/p}}}{\eta\para{\tilde{u},\tilde{u}+ a_{k+1}}}\geq c_{b-a}^{-1} ,E_{b-a}^{k}(\tilde{u},u_{1})} }}^{1/b_{11}r_{ratio} }\\
&\cdot\prod_{i=1,3}\para{\Proba{\max_{\tilde{u}\in \spara{0,\rho_{a}}}\ind{D^{G^{i}}_{m_{i}}(\tilde{u})}}}^{1/r_{G^{i}}} \para{\Proba{\max_{\set{u\in [0,\rho_{b-a}],\tilde{u}\in \spara{0,\rho_{a}}}}\ind{D^{G^{2}}_{m_{2}}(u,\tilde{u})}}}^{1/r_{G^{2}}},  
\end{eqalign}
 where $r_{ratio}^{-1}+\sum_i r_{G^{i}}^{-1}=1$ and
 \begin{equation}
S:=\set{u\in I_{k},\tilde{u}\in \spara{0,\abs{I_{\ell,k}}} ,\set{u_{i}\in I_{m_i}^{G^i}}_{1,2,3}}     
 \end{equation}
 and we also used the bounds on $\abs{I_{k}},\abs{I_{\ell,k}}$ from the previous section to decouple the factors $G^{i}$ from the $k,\ell$-sums. Here we take $r_{ratio}=1+\e$ for arbitrarily small $\e>0$, at the cost of taking arbitrarily large $r_{G^{i}}$.
\proofparagraph{Step 2: Scaling law step} This step proceed as before in \Cref{eq:summarizedbound} with the only difference that we now include the $e^{u_{i}}$ factors
\begin{eqalign}\label{eq:maineventcirclethree}
&\para{\Proba{\bigcup_{\tilde{S}}\set{e^{u_{2}-u_{3}}\frac{\eta\para{\tilde{u}+u,\tilde{u}+u+ L_{2}t^{-1/p}}}{\eta\para{\tilde{u},\tilde{u}+ \rho}}}\geq c_{b-a}^{-1}}}^{1/b_{21}b_{11}r_{ratio}}\\
\cdot&\bigg\{\int_{\R}\para{\Proba{\bigcup_{\tilde{S} }E_{b-a}^{k}(\tilde{u},y,u_{1})}}^{1/b_{22}} g_{N(0,1)}(y)\dy \bigg\}^{1/b_{11}r_{ratio}},
\end{eqalign}
where the maximum-domain changed to
 \begin{equation}
\tilde{S}:=\set{u\in \frac{\rho}{a_{k+1}}I_{k},\tilde{u}\in\spara{0, \frac{\rho}{a_{k+1}}\abs{I_{\ell,k}}} ,\set{u_{i}\in I_{m_i}^{G^i}}_{1,2,3} }.     
 \end{equation}
and the decomposition event changed to 
\begin{eqalign}
E_{b-a}^{k}(\tilde{u},y,u_{1})&=\set{\ln\frac{1}{\eta\para{\tilde{u},\tilde{u}+ L_{2}}}\leq \sigma_{k}N(0,1)-\frac{\sigma_{k}^{2}}{2} +\ln{ \frac{a_{k+1}}{e^{u_{1}}\rho(b-a)}}\leq  \ln\frac{1}{\eta\para{\tilde{u},\tilde{u}+ \rho}}},        
\end{eqalign}
with the difference being that we added the $e^{u_{1}}$-factor.
\proofparagraph{Step 3: Ratio modulus}
Now, we apply Markov inequality for $q_{0}\in [1,\beta^{-1})$ and \Holder inequality for $q_{1}^{-1}+q_{2}^{-1}=1$ with $q_{0}q_{1}\in \para{1,\beta^{-1}}$ to separate numerator and denominator
\begin{eqalign}
&\para{\Proba{\bigcup_{\tilde{S}}\set{e^{u_{2}-u_{3}}\frac{\eta\para{\tilde{u}+u,\tilde{u}+u+ L_{2}t^{-1/p}}}{\eta\para{\tilde{u},\tilde{u}+ \rho}}}\geq c_{b-a}^{-1}}}^{1/b_{21}b_{11}r_{ratio}}   \\
\leq &\para{\Expe{\max_{\tilde{S}}e^{q_{0}\para{d_{m_{2}+1}-d_{m_{3}}}}\para{\eta\para{u+\tilde{u},u+\tilde{u}+\frac{L_{2}}{t^{1/p} } }}^{q_{0}q_{1}}}}^{1/q_{1}b_{21}b_{11}r_{ratio}}\\
&\cdot \para{\Expe{\max_{\tilde{S}}\para{\eta\para{u+\tilde{u},u+\tilde{u}+\rho }}^{-q_{0}q_{2}}}}^{1/q_{2}b_{21}b_{11}r_{ratio}}
\end{eqalign}
We then use \Cref{prop:maxmoduluseta} to get
\begin{eqalign}
&\para{\Expe{\max_{\tilde{S}}e^{q_{0}\para{d_{m_{2}+1}-d_{m_{3}}}}\para{\eta\para{u+\tilde{u},u+\tilde{u}+\frac{L_{2}}{t^{1/p} } }}^{q_{0}q_{1}}}}^{1/q_{1}b_{21}b_{11}r_{ratio}}\\
\lessapprox& e^{q_{0}\para{d_{m_{2}+1}-d_{m_{3}}}}\para{\frac{L_{2}}{t^{1/p} }}^{\frac{\zeta(q_{0}q_{1})-1}{q_{1}b_{21}b_{11}r_{ratio}}}\\
\lessapprox &e^{q_{0}\para{d_{m_{2}+1}-d_{m_{3}}}}t^{-\frac{\zeta(q_{0}q_{1})-1}{pq_{1}b_{21}b_{11}r_{ratio}}}
\end{eqalign}
and we use the \Cref{prop:minmodeta} to get
\begin{eqalign}
\para{\Expe{\max_{\tilde{S}}\para{\eta\para{u+\tilde{u},u+\tilde{u}+\rho }}^{-q_{0}q_{2}}}}^{1/q_{2}}\lessapprox~\rho^{\frac{\zeta(-q_{0}q_{2})-1}{q_{2}b_{21}b_{11}r_{ratio}}}.
\end{eqalign}
\proofparagraph{Step 4: Integral over $t$ and $\para{\frac{Q_{b-a}\bullet Q_{a}}{Q(b,b+x)}}^{p}$} Here the only difference in the exponent of $t$ in \Cref{eq:boundexponentfort} is that we divide by $r_{ratio}$, which we can take it to be $r_{ratio}=1+\e$ by taking $r_{G},r_{G^{2}}$ arbitrarily large. So we get the same constraint
\begin{equation}
\beta^{-1}>2+o(\e_{i})\doncl \gamma<\sqrt{\frac{2}{2+o(\e_{i})}}  \doncl \gamma<1.
\end{equation}
\proofparagraph{Step 5: Sum over $\ell$ and event $D_{\ell,k}^{a}$}
Here the only difference in \Cref{eq:akfromtheQa} is that we restrict $p_{2}\in [1,\beta^{-1})$ in order that we have the convexity-condition that we need in \Cref{prop:momentsunitcirclerealine}. So we again get the singular factor $c_{a}a_{k}^{-\frac{1}{b_{12}\lambda_{a} } }.$
\proofparagraph{Step 6: Sum over $k$ and event $E^{k}_{a,a+x}$}
 Here we are now studying
\begin{eqalign}
&\int_{\R}\para{\Proba{\bigcup_{u_{1}\in I_{m_{1}}^{G_{1}} ,u\in [0,\rho]}\set{ \ln\frac{1}{\eta\para{\tilde{u},\tilde{u}+ L_{2}}}\leq \sigma_{k}y-\frac{\sigma_{k}^{2}}{2} +\ln{ \frac{a_{k+1}}{e^{u_{1}}\rho(b-a)}}\leq  \ln\frac{1}{\eta\para{\tilde{u},\tilde{u}+ \rho}} }}}^{1/b_{22}}\\
&\cdot g_{N(0,1)}(y)\dy,
\end{eqalign}
where now we have an extra maximum for $u_{1}$. For the probability term, one upper bound follows by taking a minimum and using the \Cref{prop:minmodeta},\Cref{prop:maxmoduluseta}
\begin{eqalign}
&\para{\Proba{\bigcup_{u_{1}\in I_{m_{1}}^{G_{1}} ,u\in [0,\rho]}\set{ \ln\frac{1}{\eta\para{u,u+ L_{2}}}\leq y-u_{1}\leq  \ln\frac{1}{\eta\para{u,u+ \rho}} }}}^{1/b_{22}}\\
&\leq c_{\rho,L_{2}} \minp{e^{-yz_{2}^{0}/b_{22}}e^{d_{m_{1}+1}z_{2}^{0}/b_{22}},e^{yz_{1}^{0}/b_{22}}e^{-d_{m_{1}}z_{1}^{0}/b_{22}}},    
\end{eqalign}
for $z_{2}:=\frac{z_{2}^{0}}{b_{22}}>0$ and $z_{1}:=\frac{z_{1}^{0}}{b_{22}}\twith z_{1}^{0}\in (0,\beta^{-1})$. Here reverse the change of variables again to be left with
\begin{eqalign}
\int_{\R} \minp{e^{-(\sigma_{k}y-M_{k,x})z_{2}}e^{d_{m_{1}+1}z_{2}},e^{(\sigma_{k}y-M_{k,x})z_{1}}e^{-d_{m_{1}}z_{1}},1}   g_{N(0,1)}(y)\dy. 
\end{eqalign}
Here we returned to a similar integral as in \Cref{eq:mainsumintegralGaussiandensity}. The only difference is that we picked up the factors $e^{d_{m_{1}+1}z_{2}},e^{-d_{m_{1}}z_{1}}$. We will need to extract the $e^{d_{m_{1}+1}}$ factors for the summations. We again only study
\begin{eqalign}
\minp{e^{(\sigma_{k}y-M_{k,x})z_{1}}e^{-d_{m_{1}}z_{1}},1}.    
\end{eqalign}
So we instead split over
\begin{eqalign}
&M_{k,x}+d_{m_{1}}-\sigma_{k}y>0\tand M_{k,x}+d_{m_{1}}-\sigma_{k}y\leq 0\\
\doncl&    \ln (k+1)+ \frac{\ln (xe^{d_{m_{1}}})++C_{1}(k)}{(1+\beta)\lambda_{b-a}}-\frac{y}{(1+\beta)\lambda_{b-a}}\sqrt{2\beta\lambda_{b-a}\ln (k+1)+C_{2}(k)}>0,
\end{eqalign}
where for convenience we grouped together the $x$ and $e^{d_{m_{1}}}$ and let $\tilde{x}:=xe^{d_{m_{1}}}$.  As defined in \Cref{eq:sequencesdefinition}, this $\tilde{x}$ can be arbitrarily small and arbitrarily large and so we have to slightly modify the proof. In the case of $m_{1}\leq -1$, we get that $\tilde{x}\leq x$ and we again satisfy \Cref{eq:assumptionx} i.e.
\begin{eqalign}\label{eq:assumptionxem}
&\frac{(1+\beta)\para{1-\frac{a_{k+1}}{\rho_{b-a}}}}{\ln\frac{1}{xe^{d_{m_{1}}}}-C_{1}(k)}<1.
\end{eqalign}
In this case we proceed as before writing $\tilde{x}$ instead of $x$, in order to be left with 
\begin{eqalign}
    \tilde{x}^{-\e}=x^{-\e}e^{-\e d_{m_{1}}}.
\end{eqalign}
So it remains to study the case $m_{1}\geq 1$ where we have $\tilde{x}\geq x$ and arbitrarily large as $m_{1}$ grows. For the opposite inequality 
\begin{eqalign}
&\frac{(1+\beta)\para{1-\frac{a_{k+1}}{\rho_{b-a}}}}{\ln\frac{1}{xe^{d_{m_{1}}}}-C_{1}(k)}\geq 1
\end{eqalign}
we get similar lower bounds. For all $y\in (-\infty,\infty)$, we get a solution if we require the same $N_{x,k}$ as before but now replacing $x$ by $\tilde{x}$
\begin{eqalign}
\ln(k+1)\geq&  \frac{\ln\frac{1}{\tilde{x}}-C_{1}(k)}{(1+\beta)\lambda_{b-a}}+\frac{1}{2}\para{\frac{y}{(1+\beta)\lambda_{b-a}}}^{2}\beta\lambda_{b-a}\\
&+\frac{1}{2}\abs{\frac{y}{(1+\beta)\lambda_{b-a}}}\sqrt{4 \frac{\ln\frac{1}{x}-C_{1}(k)}{(1+\beta)}2\beta+\para{\frac{y}{(1+\beta)}2\beta}^{2}+4C_{2}(k)}\\
=:&N_{\tilde{x},k}+g_{y}+p(\tilde{x},y,k),\\
\tfor N_{\tilde{x},k}:=&\frac{\ln\frac{1}{\tilde{x}}-C_{1}(k)}{(1+\beta)\lambda_{b-a}}, \qquad g_{y}:=\para{\frac{y}{(1+\beta)\lambda_{b-a}}}^{2}\beta\lambda_{b-a},\\
p(\tilde{x},y,k):=&\frac{1}{2}\abs{\frac{y}{(1+\beta)\lambda_{b-a}}}\sqrt{4 \frac{\ln\frac{1}{\tilde{x}}-C_{1}(k)}{(1+\beta)}2\beta+\para{\frac{y}{(1+\beta)}2\beta}^{2}+4C_{2}(k)}.
\end{eqalign}
If we follow the same steps, we again get the singular factor
\begin{eqalign}
    \tilde{x}^{-\e}=x^{-\e}e^{-\e d_{m_{1}}}.
\end{eqalign}
Now it finally remains to study the decomposition events for $\xi$.
\proofparagraph{Step 7: Events $D^{G^{i}}$}
Here we study
\begin{eqalign}
\sum_{m_{i}}  & e^{q_{0}\para{d_{m_{2}+1}-d_{m_{3}}}}e^{-\e d_{m_{1}}z_{1}}\\&\cdot\prod_{i=1,3}\para{\Proba{\max_{\tilde{u}\in \spara{0,\rho_{a}}}\ind{D^{G^{i}}_{m_{i}}(\tilde{u})}}}^{1/r_{G^{i}}} \para{\Proba{\max_{\set{u\in [0,\rho_{b-a}],\tilde{u}\in \spara{0,\rho_{a}}}}\ind{D^{G^{2}}_{m_{2}}(u,\tilde{u})}}}^{1/r_{G^{2}}},
\end{eqalign}
where 
\begin{eqalign}
D^{G^{1}}_{m_{1}}(\tilde{u}):&=  \set{  G_{\tilde{u}}^{1}\in [d_{m_{1}},d_{m_{1}+1}]=:I_{m_{1}}^{G^{1}}}\tfor m_{1}\in \Z_{0},\\
D^{G^{2}}_{m_{2}}(u,\tilde{u}):&=  \set{  G_{u,\tilde{u}}^{2}\in [d_{m_{2}},d_{m_{2}+1}]=:I_{m_{2}}^{G^{2}}}\tfor m_{2}\in \Z_{0},\\
D^{G^{3}}_{m_{3}}(\tilde{u}):&=  \set{  G_{\tilde{u}}^{3}\in [d_{m_{3}},d_{m_{3}+1}]=:I_{m_{3}}^{G^{3}}}\tfor m_{3}\in \Z_{0},
\end{eqalign}
and 
\begin{eqalign}
G_{\tilde{u}}^{1}&:=\xi(\tilde{u}+\theta_{[0,b-a]}), \\
G_{u,\tilde{u}}^{2}&:=\sup_{\theta\in [0,\rho_{b-a}]}\xi(\tilde{u}+u+\theta),\\
G_{\tilde{u}}^{3}&:=\sup_{\theta\in [0,\rho_{b-a}]}\xi(\tilde{u}+\theta).
\end{eqalign}
These sums are all decoupled and so we study them separately. We mainly use the finiteness of exponential moments \Cref{rem:finiteexponentialmomentscomparisonfield}. Since the three events are handled similarly, we just focus on the on the event $D^{G^{2}}$.
\proofparagraph{Step 8: Event $D^{G^{2}}$}
We set
\begin{equation}
e_{k}=\branchmat{\ln k^{\lambda}& k\geq 1 \\ \ln \abs{k}^{-\lambda}& k\leq -1 \\},   
\end{equation}
 for $k\in \Z_0$ and $\lambda>0$. When we have $d_{m_{2}}<0$, we keep only the upper bound and apply Chernoff inequality for $\alpha_{2,1}>0$
\begin{eqalign}\label{eq:q0alpha21}
&\para{\Proba{\max_{\set{u\in [0,\rho_{b-a}],\tilde{u}\in \spara{0,\rho_{a}}}}\ind{D^{G^{2}}_{m_{2}}(u,\tilde{u})}}}^{1/r_{G^{2}}}\\
\leq & e^{\alpha_{2,1} d_{m_{2}+1}/r_{G^{2}} }\para{\Expe{\supl{s\in [0,\rho_{a}+2\rho_{b-a}]}\expo{\alpha_{2,1}\abs{\gamma \para{H_{\e}(s)-U_{\e}(s)}}}}}^{1/r_{G^{2}}}.
\end{eqalign}
Here we take large enough $\alpha_{2,1}$ to satisfy the finiteness constraint $(q_{0}+\frac{\alpha_{2,1}}{r_{G^{2}}})\lambda>1$ for the $k$-sum. Similarly, when $d_{m_{2}}>0$, we keep the lower bound and apply Chernoff inequality for $\alpha_{2,2}>0$
\begin{eqalign}
&\para{\Proba{\max_{\set{u\in [0,\rho_{b-a}],\tilde{u}\in \spara{0,\rho_{a}}}}\ind{D^{G^{2}}_{m_{2}}(u,\tilde{u})}}}^{1/r_{G^{2}}}\\
\leq & e^{-\alpha_{2,2} d_{m_{2}} /r_{G^{2}}}\para{\Expe{\supl{s\in [0,\rho_{a}+2\rho_{b-a}]}\expo{\alpha_{2,2}\abs{\gamma \para{H_{\e}(s)-U_{\e}(s)}}}}}^{1/r_{G^{2}}}.
\end{eqalign}
\subsection{Case \sectm{$\abs{J}=\abs{I}<\delta$ and ratio of the form $\frac{Q_{H}(b,b+x)}{Q_{H}(a,a+x)}$ for $a<b$}}
Bounding this ratio follows the same strategy as for the case on the real line. In particular for the singular cases, we have the same situation.  The $q_{0}$ being arbitrarily small didn't affect it either since in \Cref{eq:q0alpha21} we can again take $\alpha_{2,1}$ large enough.


\newpage \part{Multipoint estimates }\label{partmultipointestimates}
 In this part we study multipoint estimates that appear in the \cite{AJKS}-framework (eg. \cite[eq. (89)]{AJKS}).

\section{Multipoint moments}
We fix $N\geq 1$ and generic intervals $[a_{k},b_{k}]$ and $J_{k}:=[c_{k},c_{k}+x_{k}]\subseteq [a_{k},b_{k}]$ for $k\in [N]$. We consider the gap event for some subset $S:=\set{i_{1},\cdots,i_{M}}\subset \set{1,\cdots,N}$ with $M\geq \alpha N$, for some $\alpha>0$, such that we have large enough gaps
\begin{equation}
G_{S}:=G_{i_{1},\cdots,i_{M}}:=\bigcap_{k\in S}\set{Q_{a_{i_{k}}}^{i_{k}}-Q_{b_{i_{k+1}}}^{i_{k+1}}\geq \delta_{i_{k+1}}}.
\end{equation}
If we intersect with this event, we get the following estimate on the multipoint-moments.
\begin{proposition}
We fix $p_{k},k\in S$ then we have
\begin{eqalign}
\Expe{\prod\limits_{k\in S} \para{Q^{k}(J_{k})}^{p_{k}} \ind{G_{S}}}\leq c^{\abs{S}} \prod_{k\in S} x_{k}^{q_{k}},
\end{eqalign}
for $q_{k}>0$ satisfying
\begin{eqalign}
p_{k}>-\zeta(-q_{k})+1, k\in S.
\end{eqalign}
\end{proposition}
\begin{remark}
The natural follow-up question is whether the gap event $G_{S}$ has high enough probability. If we start with the joint tail of inverses $Q^{k}(J_{k})$ with $J_{k}:=[c_{k},c_{k}+x_{k}]\subseteq [a_{k},b_{k}]$ , we upper bound by the decoupled ratio plus the error $g_{Gap}(N)$
\begin{eqalign}
\Proba{\bigcap\limits_{k\in [N]}  \set{Q^{k}(J_{k})> R_{k}}} \leq &\Proba{\bigcap\limits_{k\in [N]}  \set{Q^{k}(J_{k})> R_{k}}\cap \set{\alpha(G)\geq \alpha N}} +g_{Gap}(N)\\
\leq &\sumls{S\subset [N]\\\abs{S}\geq \alpha N}\Proba{\bigcap\limits_{k\in S}  \set{Q^{k}(J_{k})> R_{k}}\cap G_{S} }+g_{Gap}(N).
\end{eqalign}
So here to get decay $g_{Gap}(N)\to 0$, this will impose constraints on $a_{k},b_{k},\delta_{k}$ as described in \cref{indnumbertheorem}.    
\end{remark}

\begin{proof}
To ease on the notation we simply write $1,....,M$ as opposed to $i_{1},...,i_{M}$ since for now the argument only requires them to have gaps and it is not number-specific. The strategy is to peel-off one moment at a time by successively decomposing the $\para{Q^1(a_{1}),Q^{2}(b_{2})},\para{Q^{2}(a_{2}),Q^{3}(b_{3})},...$. We partially follow the proof-logic in the single-point moments of $Q([a,b])$ from \cite{binder2023inverse}.
\proofparagraph{Step 1: $\delta$-Strong Markov}
We start with the largest scale $\para{Q^{1}(J_{1})}^{p_{1}}$ for $J_{1}:=[c_{1},c_{1}+x_{1}]$ and denote the product of other scales as 
\begin{equation}
Y_{[2, M]}:= \prod\limits_{k\in [2,M]} \para{Q^{k}(J_{k})}^{p_{k}}   \ind{ G_{[3,M]}}    
\end{equation}
We write
\begin{eqalign}\label{eq:multipointmoments1}
\Expe{\prod\limits_{k\in S} \para{Q^{k}(J_{k})}^{p_{k}} \ind{G_{S}}}=&\int_{0}^{B_{\delta}}\Expe{\ind{x_{1}\geq \eta^{1}\para{Q_{c_{1}}^{1},Q_{c_{1}}^{1}+ t }} \ind{G_{1,2}}Y_{[2, M]} }p_{1}t^{p_{1}-1}\dt\\
&+\int_{B_{\delta}}^{\infty}\Expe{\ind{x_{1}\geq \eta^{1}\para{Q_{c_{1}}^{1},Q_{c_{1}}^{1}+ t }} \ind{G_{1,2}}Y_{[2, M]} }p_{1}t^{p_{1}-1}\dt,
\end{eqalign}
for $B_{\delta}:=\minp{2,\maxp{2\delta,1}}\in [1,2]$ which satisfies $B_{\delta}-\delta\geq 1/2$. For the second integral, we simply lower bound
\begin{eqalign}
\eta^{1}\para{Q_{c_{1}}^{1},Q_{c_{1}}^{1}+ t }  \geq   \eta^{1}\para{Q_{c_{1}}^{1}+\delta_{1},Q_{c_{1}}^{1}+ t}
\end{eqalign}
and use \cref{deltaSMP}. This gives the bound
\begin{eqalign}
&\int_{B_{\delta}}^{\infty}\Expe{\ind{x_{1}\geq \eta^{1}\para{Q_{c_{1}}^{1},Q_{c_{1}}^{1}+ t }} \ind{G_{1,2}}Y_{[2, M]} }p_{1}t^{p_{1}-1}\dt\leq cx^{q_{1,2}}\Expe{Y_{[2, M]}}
\end{eqalign}
for $q_{1,2}>p_{1}$. 
\proofparagraph{Step 2: Decomposition}
For the first integral we have to do the decomposition argument as done in \cite{binder2023inverse}.  We fix sequence $a_{\ell}:=\rho\delta_{1} \ell^{\lambda},\ell\geq 0$ for $\rho,\lambda\in (0,1)$. We see that $a_{\ell+1}-a_{\ell}\leq \rho\delta_{1}$. Using the gap event $G_{1,2}$, we consider the decomposition
\begin{eqalign}\label{eq:Q1c1decompo}
D^{1}_{\ell}:=\set{a_{\ell}+Q^{2}(b_{2})+\delta_{2}  \leq Q^{1}(c_{1})\leq a_{\ell+1}+Q^{2}(b_{2})+\delta_{2}}, \ell\geq 0.    
\end{eqalign}
Therefore, we decompose and bound using an infimum
\begin{eqalign}
\ind{x_{1}\geq \eta^{1}\para{Q_{c_{1}}^{1},Q_{c_{1}}^{1}+ t }}
=&\sum_{\ell\geq 0}\ind{x_{1}\geq \eta^{1}\para{Q_{c_{1}}^{1},Q_{c_{1}}^{1}+ t },D^{1}_{\ell}}\\
\leq&\sum_{\ell\geq 0}\ind{x_{1}\geq \inf_{T\in I_{\ell}}\eta^{1}\para{T,T+ t },D^{1}_{\ell}},
\end{eqalign}
where 
\begin{eqalign}
I_{\ell}:=[c_{\ell}+Q^{2}(b_{2})+\delta_{2},c_{\ell+1}+Q^{2}(b_{2})+\delta_{2}]=[c_{\ell},c_{\ell+1}]+Q^{2}(b_{2})+\delta_{2}.    
\end{eqalign}
\proofparagraph{Step 3: Tower property}
Now here for each $\ell$-term we upper bound to remove the gap event and apply the tower property for $\CF\para{[0,Q^{2}(b_{2})]}$ 
\begin{eqalign}
&\Expe{\ind{x_{1}\geq \inf_{T\in I_{\ell}}\eta^{1}\para{T,T+ t },D^{1}_{\ell}} \ind{G_{1,2}}Y_{[2, M]} }\\
\leq &\Expe{\ind{x_{1}\geq \inf_{T\in I_{\ell}}\eta^{1}\para{T,T+ t },D^{1}_{\ell}} Y_{[2, M]} }\\
= &\Expe{\Expe{\ind{x_{1}\geq \inf_{T\in I_{\ell}}\eta^{1}\para{T,T+ t },D^{1}_{\ell}}\mid \CF\para{[0,Q^{2}(b_{2})]}} Y_{[2, M]} },
\end{eqalign}
where we used that $Y_{[2, M]}\in \CF\para{[0,Q^{2}(b_{2})]}$. Then we apply conditional-\Holder to separate the decomposition event
\begin{eqalign}
&\Expe{\ind{x_{1}\geq \inf_{T\in I_{\ell}}\eta^{1}\para{T,T+ t },D^{1}_{\ell}}\mid \CF\para{[0,Q^{2}(b_{2})]}}\\\leq&\para{\Expe{\ind{x_{1}\geq \inf_{T\in I_{\ell}}\eta^{1}\para{T,T+ t }}\conditional\CF\para{[0,Q^{2}(b_{2})]}}}^{1/b_{11}}\\
\cdot &\para{\Expe{\ind{D^{1}_{\ell}} \conditional\CF\para{[0,Q^{2}(b_{2})]}}}^{1/b_{12}},
\end{eqalign}
for $1/b_{11}+1/b_{12}=1$.
\proofparagraph{Step 4: Decoupling}
For the first factor
\begin{eqalign}\label{eq:decoupq1q21}
\para{\Expe{\ind{x_{1}\geq \inf_{T\in I_{\ell}}\eta^{1}\para{T,T+ t }}\conditional\CF\para{[0,Q^{2}(b_{2})]}}}^{1/b_{11}}    
\end{eqalign}
the supremum is over a shifted interval $[a_{\ell},a_{\ell+1}]$, so we study
\begin{eqalign}
\inf_{T\in [a_{\ell},a_{\ell+1}]}\eta^{1}\para{T+Q^{2}(b_{2})+\delta_{2},T+Q^{2}(b_{2})+\delta_{2}+ t }.    
\end{eqalign}
So due to \cref{deltaSMP}, this variable is independent of events in $\CF\para{[0,Q^{2}(b_{2})]}$
\begin{eqalign}\label{eq:decoupq1q22}
\eqref{eq:decoupq1q21}=\para{\Expe{\ind{x_{1}\geq \inf_{T\in [0,\rho\delta_{1}]}\eta^{1}\para{T,T+ t }}}}^{1/b_{11}}. \end{eqalign}
Now that we decoupled, we use \cref{prop:minmodeta} for $q_{1,1}>0$
\begin{eqalign}\label{eq:decoupq1q23}
\eqref{eq:decoupq1q22}\leq x_{1}^{q_{1,1}} t^{\alpha_{2\delta_{1}}(q_{1,1})/b_{11}},
\end{eqalign}
where $\alpha_{2\delta}(p)=\zeta(-p)-1$ if $\frac{x}{2}\leq \delta$ and $\alpha_{2\delta}(p)=-p$ if $\frac{x}{2}\geq \delta$. Returning to the $t-$integral in \cref{eq:multipointmoments1} over $[0,B_{\delta_{1}}]$, we obtain the constraint
\begin{eqalign}
&\frac{\zeta(-q_{1,1})-1}{b_{11}}+p_{1}-1>-1   \doncl p_{1}> \frac{-\zeta(-q_{1,1})+1}{b_{11}}.   
\end{eqalign}
\proofparagraph{Step 5: Event $D^{1}_{\ell}$}
Here we study
\begin{eqalign}\label{eq:eventDl11}
D^{1}_{\ell}=&\set{\eta^{1}\para{0,a_{\ell}+Q^{2}(b_{2})+\delta_{2}} \leq c_{1}   \leq \eta^{1}\para{0,a_{\ell+1}+Q^{2}(b_{2})+\delta_{2}} }\\
=&\set{\eta^{1}\para{Q^{2}(b_{2})+\delta_{2},a_{\ell}+Q^{2}(b_{2})+\delta_{2}} \leq c_{1}-\eta^{1}\para{0,Q^{2}(b_{2})+\delta_{2}}   \leq \eta^{1}\para{Q^{2}(b_{2})+\delta_{2},a_{\ell+1}+Q^{2}(b_{2})+\delta_{2}} }.
\end{eqalign}
We again apply \cref{deltaSMP} to decouple
\begin{eqalign}
\eqref{eq:eventDl11}\eqdis\set{\wt{\eta}^{1}\para{0,a_{\ell}} \leq c_{1}-\eta^{1}\para{0,Q^{2}(b_{2})+\delta_{2}}   \leq \wt{\eta}^{1}\para{0,a_{\ell+1}} },    
\end{eqalign}
where $\wt{\eta}^{1}$ is an independent copy of $\eta^{1}$. To avoid the coupled term $\eta^{1}\para{0,Q^{2}(b_{2})+\delta_{2}}$
, we keep only the lower bound
\begin{eqalign}
\ind{D^{1}_{\ell}}\leq \ind{\wt{\eta}^{1}\para{0,a_{\ell}} \leq c_{1} }    
\end{eqalign}
and so
\begin{eqalign}
\Expe{\ind{D^{1}_{\ell}} \conditional\CF\para{[0,Q^{2}(b_{2})]}}\leq  \Proba{\wt{\eta}^{1}\para{0,a_{\ell}} \leq c_{1} }.   
\end{eqalign}
For $\ell=0$, we bound by zero.  When $c_{\ell}\leq \delta_{1}\tor c_{\ell}> \delta_{1}$ we use Markov with $r_{1}>0\tand r_{2}>0$ respectively,to upper bound by
\begin{eqalign}\label{eq:singularfactorfromscalinglaw}
\sum_{\ell\geq 1 }\para{\Expe{\ind{D^{1}_{\ell}} \conditional\CF\para{[0,Q^{2}(b_{2})]}}}^{1/b_{12}}\leq & \sum_{\ell\geq 1 } \para{\Proba{\wt{\eta}^{1}\para{0,a_{\ell}} \leq c_{1} }}^{1/b_{12}}  \\
\leq &\sum_{\ell\geq 1 }  c_{1}^{r_{1}/b_{12}}\para{a_{\ell}}^{\zeta(-r_{1})/b_{12}}+ c_{1}^{r_{2}/b_{12}}\para{a_{\ell}}^{-r_{2}/b_{12}}.
\end{eqalign}
So to get finiteness we require both $-\zeta(-r_{1})\lambda>b_{12},r_{2}\lambda>b_{12}$, which is possible since GMC has all its negative moments and so we take arbitrarily large $r_{1},r_{2}$. This allows to take $b_{12}$ arbitrarily large too so that we get $b_{11}=1+\e$. This procedure didn't leave any terms that are coupled with the other scales and so we proceed the same for them. 

\end{proof}
 
\section{Moments of multipoint ratios  }\label{conditionalmoments}
In the following proposition we study the first term to get moment estimates as in the single-ratio case in \cref{prop:inverse_ratio_moments}.
\begin{proposition}[Product of ratios]\label{prop:productconditionalratio}
We fix set $S\subset [N]$.
\begin{itemize}
 \item (Equal length)  Fix $\beta<\frac{2}{3}$. Suppose we have equal length intervals $x_{k}=\abs{J_k}=\abs{I_k}<\delta_k$ for each $k\in S$. Fix $\delta_k\leq 1$ and intervals $J=(c_k,c_k+x_k),I=(d_k,d_k+x_k)\subset [0,1]$ with $d_k-c_k=c_{d-c,k}x_k$ for $c_{d-c,k}>1$  and $c_k=\lambda_{1,k}\delta_k,d_k=\lambda_{2,k}\delta_k$. Then we have for all $p_k\in [1,1+\e_{k}]$ with small enough $\e_{k}>0$, a bound of the form
\begin{eqalign}\label{eq:decaynumerdeno}
\Expe{\prod\limits_{k\in S} \para{\frac{Q^{k}(J_{k})}{Q^{k}(I_{k})}}^{p_{k}} \ind{G_{S}}}\leq c^{\abs{S}}\prod\limits_{k\in S} \para{\frac{x_{k} }{\delta_{k}}}^{-\e_{ratio}(p_k)},  
\end{eqalign}
where $\e_{ratio}(p_k)>0$ can be made arbitrarily small at the cost of a larger proportionality constant $c$. The constants are uniformly bounded in $x_k$ and also in $\delta_k$.

    \item (Decreasing numerator) Fix $\beta\in \para{0,0.152}.$
    Suppose that $\abs{I_{k}}=r_{k}\delta_{k}$ for $r_{k}>0$ and $\abs{J_{k}}\to 0$ for each $k\in S$. Then we have for each $p_{k}\in [1,1+\e_{ratio}]$ some  $q_{k}>1$ such that
\begin{eqalign}\label{eq:decaynumerator}
\Expe{\prod\limits_{k\in S} \para{\frac{Q^{k}(J_{k})}{Q^{k}(I_{k})}}^{p_{k}} \ind{G_{S}}}\leq \prod\limits_{k\in S} \para{\frac{\abs{J_{k}}}{\delta_{k}}}^{q_{k}}.
\end{eqalign}

\end{itemize}

\end{proposition}
\begin{proof}
\proofparagraph{Equal length intervals}
To ease on the notation we simply write $1,....,M$ as opposed to $i_{1},...,i_{M}$ since for now the argument only requires them to have gaps and it is not number-specific. The strategy is to peel-off one ratio at a time by successively decomposing the $\para{Q^1(a_{1}),Q^{2}(b_{2})},\para{Q^{2}(a_{2}),Q^{3}(b_{3})},...$. We will follow some of the steps from the proof of \cref{prop:inverse_ratio_moments}. Here we only cover the singular case involving the scaling law and the other cases are handled similarly. 
\proofparagraph{Step 1: Decomposition}
We start with the largest scale $\para{\frac{Q^{1}(J_{1})}{Q^{1}(I_1)}}^{p_{1}}$ for $J_{1}:=[c_{1},c_{1}+x_{1}],I_{1}:=[d_{1},d_{1}+y_{1}]$ and denote the product of other scales as 
\begin{equation}
Y_{[2, M]}:= \prod\limits_{k\in [2,M]} \para{\frac{Q^{k}(J_{k})}{Q^{k}(I_{k})}}^{p_{k}}   \ind{ G_{[3,M]}}    
\end{equation}
We decouple this first ratio using the gap $G_{1,2}$. We only cover the details of the particular case $c_{1}< d_{1}$ and $x_{1}=d_{1}-c_{1}$ and 
\begin{equation}
Q^{1}(J_{1})\in [0,\rho] \tand Q^{1}(I_{1})\in [0,1]    
\end{equation}
and so we study
\begin{eqalign}\label{eq:multipointratiossingularcase}
\Expe{\para{\frac{Q^{1}(J_{1})}{Q^{1}(I_1)}}^{p_{1}}\ind{Q^{1}(J_{1})\in [0,\rho] \tand Q^{1}(I_{1})\in [0,1]}\ind{G_{1,2}}Y_{[2, M]}}.
\end{eqalign}
This corresponds to the singular case and the other cases are easier to handle, for details see the proof of \cref{prop:inverse_ratio_moments}. As done there we first decompose the numerator 
\begin{equation}
D^{x_{1}}_{k}:=\set{ Q^{1}(J_{1})\in [r_{k+1},r_{k}]},  \end{equation}
for the same decaying sequence $r_{k}:=\rho k^{-\lambda}$ for $k\geq 1$ and $\lambda>0$. We then apply layercake to upper bound by
\begin{eqalign}\label{eq:multipointratiossingularcase2}
\eqref{eq:multipointratiossingularcase}\leq &\Expe{\ind{G_{1,2}}Y_{[2, M]}}\\
&+\sum_{k}\int_{[1,\infty)}\Expe{\ind{\eta^{1}\para{Q_{d_{1}}^{1},Q_{d_{1}}^{1}+ r_{k}t^{-1/p} }\geq y_{1}, D^{x_{1}}_{k}} \ind{G_{1,2}}Y_{[2, M]} }\dt.
\end{eqalign}
Next for each k-term we decompose $Q^{1}(c_{1})$.  We set $c_{\ell,k}:=\rho_{a}\delta_{1}r_{k}c_{\ell}$ for some small enough $\rho_{a}\in (0,1)$ and $c_{\ell}:=\ell^{\lambda_{a}}$ for $\lambda_{a}\in (0,1)$ so that indeed $c_{\ell+1}-c_{\ell}\to 0$ as $\ell\to +\infty$. So we consider
\begin{eqalign}\label{eq:Q1c1}
D^{c_{1}}_{\ell,k}:=\set{  Q^{1}(c_{1})\in [c_{\ell,k},c_{\ell+1,k}]+Q^{2}(b_{2})+\delta_{2}=:I_{\ell,k}^{c_{1}}+Q^{2}(b_{2})+\delta_{2}} ,    
\end{eqalign}
where we included the information from the gap event $G_{1,2}$ and so for the probability term in \cref{eq:multipointratiossingularcase2} we write
\begin{eqalign}
&\Expe{\ind{\eta^{1}\para{Q_{d_{1}}^{1},Q_{d_{1}}^{1}+ r_{k}t^{-1/p} }\geq y_{1}, D^{x_{1}}_{k}} \ind{G_{1,2}}Y_{[2, M]} }\\
=&\sum_{\ell}\Expe{\ind{\eta^{1}\para{Q_{d_{1}}^{1},Q_{d_{1}}^{1}+ r_{k}t^{-1/p} }\geq y_{1}, D^{x_{1}}_{k} ,D^{c_{1}}_{\ell,k}} \ind{G_{1,2}}Y_{[2, M]} }\\
\leq&\sum_{\ell}\Expe{\ind{\eta^{1}\para{Q_{d_{1}}^{1},Q_{d_{1}}^{1}+ r_{k}t^{-1/p} }\geq y_{1}, D^{x_{1}}_{k} ,D^{c_{1}}_{\ell,k}} Y_{[2, M]} },
\end{eqalign}
where in the last step we upper bounded to remove the gap event $\ind{G_{1,2}}$ now that we used it. 
\proofparagraph{Step 2: Tower property}
Now here for each $\ell$-term we apply the tower property for $\CF\para{[0,Q^{2}(b_{2})]}$ 
\begin{eqalign}
&\Expe{\ind{\eta^{1}\para{Q_{d_{1}}^{1},Q_{d_{1}}^{1}+ r_{k}t^{-1/p} }\geq y_{1}, D^{x_{1}}_{k} ,D^{c_{1}}_{\ell,k}} Y_{[2, M]} }\\
=&\Expe{\Expe{\ind{\eta^{1}\para{Q_{d_{1}}^{1},Q_{d_{1}}^{1}+ r_{k}t^{-1/p} }\geq y_{1}, D^{x_{1}}_{k} ,D^{c_{1}}_{\ell,k}}\conditional\CF\para{[0,Q^{2}(b_{2})]}} Y_{[2, M]} }.
\end{eqalign}
Next we proceed as in \cref{prop:inverse_ratio_moments} inside the conditional expectation, we take supremum over $u\in I_{\ell,k}^{c_{1}}+Q^{2}(b_{2})+\delta_{2}$
\begin{eqalign}\label{eqcondifilt1}
\Exp\Bigg[&\maxls{u\in I_{\ell,k}^{c_{1}}+Q^{2}(b_{2})+\delta_{2}}\para{\ind{\eta^{1}\para{u+Q_{d_{1}-c_{1}}^{1}\bullet u,u+Q_{d_{1}-c_{1}}^{1}\bullet u+ r_{k}t^{-1/p} }\geq y_{1}, D^{x_{1}}_{k}(u) }}\\
&\cdot \ind{D^{c_{1}}_{\ell,k}} \conditional\CF\para{[0,Q^{2}(b_{2})]}\Bigg],
\end{eqalign}
where 
\begin{equation}
D^{x_{1}}_{k}(u):= \set{ Q^{1}_{x_{1}}\bullet u\in [r_{k+1},r_{k}]}.   
\end{equation}
So now that we used the event $D^{c_{1}}_{\ell,k}$, as in the proof of \cref{prop:inverse_ratio_moments}  we apply  conditional-\Holder to remove it out
\begin{eqalign}
&\eqref{eqcondifilt1}\leq\\
&\Bigg(\Expe{\maxls{u\in I_{\ell,k}^{c_{1}}+Q^{2}(b_{2})+\delta_{2}}\para{\ind{\eta^{1}\para{u+Q_{d_{1}-c_{1}}^{1}\bullet u,u+Q_{d_{1}-c_{1}}^{1}\bullet u+ r_{k}t^{-1/p} }\geq y_{1},D^{x_{1}}_{k}(u) }}\conditional\CF\para{[0,Q^{2}(b_{2})]}}\Bigg)^{1/b_{11}}\\
\cdot &\para{\Expe{\ind{D^{c_{1}}_{\ell,k}} \conditional\CF\para{[0,Q^{2}(b_{2})]}}}^{1/b_{12}},
\end{eqalign}
for $b_{11}^{-1}+b_{12}^{-1}=1$.
\proofparagraph{Step 3: Decoupling}
For the first factor, we simply use \cref{deltaSMP} to decouple
\begin{eqalign}\label{eq:smpquantity}
&\Expe{\maxls{u\in I_{\ell,k}^{c_{1}}+Q^{2}(b_{2})+\delta_{2}}\para{\ind{\eta^{1}\para{u+Q_{d_{1}-c_{1}}^{1}\bullet u,u+Q_{d_{1}-c_{1}}^{1}\bullet u+ r_{k}t^{-1/p} }\geq y_{1},D^{x_{1}}_{k}(u) }}\conditional\CF\para{[0,Q^{2}(b_{2})]}}\\
=&\Expe{\maxls{u\in I_{\ell,k}^{c_{1}}}\para{\ind{\eta^{1}\para{u+Q_{d_{1}-c_{1}}^{1}\bullet u,u+Q_{d_{1}-c_{1}}^{1}\bullet u+ r_{k}t^{-1/p} }\geq y_{1},D^{x_{1}}_{k}(u) } }}.
\end{eqalign}
Now that we decoupled this ratio, we proceed as in the proof of the unconditional version \cref{prop:inverse_ratio_moments}. Next we study the event $D^{c_{1}}_{\ell,k}$.
\proofparagraph{Step 4: The event $D^{c_{1}}_{\ell,k}$}
Here we study
\begin{eqalign}\label{eq:eventclk1}
&D^{c_{1}}_{\ell,k}\\
=&\set{\eta^{1}\para{0,c_{\ell,k}+Q^{2}(b_{2})+\delta_{2}} \leq c_{1}   \leq \eta^{1}\para{0,c_{\ell+1,k}+Q^{2}(b_{2})+\delta_{2}} }\\
=&\set{\eta^{1}\para{Q^{2}(b_{2})+\delta_{2},c_{\ell,k}+Q^{2}(b_{2})+\delta_{2}} \leq c_{1}-\eta^{1}\para{0,Q^{2}(b_{2})+\delta_{2}}   \leq \eta^{1}\para{Q^{2}(b_{2})+\delta_{2},c_{\ell+1,k}+Q^{2}(b_{2})+\delta_{2}} }.
\end{eqalign}
We again apply \cref{deltaSMP} to decouple
\begin{eqalign}
\eqref{eq:eventclk1}\eqdis\set{\wt{\eta}^{1}\para{0,c_{\ell,k}} \leq c_{1}-\eta^{1}\para{0,Q^{2}(b_{2})+\delta_{2}}   \leq \wt{\eta}^{1}\para{0,c_{\ell+1,k}} },    
\end{eqalign}
where $\wt{\eta}^{1}$ is an independent copy of $\eta^{1}$. To avoid the coupled term $\eta^{1}\para{0,Q^{2}(b_{2})+\delta_{2}}$, we keep only the lower bound
\begin{eqalign}
\Expe{\ind{D^{c_{1}}_{\ell,k}} \conditional\CF\para{[0,Q^{2}(b_{2})]}}\leq  \Proba{\wt{\eta}^{1}\para{0,c_{\ell,k}} \leq c_{1} }.   
\end{eqalign}
Now that we decoupled we proceed as in proof of \cref{prop:inverse_ratio_moments}. For $\ell=0$, we bound by zero.  When $c_{\ell,k}\leq \delta_{i_{1}}\tor c_{\ell,k}> \delta_{i_{1}}$ we use Markov with $p_{1}>0\tand p_{2}>0$ respectively,to upper bound by
\begin{eqalign}\label{eq:singularfactorfromscalinglawratio}
&\sum_{\ell\geq 1 }\para{\Expe{\ind{D^{c_{1}}_{\ell,k}} \conditional\CF\para{[0,Q^{2}(b_{2})]}}}^{1/b_{12}}\\
\leq & \sum_{\ell\geq 1 } \para{\Proba{\wt{\eta}^{1}\para{0,c_{\ell,k}} \leq c_{1} }}^{1/b_{12}}  \\
\leq &\sum_{\ell\geq 1 }  c_{1}^{p_{1}/b_{12}}\para{c_{\ell,k}}^{\zeta(-p_{1})/b_{12}}+ c_{1}^{p_{2}/b_{12}}\para{c_{\ell,k}}^{-p_{2}/b_{12}}.
\end{eqalign}
So from here to get finiteness we require both $-\zeta(-p_{1})\lambda_{a}>b_{12},p_{2}\lambda_{a}>b_{12}$. Therefore, we get the singular power $r_{k}^{-\frac{1+\e}{\lambda_{a}}}$ for arbitrarily small $\e>0$. Here we are completely back to the unconditional setting of \cref{prop:inverse_ratio_moments}. This procedure didn't leave any terms that are coupled with the other scales and so we proceed the same for them. So we bootstrap the same for the next scales. 
\proofparagraph{Decreasing numerator case}
We assume $c_{1}<d_{1}$ because the other case is symmetric. Here we use the semigroup formula to write
\begin{equation}
 \frac{Q^{1}(J_{1})}{Q^{1}(I_1)}=\frac{Q_{x_{1}}^{1}\bullet Q_{c_{1}}^{1}}{Q^{1}_{y_{1}}\bullet \para{Q_{d_{1}-c_{1}}^{1}\bullet Q_{c_{1}}^{1}+Q_{c_{1}}^{1}}}.   
\end{equation}
Here the decay only comes from the numerator. So there is no need for scaling law and so in the decomposition event \cref{eq:Q1c1}, we only index over $\ell\geq 0$ without any presence of $r_{k}$ 
\begin{eqalign}
D^{c}_{\ell}:=\set{  Q^{1}(c_{1})\in [\ell^{\lambda},(\ell+1)^{\lambda}]+Q^{2}(b_{2})+\delta_{2}}.    
\end{eqalign}
That means there won't be any singular factor $r_{k}^{-1/\lambda}$. As above we take the maximum and use the conditional-\Holder  $b_{11}^{-1}+b_{12}^{-1}=1$ to separate out the event $D^{c}_{\ell}$
\begin{eqalign}
&\sum_{\ell}\Expe{\para{\frac{Q^{1}(J_{1})}{Q^{1}(I_1)}}^{p_{1}}\ind{D^{c_{1}}_{\ell}}\conditional\CF\para{[0,Q^{2}(b_{2})]}}\\
\leq &\sum_{\ell}\para{ \Expe{\max_{u\in [0,2]}\para{\frac{Q_{x_{1}}^{1}\bullet u}{Q^{1}_{y_{1}}\bullet \para{Q_{d_{1}-c_{1}}^{1}\bullet u+u}}}^{b_{11}p}}}^{1/b_{11}}\cdot \para{\Expe{\ind{D^{c_{1}}_{\ell}} \conditional\CF\para{[0,Q^{2}(b_{2})]}}}^{1/b_{12}}.
\end{eqalign}
We deal with the second term as we dealt with $D^{c_{1}}_{\ell,k}$in the proof of the first part. So we are left with studying
\begin{equation}\label{eq:maximumratiodecnum}
\para{ \Expe{\max_{u\in [0,2]}\para{\frac{Q_{x_{1}}^{1}\bullet u}{Q^{1}_{y_{1}}\bullet \para{Q_{d_{1}-c_{1}}^{1}\bullet u+u}}}^{b_{11}p}}}^{1/b_{11}}.   
\end{equation}
The denominator is just a constant since we already scaled by $\delta_{i_{1}}$ and we can use the \Cref{lem:maxdshiftedGMC}. So we simply apply \Holder again to separate numerator and denominator
\begin{equation}\label{eq:holderseparatnumde}
 \para{\Expe{\max_{u\in [0,2]}\para{Q_{x_{1}}^{1}\bullet u}^{b_{11}pp_{1}}}}^{1/b_{11}p_{1}}\para{\Expe{\max_{u\in [0,2]}\para{\frac{1}{Q^{1}_{y_{1}}\bullet \para{Q_{d_{1}-c_{1}}^{1}\bullet u+u}}}^{b_{11}pp_{2}}}}^{1/b_{11}p_{2}},  
\end{equation}
for $p_{1}^{-1}+p_{2}^{-1}=1$. For the first factor we apply layercake, Markov for for $q_{1},q_{2}>0$ and then \Cref{lem:minshiftedGMC}
\begin{eqalign}\label{eq:integralconstraintsnumersmall}
\Expe{\max_{u\in [0,2]}\para{Q_{x_{1}}^{1}\bullet u}^{pb_{11}p_{1}}}&=\int_{0}^{\infty}\Proba{x_{1}\geq \max_{u\in [0,2]}\eta\spara{u,u+t^{1/pb_{11}p_{1}}}}\dt\\
\leq&x_{1}^{q_{1}}\int_{0}^{1}t^{\frac{\zeta(-q_{1})-1}{b_{11}pp_{1}}} \dt+x_{1}^{q_{2}}\int_{1}^{\infty}t^{\frac{-q_{2}}{b_{11}pp_{1}}}\dt.
\end{eqalign}
For finiteness we require
\begin{eqalign}
 \frac{-\zeta(-q_{1})+1}{b_{11}pp_{1}}<1\tand  \frac{q_{2}}{b_{11}pp_{1}}>1.  
\end{eqalign}
Next we get the exponents for $x$ to be $\frac{q_{1}}{p_{1}}>1+\e$ and $\frac{q_{2}}{p_{1}}>1+\e$ for small $\e>0$. The second inequality is cleared by taking large $q_{2}>0$. For the first inequality we have that an arbitrarily large $q_{1}$ will break the first integral constraint. In the unconditional work we took $b_{11}=1+\e$ and we do the same here. So for the exponent of we ask $\frac{q_{1}}{p_{1}}=1+\e$, which translates the integral constraint to
\begin{equation}\label{eq:integralconstraint}
1<p_{1}\tand \frac{(1+\e)(1+\beta(p_{1}(1+\e)+1))+\frac{1}{p_{1}}}{p}<1.
\end{equation}
Since $pp_{2}\in [1,\beta^{-1})$, we further take $p=\beta^{-1}\frac{p_{1}-1}{p_{1}}-o(\e)$, which makes the above into a single variable problem
\begin{equation}
1<p_{1}\tand \frac{(1+\e)(1+\beta(p_{1}(1+\e)+1))+\frac{1}{p_{1}}}{\beta^{-1}\frac{p_{1}-1}{p_{1}}}<1.
\end{equation}
This has a solution $p_{1}>1$ when we further require 
\begin{equation}
    \beta\in \para{0,0.152}.
\end{equation}
If we have $c_{1}>d_{1}$, we would have to use \Cref{lem:maxdshiftedGMC} instead of \Cref{prop:minmodeta} and it has the same exponent and the same constraint for $\beta$.
\end{proof}
\newpage \section{Multipoint and unit circle}\label{sec:multipointmaximum}
 This is the relevant part for the Lehto-welding. Namely we study the multipoint estimate for the inverse $Q_{H}$ on the unit circle. In this section we will use the following notations. We consider the events
\begin{eqalign}
G_{[1,M]}:=&\set{Q^{i}(a_{i})\geq Q^{n_{i+1}}(b_{n_{i+1}})+\delta_{n_{i+1}}, i\in [1,M]},\\
\Xi_{[1,M]}:=&\set{\supl{(s,t)\in [0,Q^{i}(b_{i})]^{2}}\abs{\xi(t)-\xi(s)}\leq u_{i,\xi} ,i\in [1,M]},\\
U_{[1,M]}:=&\set{\supl{(s,t)\in [0,Q^{i}(b_{i})]^{2}}\abs{U^{1}_{i}(t)-U^{1}_{i}(s)}\leq u_{i,us}, i\in [1,M]},
\end{eqalign}
where we use the shorthand notation $Q^{i}=Q^{n_{i}}$ and consider any sequences $u_{i,\xi},u_{i,us}\in (0,1)$ strictly decreasing to zero as $i\to+\infty$ and interval
\begin{eqalign}
P_{i}:=[-\para{u_{i,\xi}+u_{i,s}},u_{i,\xi}+u_{i,s}].   \end{eqalign}
We will bound the following product
\begin{eqalign}\label{eq:maineventmultiunitcircle}
&\prod\limits_{k\in [1,M]}\para{\frac{Q_{H}(d_{k},d_{k}+y_{k})}{Q_{H}(c_{k},c_{k}+x_{k})}}^{p_{k}}   \ind{ G_{[1,M]} ,\Xi_{[1,M]}, U_{[1,M]}},
\end{eqalign}
where for set $S:=\set{n_{1},...,n_{M}}\subset [1,N]$. In our ensuing work on Lehto-welding we obtained the existence of this set $S$ for these events using deviation estimates. 
\begin{proposition}[Product of ratios]\label{prop:Multipointunitcircleandmaximum}
We fix set $S:=\set{n_{1},...,n_{M}}\subset [N]$.
\begin{itemize}
 \item (Equal length)  Fix $\gamma<\frac{2}{\sqrt{3}}$. Suppose we have equal length intervals $x_{k}=\abs{J_k}=\abs{I_k}\leq \delta_k$ for each $k\in S$. Fix $\delta_k\leq 1$ and intervals $J=(c_k,c_k+x_k),I=(d_k,d_k+x_k)\subset [0,1]$ with $d_k-c_k=c_{d-c,k}x_k$ or $c_k-d_{k}=c_{d-c,k}x_k$ for some $c_{d-c,k}>1$  and $c_k=\lambda_{1,k}\delta_k,d_k=\lambda_{2,k}\delta_k$ for some $\lambda_{i,k}>0$.  Then we have for all $p_k\in [1,1+\e_{k}]$ with small enough $\e_{k}>0$, a bound of the form
\begin{eqalign}\label{eq:decaynumerdenounic}
\Expe{\prod\limits_{k\in [1,M]}\para{\frac{Q_{H}(d_{k},d_{k}+y_{k})}{Q_{H}(c_{k},c_{k}+x_{k})}}^{p_{k}}   \ind{ G_{[1,M]} ,\Xi_{[1,M]}, U_{[1,M]}}}\leq c^{\abs{S}}\prod\limits_{k\in S} \para{\frac{x_{k} }{\delta_{k}}}^{-\e_{ratio}(p_k)},  
\end{eqalign}
where $\e_{ratio}(p_k)>0$ can be made arbitrarily small at the cost of a larger proportionality constant $c$. The constants are uniformly bounded in $x_k$ and also in $\delta_k$.

    \item (Decreasing numerator) Fix $\beta\in  (0,0.152).$ Suppose that $\abs{I_{k}}=r_{k}\delta_{k}$ for $r_{k}>0$ and $\abs{J_{k}}\to 0$ for each $k\in S$. Then we have for each $p_{k}\in [1,1+\e_{ratio,1}]$ some  $q_{k}=1+\e_{ratio,2}>1$ such that
\begin{eqalign}\label{eq:decaynumeratorunic}
\Expe{\prod\limits_{k\in [1,M]}\para{\frac{Q_{H}(d_{k},d_{k}+y_{k})}{Q_{H}(c_{k},c_{k}+x_{k})}}^{p_{k}}   \ind{ G_{[1,M]} ,\Xi_{[1,M]}, U_{[1,M]}}}\leq \prod\limits_{k\in S} \para{\frac{\abs{J_{k}}}{\delta_{k}}}^{q_{k}}.
\end{eqalign}

\end{itemize}

\end{proposition}

\subsection{Proof of \sectm{\cref{prop:Multipointunitcircleandmaximum}}: Equal lengths}
In order to achieve any decoupling the $\xi$ field and upper scales need to be removed. If we take a maximum over their domains $P_{k}$,
\begin{eqalign}
\prod_{k\in [1,M]}\para{\frac{\maxls{u_{1},u_{2}\in P_{k}}Q^{k}\para{e^{u_{1} }\spara{c_{k},c_{k}+x_{k}e^{u_{2}}}}}{\minls{u_{3},u_{4}\in P_{k}}Q^{k}\para{e^{u_{3}}\spara{d_{k},d_{k}+y_{k}e^{u_{4}}}}}}^{p_{k}}  
\end{eqalign}
then we can no longer use the trick of \cref{prop:inverse_ratio_moments} where we divided and multiplied by $Q_{b-a}\bullet Q_{a}$ and applied scaling law with respect to its deterministic approximation because now the denominator has "multiple" transitions $Q_{e^{u_{3}}(d_{k}-c_{k})}\bullet Q_{e^{u_{3}}c_{k}}$ for each $u_{3}$ and so it's no longer clear which factor to decompose and apply the scaling law on.\\
So instead we move back to the level of $Q_{H}$. Let $C_{1},C_{2}$ contain all the $k$ such that $d_{k}\geq c_{k}$ and $d_{k}<c_{k}$ respectively. We divide and multiply by the intermediate-intervals 
\begin{eqalign}
M_{k}:=\branchmat{Q_{H}(c_{k},c_{k}+d_{k}-c_{k}) &\tcwhen k\in C_{1}\\ Q_{H}(c_{k},d_{k}+y_{k}) &\tcwhen k\in C_{2}}.   
\end{eqalign}
For the $C_{1}$-group we use that $x_{k}<d_{k}-c_{k}$ to bound the first ratios by one:
\begin{eqalign}
&\prod\limits_{k\in C_1} \para{\frac{Q_{H}(c_{k},c_{k}+x_{k})}{Q_{H}(d_{k},d_{k}+y_{k})}}^{p_{k}}\leq \prod\limits_{k\in C_{1}}\para{\frac{Q_{H}(c_{k},c_{k}+d_{k}-c_{k})}{Q_{H}(d_{k},d_{k}+y_{k})}}^{p_{k}}
\end{eqalign}
and for the $C_{2}-$group we use $Q_{H}(d_{k},d_{k}+y_{k})\leq Q_{H}(c_{k},d_{k}+y_{k})$ to bound
\begin{equation}
\prod\limits_{k\in C_{2}} \para{\frac{Q_{H}(d_{k},d_{k}+y_{k})}{Q_{H}(c_{k},c_{k}+x_{k})}}^{p_{k}}    \leq \prod\limits_{k\in C_{2}} \para{\frac{Q_{H}(c_{k},d_{k}+y_{k})}{Q_{H}(c_{k},c_{k}+x_{k})}}^{p_{k}}.    
\end{equation}
The strategy will be to first transition to $Q_{U}$, then decouple to reduce single-point estimates and then study the first one and bootstrap for the rest.
\proofparagraph{Decoupling part}
\proofparagraph{Step 1: Decomposing the intermediate intervals for all the scales}
As we did before in the single-point case \cref{prop:inverse_ratio_moments}, we decompose the intermediate interval $Q_{H}(c_{k},c_{k}+d_{k}-c_{k})$.\\
So we consider the following decompositions for $k\in [1,M]$
\begin{eqalign}\label{eq:decompositionevents}
D^{k}_{\ell_{k},M}:=&\branchmat{\set{M_{k}=Q_{H}(c_{k},c_{k}+d_{k}-c_{k})\in [a_{\ell_{k}},a_{\ell_{k}+1}]=:I_{\ell_{k}}^{k}}& \tcwhen k\in C_{1} \\ \set{M_{k}=Q_{H}(c_{k},d_{k}+y_{k})\in [a_{\ell_{k}},a_{\ell_{k}+1}]=:I_{\ell_{k}}^{k}}& \tcwhen k\in C_{2} }    \tfor \ell_{k}\in \Z,\\
S^{k}_{m_{k},c}(\ell_{k}):=&\set{Q_{H}(c_{k})\in [b_{m_{k}}(\ell_{k}),b_{m_{k}+1}(\ell_{k})]=:I_{m_{k},c}(\ell_{k})}     \tfor m_{k}\in \Z,
\end{eqalign}
where the $a_{\ell_{k}},b_{m_{k}}(\ell_{k})$ are bi-directional sequences \textit{possibly} going to zero and infinity; and the notation $b_{m_{k}}(\ell_{k})$ means that this sequence will also depend on $\ell_{k}$ (as required in the singular case of the proof in \cref{prop:inverse_ratio_moments}). In particular we let
 \begin{eqalign}
a_{k}:=\branchmat{\rho_{b-a}\abs{k}^{-\lambda_{neg}} &k\leq -1\\
\rho_{b-a}k^{\lambda_{pos}}& k\geq 1}
\end{eqalign}
for $\lambda_{neg}>0$ and $\lambda_{pos}\in (0,1)$, and we will see $b_{m_{k}}(\ell_{k})$ in \eqref{eq:usinggapQhQu3}.\\
In the next paragraph, we will also incorporate the gap-events in their definition so as to achieve decoupling. We will also split them over $\rho_{b-a}$ to get the various cases as in \cref{prop:inverse_ratio_moments}. So we write
\begin{eqalign}\label{eq:maineventmultiunitcirclepreremovalxiu}
\eqref{eq:maineventmultiunitcircle}=\sumls{\ell_{i},m_{i}\\ i\in [1,M]}
&\prod\limits_{k\in C_{1}} \ind{D^{k}_{\ell_{k},M},S^{k}_{m_{k},c}(\ell_{k})}\para{\frac{M_{k}}{B_{k}}}^{p_{k}}\\
&\cdot\prod\limits_{k\in C_{2}}\ind{D^{k}_{\ell_{k},M},S^{k}_{m_{k},c}(\ell_{k})}\para{\frac{M_{k}}{A_{k}}}^{p_{k}}\\&\cdot \ind{ G_{[1,M]},\Xi_{[1,M]}, U_{[1,M]}}.
\end{eqalign}
Next we study the relation to $Q_{U}^{k}$.
\proofparagraph{Step 2: Relation with $Q_{U}^{k}$}
 We first transition to $Q=Q_{U}$ over general intervals $(z_{n,M},z_{n,M}+r_{n,M})$ for 
\begin{eqalign}
z_{n,M}:=z_{n}\expo{\xi(b_{n})+\thickbar{U}^{1}_{n}(\theta_{b_{n}})} \tand r_{n,M}:=  r_{n}\expo{\xi(b_{n})+\thickbar{U}^{1}_{n}(\theta_{b_{n}})}, 
\end{eqalign}
for  $\theta_{b_{n}}\in [0,Q^{n}(b_{n})].$  We write
\begin{eqalign}
&Q_{H}(z_{n,M},z_{n,M}+r_{n,M})\\
=& Q\para{\frac{e^{\xi(b_{n})}}{e^{\xi(z_{n,M})}}e^{\thickbar{U}^{1}_{n}(\theta_{b_{n}})}\cdot\spara{p_{n},p_{n}+r_{n}\expo{\xi(z_{n,M})-\xi(z_{n,M},z_{n,M}+r_{n,M})}}}\\
=:&Q\para{\spara{A,A+B}},\\
\tfor &A:=\frac{e^{\xi(b_{n})}}{e^{\xi(z_{n,M})}}e^{\thickbar{U}^{1}_{n}(\theta_{b_{n}})}p_{n},\\
\tand &B:=\frac{e^{\xi(b_{n})}}{e^{\xi(z_{n,M},z_{n,M}+r_{n,M})}}e^{\thickbar{U}^{1}_{n}(\theta_{b_{n}})}r_{n}.
\end{eqalign}
Then we transition to $Q^{n}$
\begin{eqalign}
&Q\para{\spara{A,A+B}}=Q^{n}\para{\frac{1}{e^{\bar{U}_{n}^{1}(\theta_A)}}\spara{A,A+B\frac{e^{\bar{U}_{n}^{1}(\theta_A)}}{e^{\bar{U}_{n}^{1}(\theta_{A,A+B})}}}},
\end{eqalign}
where
\begin{eqalign}
&\theta_{A}\in \spara{0,Q\para{A}}=\spara{0,Q\para{\frac{e^{\xi(b_{n})}}{e^{\xi(z_{n,M})}}e^{\thickbar{U}^{1}_{n}(\theta_{b_{n}})}p_{n}}},\\    
&\theta_{A,A+B}\in \spara{Q\para{A},Q\para{A+B}}=\spara{Q(A),Q\para{A+\frac{e^{\xi(b_{n})}}{e^{\xi(z_{n,M},z_{n,M}+r_{n,M})}}e^{\thickbar{U}^{1}_{n}(\theta_{b_{n}})}r_{n}}}.    
\end{eqalign}
We next insert the gap information into the event $D^{k}_{\ell_{k},M}$. By using the events $\Xi_{[1,M]}, U_{[1,M]}$, we identify the supersets for the mean-values
\begin{eqalign}
Q\para{A+B}=&Q\para{e^{\thickbar{U}^{1}_{n}(\theta_{b_{n}})}\para{\frac{e^{\xi(b_{n})}}{e^{\xi(z_{n,M})}}p_{n}+\frac{e^{\xi(b_{n})}}{e^{\xi(z_{n,M},z_{n,M}+r_{n,M})}}r_{n}}}\leq Q\para{e^{\thickbar{U}^{1}_{n}(\theta_{b_{n}})}b_{n}}
= Q^{n}(b_{n}),    
\end{eqalign}
and so $\theta_{A},\theta_{A+B}\in [0,Q^{n}(b_{n})]$. Whereas for lower bound we have
\begin{eqalign}
&Q^{n}\para{\frac{e^{\thickbar{U}^{1}_{n}(\theta_{b_{n}})}}{e^{\bar{U}_{n}^{1}(\theta_A)}}\frac{e^{\xi(b_{n})}}{e^{\xi(z_{n,M})}}p_{n}}\geq Q^{n}(a_{n})    
\end{eqalign}
and so we have the interval inclusion
\begin{eqalign}\label{eq:pertubationinclusion}
\spara{Q_{U}^{n}(e^{\bar{U}_{n}^{1}(\theta_A)}A),Q_{U}^{n}(e^{\bar{U}_{n}^{1}(\theta_A)}(A+B)}\subseteq [Q^{n}(a_{n}) ,Q^{n}(b_{n}) ].   
\end{eqalign}
\proofparagraph{Step 3: Using the Gap event}
 It is convenient to rewrite the decomposition \cref{eq:decompositionevents}
\begin{eqalign}\label{eq:usinggapQhQu1}
 Q_{H}(c_{n})\in [b_{m_{n}}(\ell_{n}),b_{m_{n}+1}^{n}(\ell_{n})]=I_{m_{n},a}^{n}(\ell_{n}).
\end{eqalign}
We have $Q_{H}(c_{n})\geq Q^{n}(a_{n})$ and so from the gap event 
\begin{eqalign}\label{eq:usinggapQhQu2}
G_{n,n+1}:=\set{Q^{n}(a_{n})>Q^{n+1}(b_{n+1})+\delta_{n+1}},    
\end{eqalign}
we have $Q_{H}(c_{n})>Q^{n+1}(b_{n+1})+\delta_{n+1}$. That means we rewrite the decomposition using the sequence
\begin{equation}\label{eq:usinggapQhQu3}
b_{m_{n}}(\ell_{n}):=\wt{b}_{m_{n}}(\ell_{n})+Q^{n+1}(b_{n+1})+\delta_{n+1}    
\end{equation}
for adeterministic bi-infinite sequence $\wt{b}_{m_{n}}(\ell_{n})$ going both to zero and plus infinity. In particular, we let 
\begin{eqalign}\label{eq:divergingsequencebtilde}
\wt{b}_{m}(\ell):=\rho_{a}a_{m}\ell^{\lambda_{a}},    
\end{eqalign}
for some small enough $\rho_{a}\in (0,1)$ and $\lambda_{a}\in (0,1)$. We further denote their intervals as $\wt{I}_{m_{n},a}^{n}(\ell_{n}):=[\wt{b}_{m_{n}}(\ell_{n}),\wt{b}_{m_{n}+1}^{n}(\ell_{n})]$.
\proofparagraph{Removing the $\Xi_{[1,M]}, U_{[1,M]}$ via maximum}
Now we are finally in the position to remove the $\xi$ and upper-scale $U^{1}_{n}$ fields. All the quantities in \cref{eq:maineventmultiunitcirclepreremovalxiu} that involve $Q_{H}$ need to get switched to $Q^{k}_{U}$ in order for the decoupling argument to go through. \\
We start from the outside i.e. we temporarily leave $Q_{H}(c_{k}),M_{k}=Q_{H}(c_{k},c_{k}+d_{k}-c_{k})$ untouched. We also return to the scaled notation 
\begin{equation}
x_{k,U,\xi}=x_{k}e^{\thickbar{U}^{1}_{k}(\theta_{b_{k}})}e^{\xi(b_{k})}.    
\end{equation}
\proofparagraph{Step 4: Switching the $Q_{H}(x_{k})$ to $Q_{U}^{k}$}
For the term $A_{k}=Q_{H}(x_{k})\bullet Q_{H}(c_{k})=Q_{H}(x_{k})\bullet T$, with $T:= Q_{H}(c_{k})$, using the transition formulas in \cref{sec:transitionformuals} we bound
\begin{eqalign}
\min_{u_{1},u_{2}\in I_{T}} Q_{U}^{k}\para{\frac{e^{\thickbar{U}^{1}_{k}(\theta_{b_{k}})}}{e^{\bar{U}_{k}^{1}(u_{1})}}\frac{e^{\xi(b_{k})}}{e^{ \xi(u_{2})}}x_{k}}\bullet T  \leq Q_{H}(x_{k,U,\xi})\bullet T\leq \max_{u_{1},u_{2}\in I_{T}} Q_{U}^{k}\para{\frac{e^{\thickbar{U}^{1}_{k}(\theta_{b_{k}})}}{e^{\bar{U}_{k}^{1}(u_{1})}}\frac{e^{\xi(b_{k})}}{e^{ \xi(u_{2})}}x_{k}}\bullet T,   
\end{eqalign}
where $I_{T}:= [T,T+Q_{H}(p)\bullet T]$. We now examine the relation of those the extremizers to $Q^{k}_{U}(b_{k})$. We have
\begin{eqalign}
 u_{i}\leq T+ Q_{H}(p)\bullet T=Q_{H}(c_{k,U,\xi}+x_{k,U,\xi})\leq  Q^{k}_{U}(b_{k}),  
\end{eqalign}
where we applied the events $\Xi_{[1,M]}, U_{[1,M]}$.
\proofparagraph{Step 5: Switching the $Q_{H}(y_{k})$ to $Q_{U}^{k}$}
For the term $B_{k}=Q_{H}(y_{k})\bullet \para{M_{k}+Q_{H}(c_{k})}=Q_{H}(y_{k})\bullet S_{k}$, with $S_{k}:=M_{k}+Q_{H}(c_{k})$, we again bound
\begin{eqalign}
 &\min_{u_{1},u_{2}\in I_{S}}Q_{U}^{k}\para{\frac{e^{\thickbar{U}^{1}_{k}(\theta_{b_{k}})}}{e^{\bar{U}_{k}^{1}(u_{1})}}\frac{e^{\xi(b_{k})}}{e^{\xi(u_{2})}}y_{k}}\bullet S_{k}\leq Q_{H}(y_{k})\bullet S_{k}\leq \max_{u_{1},u_{2}\in I_{S}}Q_{U}^{k}\para{\frac{e^{\thickbar{U}^{1}_{k}(\theta_{b_{k}})}}{e^{\bar{U}_{k}^{1}(u_{1})}}\frac{e^{\xi(b_{k})}}{e^{\xi(u_{2})}}y_{k}}\bullet S_{k},
\end{eqalign}
where $I_{S}:= [S,S+Q_{H}(\tilde{p})\bullet S]\subset [0,Q^{k}_{U}(b_{k})]$ due to the events $\Xi_{[1,M]}, U_{[1,M]}$. As a result, we can take the following  maximum for the outer inverses over interval the $P_{k}$
\begin{eqalign}\label{eq:productands}
&\prod\limits_{k\in C_{1}}\para{\frac{M_{k}}{B_{k}}}^{p_{k}}\cdot\prod\limits_{k\in C_{2}}\para{\frac{M_{k}}{A_{k}}}^{p_{k}}\leq\prod\limits_{k\in C_{1}}\maxls{u_{k}\in P_{k} }\para{\frac{M_{k}}{B_{k}(u_{k})}}^{p_{k}}\cdot\prod\limits_{k\in C_{2}}\maxls{u_{k}\in P_{k} }\para{\frac{M_{k}}{A_{k}(u_{k})}}^{p_{k}},
\end{eqalign}
for
\begin{eqalign}
A_{k}(u_{k}):=Q_{U}^{k}\para{e^{2u_{k}}x_{k}}\bullet Q_{H}(c_{k})\tand B_{k}(u_{k}):= &Q_{U}^{k}\para{e^{2u_{k}}y_{k}}\bullet\para{M_{k}+Q_{H}(c_{k})}.   
\end{eqalign}
\proofparagraph{Step 6: Switching the $M_{k}=Q_{H}(c_{k},c_{k}+d_{k}-c_{k})$ to $Q^{k}_{U}$}
Next we repeat for $M_{k}$. Before we modify it we take a maximum over $M_{k}$ using the event $D^{k}_{\ell_{k},M}$
\begin{eqalign}\label{eq:mainsummandunitcircle}
\eqref{eq:productands}\leq &\prod\limits_{k\in C_{1}} \maxls{u_{k}\in P_{k} \\T_{k}\in I_{\ell_{k}}^{k}}\para{\frac{a_{\ell_{k}+1}}{B_{k}(u_{k},T_{k})}}^{p_{k}}\prod\limits_{k\in C_{2}}\para{\frac{a_{\ell_{k}+1}}{A_{k}(u_{k})}}^{p_{k}}
\end{eqalign}
for $B_{k}(u_{k},T_{k}):=Q_{U}^{k}\para{e^{2u_{k}}y_{k}}\bullet\para{T_{k}+Q_{H}(c_{k})}$. We repeat as above to upper bound the decomposition event $D^{k}_{\ell_{k},M}$ using the events $\Xi_{[1,M]}, U_{[1,M]}$ and maximum
\begin{eqalign}
&\ind{D^{k}_{\ell_{k},M}}\leq  \maxls{u_{k}\in P_{k} }\ind{  D^{k}_{\ell_{k},M}(u_{k})},\\
\tfor D^{k}_{\ell_{k},M}(u_{k}):=&\set{Q_{U}^{k}\para{e^{2u_{k}}z_{k}}\bullet Q_{H}(c_{k})\in [a_{\ell_{k}},a_{\ell_{k}+1}]=I_{\ell_{k}}^{k}},\\ 
z_{k}:=&\branchmat{d_{k}-c_{k}&\tcwhen k\in C_{1}\\ d_{k}-c_{k}+y_{k} &\tcwhen k\in C_{2}}.
\end{eqalign}
\proofparagraph{Step 7: Switching the $ Q_{H}(c_{k}) $ to $ Q_{U}^{k}$}
Finally, we repeat for $ Q_{H}(c_{k}) $. We take maximum using the event $S^{k}_{m_{k},c}(\ell_{k})$ and bound that event by maximum too using the events $\Xi_{[1,M]}, U_{[1,M]}$ to get the following overall bound
\begin{eqalign}\label{eq:maineventmultiunitcirclepreremovalxiupredec}
&\eqref{eq:maineventmultiunitcirclepreremovalxiu}\leq
 \sumls{(\ell_{1},m_{1},\cdots,\ell_{M},m_{M})}\prod_{k\in [1,M]}E_{k}^{D}E_{k}^{S} R_{k}\ind{ G_{[1,M]}}\\ =&\sumls{(\ell_{1},m_{1},\cdots,\ell_{M},m_{M})}\prod_{k}\para{\maxls{u_{k}\in P_{k}\\\wt{T}_{k}\in\wt{I}_{m_{k},c}^{k}(\ell_{k})  }\ind{D^{k}_{\ell_{k},M}(u_{k},\wt{T}_{k})}\maxls{u_{k}\in P_{k} }\ind{S^{k}_{m_{k},c}(\ell_{k},u_{k})} }\cdot\\
 \cdot&\prod\limits_{k\in C_{1}} \para{\maxls{u_{k}\in P_{k} \\T_{k}\in I_{\ell_{k}}^{k}\\\wt{T}_{k}\in\wt{I}_{m_{k},c}^{k}(\ell_{k}) }\para{\frac{a_{\ell_{k}+1}}{Q_{U}^{k}\para{e^{2u_{k}}y_{k}}\bullet \para{T_{k}+\wt{T}_{k}+Q^{k+1}(b_{k+1})+\delta_{k+1} }}}^{p_{k}}}\\
 \cdot &\prod\limits_{k\in C_{2}}\para{\maxls{u_{k}\in P_{k}\\\wt{T}_{k}\in\wt{I}_{m_{k},c}^{k}(\ell_{k})  }\para{\frac{a_{\ell_{k}+1}}{Q_{U}^{k}\para{e^{2u_{k}}x_{k}}\bullet \para{\wt{T}_{k}+Q^{k+1}(b_{k+1})+\delta_{k+1}  }}}^{p_{k}} }\cdot \ind{ G_{[1,M]}}\\ =:&\sumls{(\ell_{1},m_{1},\cdots,\ell_{M},m_{M})}\prod_{k}E_{k}^{D}E_{k}^{S}R_{k}\ind{ G_{[1,M]}},
\end{eqalign}
for 
\begin{eqalign}\label{eq:alltheventsdecompositions}
&R_{k}:=\branchmat{\maxls{u_{k}\in P_{k} \\T_{k}\in I_{\ell_{k}}^{k}\\\wt{T}_{k}\in\wt{I}_{m_{k},c}^{k}(\ell_{k}) }\para{\frac{a_{\ell_{k}+1}}{Q_{U}^{k}\para{e^{2u_{k}}y_{k}}\bullet \para{T_{k}+\wt{T}_{k}+Q^{k+1}(b_{k+1})+\delta_{k+1} }}}^{p_{k}}&\tcwhen k\in C_{1}\\ \maxls{u_{k}\in P_{k}\\\wt{T}_{k}\in\wt{I}_{m_{k},c}^{k}(\ell_{k})  }\para{\frac{a_{\ell_{k}+1}}{Q_{U}^{k}\para{e^{2u_{k}}x_{k}}\bullet \para{\wt{T}_{k}+Q^{k+1}(b_{k+1})+\delta_{k+1}  }}}^{p_{k}} &\tcwhen k\in C_{2}}\\
&E_{k}^{D}:=\maxls{u_{k}\in P_{k}\\\wt{T}_{k}\in\wt{I}_{m_{k},c}^{k}(\ell_{k})  }\ind{D^{k}_{\ell_{k},M}(u_{k},\wt{T}_{k})}\\
&E_{k}^{S}:=\maxls{u_{k}\in P_{k} }\ind{S^{k}_{m_{k},c}(\ell_{k},u_{k})} \\
D^{k}_{\ell_{k},M}(u_{k},\wt{T}_{k}):=&\set{Q_{U}^{k}\para{e^{2u_{k}}(d_{k}-c_{k})}\bullet\para{\wt{T}_{k}+Q^{k+1}(b_{k+1})+\delta_{k+1}}\in [a_{\ell_{k}},a_{\ell_{k}+1}]=I_{\ell_{k}}^{k}},
\end{eqalign}
\begin{equation*}
S^{k}_{m_{k},c}(\ell_{k},u_{k}):=\{Q_{U}^{k}(e^{2u_{k}}c_{k})\in  [\wt{b}_{m_{k}}^{k}(\ell_{k}),\wt{b}_{m_{k}+1}^{k}(\ell_{k})]+Q^{k+1}(b_{k+1})+\delta_{k+1}=\wt{I}_{m_{k},c}^{k}(\ell_{k})+Q^{k+1}(b_{k+1})+\delta_{k+1}\}.    
\end{equation*}
\proofparagraph{Step 8: Decoupling} Now that we turned everything into $Q^{k}_{U}$, we can decouple. We will apply \cref{deltaSMP} in order to remove the shifts $Q^{k+1}(b_{k+1})+\delta_{k+1}$ and so we will return to the single-point estimates. We will do this by bootsrapping starting from the first scale $R_{1}$. First as we did in the proof of \cref{prop:productconditionalratio}, we scale all the terms by $\delta_{1}$ which means that the first scale is now $Q^{1}$ (i.e. $U$ of height one) but for the other scales to ease on notation we still just write $Q^{2},Q^{3}$ (as opposed to $Q^{\delta_{2}/\delta_{1}}$). We start with applying the tower property for the filtration $\CF(Q^{2}(b_{2}))$
\begin{eqalign}\label{maineventdecouplingsums}
&\sumls{(\ell_{1},m_{1},\cdots,\ell_{M},m_{M})}\Expe{\prod_{k\in [1,M]} E_{k}^{D}E_{k}^{S}R_{k}\ind{ G_{[1,M]}}}\\
=&\sumls{(\ell_{1},m_{1},\cdots,\ell_{M},m_{M})}\Expe{\Expe{E_{1}^{D}E_{1}^{S}R_{1}\conditional \CF(Q^{2}(b_{2})}\prod_{k\in [2,M]}E_{k}^{D}E_{k}^{S} R_{k}\ind{ G_{[1,M]}}}.    
\end{eqalign}
 We separate the decomposition event $E_{1}^{S}$ by applying conditional \Holder
\begin{equation}
\Expe{E_{1}^{D}E_{1}^{S}R_{1}\conditional \CF(Q^{2}(b_{2}))}\leq \para{\Expe{\para{E_{1}^{D}R_{1}}^{b_{11}}\conditional \CF(Q^{2}(b_{2}))}}^{1/b_{11}}\para{\Proba{E_{1}^{S}\conditional \CF(Q^{2}(b_{2}))}}^{1/b_{12}},        
\end{equation}
for $b_{11}^{-1}+b_{12}^{-1}=1$. So now we study those two factors. As explained in the previous section, for the first event we have a delta-gap and so we can apply \cref{deltaSMP}
\begin{eqalign}\label{eq:mainchangeeventfirstscale}
&\Expe{\para{E_{1}^{D}R_{1}}^{b_{11}}\conditional \CF(Q^{2}(b_{2})}=\Expe{\para{\tilde{E}_{1}^{D}\tilde{R}_{1}}^{b_{11}}},
\end{eqalign}
where we removed the shift from \cref{eq:alltheventsdecompositions}
\begin{eqalign}\label{eq:alltheventsdecompositions2}
\tilde{R}_{k}:=&\branchmat{\maxls{u_{k}\in P_{k} \\T_{k}\in I_{\ell_{k}}^{k}\\\wt{T}_{k}\in\wt{I}_{m_{k},c}^{k}(\ell_{k}) }\para{\frac{a_{\ell_{k}+1}}{Q_{U}^{k}\para{e^{2u_{k}}y_{k}}\bullet \para{T_{k}+\wt{T}_{k} }}}^{p_{k}}&\tcwhen k\in C_{1}\\ \maxls{u_{k}\in P_{k}\\\wt{T}_{k}\in\wt{I}_{m_{k},c}^{k}(\ell_{k})  }\para{\frac{a_{\ell_{k}+1}}{Q_{U}^{k}\para{e^{2u_{k}}x_{k}}\bullet \para{\wt{T}_{k}  }}}^{p_{k}} &\tcwhen k\in C_{2}}\\
\tilde{E}_{k}^{D}:=&\maxls{u_{k}\in P_{k}\\\wt{T}_{k}\in\wt{I}_{m_{k},c}^{k}(\ell_{k})  }\ind{D^{k}_{\ell_{k},M}(u_{k},\wt{T}_{k})}\\
\tilde{E}_{k}^{S}:=&\maxls{u_{k}\in P_{k} }\ind{S^{k}_{m_{k},c}(\ell_{k},u_{k})} \\
\tilde{D}^{k}_{\ell_{k},M}(u_{k},\wt{T}_{k}):=&\set{Q_{U}^{k}\para{e^{2u_{k}}(d_{k}-c_{k})}\bullet\para{\wt{T}_{k}}\in [a_{\ell_{k}},a_{\ell_{k}+1}]=I_{\ell_{k}}^{k}},\\
\tilde{S}^{k}_{m_{k},c}(\ell_{k},u_{k}):=&\set{Q_{U}^{k}(e^{2u_{k}}c_{k})\in [\wt{b}_{m_{k}}^{k}(\ell_{k}),\wt{b}_{m_{k}+1}^{k}(\ell_{k})]=\wt{I}_{m_{k},c}^{k}(\ell_{k})}.
\end{eqalign}
Since we decoupled the decomposition event $E_{1}^{S}$, we also translate
\begin{eqalign}
\wt{T}_{1}\in\wt{I}_{m_{1},a}(\ell_{1})\mapsto \wt{T}_{1}\in\spara{0,\abs{\wt{I}_{m_{1},a}(\ell_{1})}}\subset [0,\rho_{a}a_{m_{1}}]=:I_{\rho,m_{1}},    
\end{eqalign}
in order to remove the $\ell_1$-dependence and thus bring in the $\ell_{1}$-sum and we used the bound
\begin{equation}
 \abs{\tilde{I}_{k,a}(\ell)}=\rho_{a}a_{k}((\ell+1)^{\lambda_{a}}-\ell^{\lambda_{a}})\leq    \rho_{a}a_{k}.
 \end{equation}
\proofparagraph{Step 9: Event $E_{1}^{S}$}
 Here we study the event
\begin{eqalign}\label{eq:eventclk}
&\set{Q_{U}^{1}(e^{2u_{1}}c_{1})\in \spara{c_{\ell,k}+Q^{2}(b_{2})+\delta_{2},c_{\ell+1,k}+Q^{2}(b_{2})+\delta_{2}}}\\
\subset &\set{\eta^{1}\para{0,c_{\ell,k}+Q^{2}(b_{2})+\delta_{2}} \leq e^{2u_{1}}c_{1}   }\\
=&\set{\eta^{1}\para{Q^{2}(b_{2})+\delta_{2},c_{\ell,k}+Q^{2}(b_{2})+\delta_{2}} \leq e^{2u_{1}}c_{1}-\eta^{1}\para{0,Q^{2}(b_{2})+\delta_{2}}   }\\
\subset& \set{\eta^{1}\para{Q^{2}(b_{2})+\delta_{2},c_{\ell,k}+Q^{2}(b_{2})+\delta_{2}} \leq e^{2u_{1}}c_{1}} ,
\end{eqalign}
where in the context of \cref{eq:alltheventsdecompositions}, we have $c_{\ell_{1},m_{1}}=\wt{b}_{m_{1}}(\ell_{1})$ 
and then we again apply \cref{deltaSMP} to decouple
\begin{eqalign}
\set{\eta^{1}\para{Q^{2}(b_{2})+\delta_{2},c_{\ell,k}+Q^{2}(b_{2})+\delta_{2}} \leq e^{2u_{1}}c_{1}}\eqdis\set{\wt{\eta}^{1}\para{0,c_{\ell,k}} \leq c_{1} e^{2u_{1}} },    
\end{eqalign}
where $\wt{\eta}^{1}$ is an independent copy of $\eta^{1}$. So we have
\begin{equation}
\Expe{E_{1}^{S}\conditional \CF\para{[0,Q^{2}(b_{2})]}}\leq \Expe{\maxls{u_{1}\in P_{1} }\ind{\wt{\eta}^{1}\para{0,c_{\ell,k}} \leq e^{2u_{1}}c_{1}}}= \Proba{\wt{\eta}^{1}\para{0,c_{\ell,k}} \leq e^{4}c_{1}},   
\end{equation}
where we use the $P_{1}\subset (-2,2)$ from \cref{def:exponentialchoiceparam}.
Now that we decoupled we proceed as in \cref{prop:inverse_ratio_moments}.  Recall that $c_{\ell,k}=\rho_{a}a_{k}c_{\ell}$ for some small enough $\rho_{a}\in (0,1)$ and $c_{\ell}:=\ell^{\lambda_{a}}$ for $\lambda_{a}\in (0,1)$ so that indeed $c_{\ell+1}-c_{\ell}\to 0$ as $\ell\to +\infty$. For $\ell=0$, we bound by zero.  When $c_{\ell,k}\leq \delta_{1}\tor c_{\ell,k}> \delta_{1}$ we use Markov with $p_{1}>0\tand p_{2}>0$ respectively,to upper bound by
\begin{eqalign}\label{eq:singularfactorfromscalinglaw0}
&\sum_{\ell\geq 1 }\para{\Expe{E_{1}^{S} \conditional\CF\para{[0,Q^{2}(b_{2})]}}}^{1/b_{12}}\leq  \sum_{\ell\geq 1 } \para{\Proba{\wt{\eta}^{1}\para{0,c_{\ell,k}} \leq e^{4}c_{1} }}^{1/b_{12}}  \\
\leq &\sum_{\ell\geq 1 }  c_{1}^{p_{1}/b_{12}}\para{c_{\ell,k}}^{\zeta(-p_{1})/b_{12}}+ c_{1}^{p_{2}/b_{12}}\para{c_{\ell,k}}^{-p_{2}/b_{12}}\\
\leq &c a_{k}^{-\frac{1+\e}{\lambda_{a}}}.
\end{eqalign}
So from here to get finiteness we require both $-\zeta(-p_{1})\lambda_{a},p_{2}\lambda_{a}>b_{12}$. Therefore, we get the singular power $a_{k}^{-\frac{1+\e}{\lambda_{a}}}$ for arbitrarily small $\e>0$.
\proofparagraph{Step 10: Bootstrap}
Therefore, we removed all the conditional expectations and are left with
\begin{eqalign}\label{maineventdecouplingsums2}
\eqref{maineventdecouplingsums}\leq & \para{\sumls{\ell_{1},m_{1}}a_{m_{1}}^{-\frac{1+\e}{\lambda_{a}}} \para{\Expe{\para{\tilde{E}_{1}^{D}\tilde{R}_{1}}^{b_{11}}}}^{1/b_{11}}} \cdot\para{ \sumls{\ell_{i},m_{i}\\ i\in [2,M]}\Expe{\prod_{k\in [2,M]}E_{k}^{D}E_{k}^{S} R_{k}\ind{ G_{[1,M]}}}}.    
\end{eqalign}
So that means that we decoupled the first scale at the cost of a summation that will give a singular factor in $x^{-\e_{ratio}}$. In particular, we next prove that
\begin{eqalign}\label{maineventdecouplingsums3}
\eqref{maineventdecouplingsums2}\leq & cx^{-\e_{ratio}}  \cdot\para{ \sumls{\ell_{i},m_{i}\\ i\in [2,M]}\Expe{\prod_{k\in [2,M]}E_{k}^{D}E_{k}^{S} R_{k}\ind{ G_{[1,M]}}}}.    
\end{eqalign}
Repeating for the rest of the scales in the order given by the gap-event, will give us $M$ such singular factors. 
\proofparagraph{Step 11: Ratio part for $1\in C_{1}$}
Here we study the first factor
\begin{equation}
\para{\Expe{\para{\tilde{E}_{1}^{D}\tilde{R}_{1}}^{b_{11}}}}^{1/b_{11}}   
\end{equation}
where recall from \cref{eq:alltheventsdecompositions2}
\begin{eqalign}
\tilde{E}_{1}^{D}=&\maxls{u_{1}\in P_{1}\\\wt{T}_{1}\in I_{\rho,\ell_{1}}  }\ind{D^{1}_{\ell_{1},M}(u_{1},\wt{T}_{1})},    \\
D^{1}_{\ell_{1},M}(u_{1},\wt{T}_{1})=&\set{Q_{U}^{1}\para{e^{2u_{1}}(d_{1}-c_{1})}\bullet\para{\wt{T}_{1}}\in [a_{\ell_{1}}^{1},a_{\ell_{1}+1}^{1}]=I_{\ell_{1}}^{1}}
\end{eqalign}
and
\begin{eqalign}
\tilde{R}_{1}=&\maxls{u_{1}\in P_{1} \\T_{1}\in I_{\ell_{1}}^{1}\\\wt{T}_{1}\in I_{\rho,\ell_{1}}}\para{\frac{a_{\ell_{1}+1}^{1}}{B_{1}(u_{1},T_{1},\wt{T}_{1})}}^{p_{1}}.
\end{eqalign}
We let
\begin{eqalign}
  B^{min}_{1}:=\minls{u_{1}\in P_{1} \\T_{1}\in I_{\ell_{1}}^{1}\\\wt{T}_{1}\in I_{\rho,\ell_{1}}}\para{B_{1}(u_{1},T_{1},\wt{T}_{1})}.  
\end{eqalign}
We go over each case as in the proof of \cref{prop:inverse_ratio_moments} for
\begin{eqalign}
&B^{min}_{1}\in [0,1]\cup[1,\infty],\\
&\tand a_{\ell_{1}}^{1}\in [0,\rho_{b-a}]\cup[\rho_{b-a},\infty).
\end{eqalign}

\proofparagraph{Step 12: Case $B^{min}_{1}\in [1,\infty)$}
In this case we are left with studying the sum
\begin{equation}
\sumls{\ell_{1}\in \Z\setminus\set{0}}a_{\ell_{1}}^{p_{1}-\frac{1+\e}{\lambda_{a}b_{12} }}\para{\Expe{\maxls{u_{1}\in P_{1}\\\wt{T}_{1}\in I_{\rho,\ell_{1}}  }\ind{Q_{U}^{1}\para{e^{2u_{1}}(d_{1}-c_{1})}\bullet\para{\wt{T}_{1}}\in [a_{\ell_{1}}^{1},a_{\ell_{1}+1}^{1}]}}}^{1/b_{11}}    
\end{equation}
for bi-infinite sequence $a_{\ell_{1}}$ going to zero and infinity. In particular, we set
\begin{eqalign}
a_{k}:=\branchmat{\rho_{b-a}\abs{k}^{-\lambda_{neg}} &k\leq -1\\
\rho_{b-a}k^{\lambda_{pos}}& k\geq 1}
\end{eqalign}
for $\lambda_{neg}>0$ and $\lambda_{pos}\in (0,1)$. For $k\geq 1$, we only keep the lower bound
\begin{eqalign}
&\Expe{\maxls{u_{1}\in P_{1}\\\wt{T}_{1}\in I_{\rho,\ell_{1}}  }\ind{Q_{U}^{1}\para{e^{2u_{1}}(d_{1}-c_{1})}\bullet\para{\wt{T}_{1}}\in [a_{\ell_{1}}^{1},a_{\ell_{1}+1}^{1}]}}\\
\leq &\Proba{ \minls{\wt{T}_{1}\in I_{\rho,\ell_{1}}  }\eta_{U}^{1}\para{\wt{T}_{1},\wt{T}_{1}+\rho_{b-a}k^{\lambda_{pos}}} \leq e^{2R_{1}^{pert}}(d_{1}-c_{1}) }    
\end{eqalign}
Here we use \cref{prop:minmodeta} for $p_{min}>0$ to get the sum
\begin{eqalign}
\sum_{k\geq 1}k^{\lambda_{pos}\para{p_{1}-\frac{1+\e}{\lambda_{a}b_{12} }+\zeta(-p_{min})-1}/b_{11}}.    
\end{eqalign}
This is finite by taking large enough $p_{min}$. For $k\leq -1$ we bound the probability by one and just study
\begin{eqalign}
\sumls{k\geq 1}k^{-\lambda_{neg}\para{p_{1}-\frac{1+\e}{\lambda_{a}b_{12} }}}.    
\end{eqalign}
Here as in the proof \cref{prop:inverse_ratio_moments} we take $p_{1}=1+\e$ and $b_{12}$ large enough so that the overall bracketed-exponent is positive $p_{1}-\frac{1+\e}{\lambda_{a}b_{12} }>0$ and then we just take $\lambda_{neg}$ arbitrarily large.
\proofparagraph{Step 13: Case $B^{min}_{1}\in [0,1]$ and $a_{\ell_{1}}\in [\rho_{b-a},\infty)$}
Here we study the sum
\begin{eqalign}
\sumls{\ell_{1}\geq 1}a_{\ell_{1}}^{p_{1}-\frac{1+\e}{\lambda_{a}b_{12} }}\para{\Expe{\para{\tilde{E}_{1}^{D}\ind{B^{min}_{1}\in [0,1]}\maxls{u_{1}\in P_{1} \\T_{1}\in I_{\ell_{1}}^{1}\\\wt{T}_{1}\in I_{\rho,\ell_{1}}}\para{\frac{1}{B_{1}(u_{1},T_{1},\wt{T}_{1})}}^{p_{1}}}^{b_{11}}}}^{1/b_{11}}.   
\end{eqalign}
As in the proof of \cref{prop:inverse_ratio_moments}, we start with applying layercake to write the expectation term as
\begin{eqalign}
&\int_{1}^{\infty}\Expe{\tilde{E}_{1}^{D}\ind{\frac{1}{B^{min}_{1}} \geq t^{1/p_{1}b_{11}} }} \dt   \\
=&\int_{1}^{\infty}\Expe{\tilde{E}_{1}^{D}\max_{u_{1},T_{1},\wt{T}_{1}}\ind{ \eta\para{T_{1}+\wt{T}_{1} ,T_{1}+\wt{T}_{1} +t^{-1/p_{1}b_{11}}}\geq e^{2u_{1}}y_{1} }} \dt  \\
\leq &\int_{1}^{\infty}\Proba{\minls{\wt{T}_{1}\in I_{\rho,\ell_{1}}  }\eta_{U}^{1}\para{\wt{T}_{1},\wt{T}_{1}+\rho_{b-a}k^{\lambda_{pos}}} \leq e^{2R_{1}^{pert}}(d_{1}-c_{1})}\\
&\cdot\Proba{  \max_{T_{1},\wt{T}_{1}}\eta\para{T_{1}+\wt{T}_{1} ,T_{1}+\wt{T}_{1} +t^{-1/p_{1}b_{11}}}\geq e^{-2R_{1}^{pert}}y_{1} } \dt,  
\end{eqalign}
where in the last inequality we only kept the lower bound in $\tilde{E}_{1}^{D}$ and then applied \nameref{FKGineq}. From here the proof proceeds as in the section "\textit{Case $Q(b,b+x)\in [0,1]$ and $Q_{b-a}\bullet Q_{a}\geq \rho_{b-a}$}" in the proof of \cref{prop:inverse_ratio_moments}.
\proofparagraph{Step 14: Case $B^{min}_{1}\in [0,1]$ and $a_{\ell_{1}}\in [0,\rho_{b-a}]$}
Here we study the sum
\begin{eqalign}
\sumls{\ell_{1}\leq -1}a_{\ell_{1}}^{-\frac{1+\e}{\lambda_{a}b_{12} }}\para{\Expe{\para{\tilde{E}_{1}^{D}\ind{B^{min}_{1}\in [0,1]}\maxls{u_{1}\in P_{1} \\T_{1}\in I_{\ell_{1}}^{1}\\\wt{T}_{1}\in I_{\rho,\ell_{1}}}\para{\frac{a_{\ell_{1}}}{B_{1}(u_{1},T_{1},\wt{T}_{1})}}^{p_{1}}}^{b_{11}}}}^{1/b_{11}}.   
\end{eqalign}
This is again the singular case that requires the use of scaling law. In the expectation we have
\begin{eqalign}
\Expe{\tilde{E}_{1}^{D}\ind{\frac{\max_{T_{1},\wt{T}_{1}}\eta\para{T_{1}+\wt{T}_{1} ,T_{1}+\wt{T}_{1} +a_{\ell_{1}}t^{-1/p_{1}b_{11}}}}{\minls{\wt{T}_{1}\in I_{\rho,\ell_{1}}  }\eta_{U}^{1}\para{\wt{T}_{1},\wt{T}_{1}+\rho_{b-a}a_{\ell_{1}+1}} }\geq e^{-2R_{1}^{pert}}c_{d-c,y}} },   
\end{eqalign}
where we did the same trick as in the singular case of \cref{prop:inverse_ratio_moments} of using the lower bound. From here we proceed the same because we have the same constraints on the shifts. The only difference is the bound $ e^{-2R_{1}^{pert}}$ which we treat as a constant. 
\proofparagraph{Step 15: Ratio part for $1\in C_{2}$}
Here we study 
\begin{equation}
\para{\Expe{\para{\tilde{E}_{1}^{D}\tilde{R}_{1}}^{b_{11}}}}^{1/b_{11}},   
\end{equation}
where recakk
\begin{eqalign}
\tilde{E}_{1}^{D}=&\maxls{u_{1}\in P_{1}\\\wt{T}_{1}\in I_{\rho,\ell_{1}}  }\ind{D^{1}_{\ell_{1},M}(u_{1},\wt{T}_{1})},    \\
D^{1}_{\ell_{1},M}(u_{1},\wt{T}_{1})=&\set{Q_{U}^{1}\para{e^{2u_{1}}(d_{1}-c_{1}+y_{1})}\bullet\para{\wt{T}_{1}}\in [a_{\ell_{1}}^{1},a_{\ell_{1}+1}^{1}]=I_{\ell_{1}}^{1}},
\end{eqalign}
and
\begin{eqalign}
\tilde{R}_{1}=&\maxls{u_{1}\in P_{1}\\\wt{T}_{1}\in I_{\rho,\ell_{1}}  }\para{\frac{a_{\ell_{1}+1}^{1}}{A_{1}(u_{1},\wt{T}_{1})}}^{p_{1}}=\para{\frac{a_{\ell_{1}+1}^{1}}{A^{min}}}^{p_{1}}\tand A^{min}:=\minls{u_{1}\in P_{1}\\\wt{T}_{1}\in I_{\rho,\ell_{1}}  }A_{1}(u_{1},\wt{T}_{1})
\end{eqalign}
As explained in the proof of \cref{prop:inverse_ratio_moments} this is an easier case of the previous because now they have matching shifts $Q_{c_{1}}$. So we leave it.
\subsection{Proof of \sectm{\cref{prop:Multipointunitcircleandmaximum}}: Decreasing numerator}
 Here we again study the product
\begin{eqalign}\label{eq:maxproductunitcircledecre}
&\Expe{\prod\limits_{k} \para{\frac{Q_{H}(c_{k},c_{k}+x_{k})}{Q_{H}(d_{k},d_{k}+y_{k})}}^{p_{k}}\ind{ G_{[1,M]},\Xi_{[1,M]}, U_{[1,M]}}}.
\end{eqalign}
After we scale by $\delta$, the denominators will be  nonsingular and so there is no need for the log-normal scaling law. We assume $c_{k}<d_{k}$ because the other case is symmetric. We start with decomposing $Q_{H}(c_{k})$ using the analogous decomposition events from \cref{eq:alltheventsdecompositions}
\begin{eqalign}
S_{m_{k},c}(u_{k}):=&\set{Q_{U}^{k}(e^{2u_{k}}c_{k})\in [\wt{b}_{m_{k}},\wt{b}_{m_{k}+1}]+Q^{k+1}(b_{k+1})+\delta_{k+1}},\\
\wt{I}_{m_{k},c}:=&[\wt{b}_{m_{k}},\wt{b}_{m_{k}+1}],
\end{eqalign}
for the diverging sequence from \cref{eq:divergingsequencebtilde} but for $\ell=1$, i.e. $\wt{b}_{m}:=\wt{b}_{m}(1)=\rho_{a}a_{m}$. So using the the semigroup formula we upper bound by a maximum
\begin{eqalign}\label{eq:maxproductunitcircledecre2}
\eqref{eq:maxproductunitcircledecre}\leq \Exp\Bigg[&\prod\limits_{k}\sum_{m_{k}}\max_{u_{k},T_{k}}\ind{S^{k}_{m_{k},c}(\ell_{k},u_{k})}\para{\frac{Q_{H}(x_{k})\bullet T_{k}}{Q_{H}(y_{k})\bullet \para{Q_{H}(d_{k}-c_{k})\bullet T_{k}+T_{k}}}}^{p_{k}}\\
&\cdot\ind{ G_{[1,M]},\Xi_{[1,M]}, U_{[1,M]}}\Bigg],   
\end{eqalign}
where the maximum is over $u_{k}\in P_{k}$ and $T_{k}\in \wt{I}_{m_{k},c}^{k}(\ell_{k})+Q^{k+1}(b_{k+1})+\delta_{k+1}$. Here we need to further switch to $Q_{U}^{k}$ by taking maximum again
\begin{eqalign}\label{eq:maxproductunitcircledecre3}
&\max_{T_{k}}\para{\frac{Q_{H}(x_{k})\bullet T_{k}}{Q_{H}(y_{k})\bullet \para{Q_{H}(d_{k}-c_{k})\bullet T_{k}+T_{k}}}}^{p_{k}}\\    
\leq &\max_{T_{k}}\max_{u^{i}_{k}\in P_{k}}\para{\frac{Q_{U}^{k}(e^{u^{1}_{k}}x_{k})\bullet T_{k}}{Q_{U}^{k}(e^{u^{2}_{k}}y_{k})\bullet \para{Q_{U}^{k}(e^{u^{3}_{k}}(d_{k}-c_{k})\bullet T_{k}+T_{k}}}}^{p_{k}}=:R_{k}.
\end{eqalign}
\proofparagraph{Decoupling}
As we did in the proof of \cref{prop:productconditionalratio}, we scale all the terms by $\delta_{1}$ which means that the first scale is now $Q^{1}$ (i.e. $U$ of height one) but for the other scales to ease on notation we still just write $Q^{2},Q^{3}$ (as opposed to $Q^{\delta_{2}/\delta_{1}}$). That also means that the denominator for $k=1$ is a constant term that does not depend on $\delta_{1}$ (since we set $y_{1}=r_{1}\delta_{1}$).\\
Next as done in the equal-length case, we apply the tower-property
\begin{eqalign}
\sum_{m_{1}}\Expe{ E_{1}^{S}  R_{1}Y_{[2,M]}}    
=\sum_{m_{1}}\Expe{\Expe{ E_{1}^{S}  R_{1}\conditional \CF(Q^{2}(b_{2}))} Y_{[2,M]}},    
\end{eqalign}
where $E_{1}^{S}:=\max_{u_{1},T_{1}}\ind{S^{1}_{m_{1},a}(u_{k})}$ and $Y_{[2,M]}$ are the remaining factors. Since we used the decomposition event, we separate it via conditional-\Holder
\begin{equation}
\Expe{E_{1}^{S} R_{1}\conditional \CF(Q^{2}(b_{2})}\leq \para{\Expe{\para{R_{1}}^{b_{11}}\conditional \CF(Q^{2}(b_{2})}}^{1/b_{11}}\para{\Proba{E_{1}^{S}\conditional \CF(Q^{2}(b_{2})}}^{1/b_{12}},        
\end{equation}
for $b_{11}^{-1}+b_{12}^{-1}=1$. For the event $E_{1}^{S}$ the situation is simpler than before since there is no $\ell$-dependence and so we can use the same method as before in \cref{eq:eventclk}. For the first factor, we can apply \cref{deltaSMP} to be left with studying
\begin{eqalign}\label{eq:ratiowithmaximum}
\para{\Expe{\max_{T_{1}\in [0,2]}\max_{u^{i}_{1}\in P_{1}}\para{\frac{Q_{U}^{1}(e^{u^{1}_{1}}x_{1})\bullet T_{1}}{Q_{U}^{1}(e^{u^{2}_{1}}y_{1})\bullet \para{Q_{U}^{1}(e^{u^{3}_{1}}(d_{1}-c_{1})\bullet T_{1}+T_{1}}}}^{b_{11}p_{1}}}}^{1/b_{11}}.
\end{eqalign}
By applying \Holder and bounding $e^{u^{1}_{1}}x_{1}\leq e^{2R_{1}^{pert}}x_{1}$, we almost return to the setting of \cref{eq:maximumratiodecnum}. For the numerator we simply substitute $x_{1}$ by $e^{2R_{1}^{pert}}x_{1}$ to get same exponent for $x_{1}$. For the denominator we instead use a version with perturbation in \cref{lem:minshiftedGMCpertubation}. But in the above specific situation it gives the same exponent as the un-perturbed version in \cref{lem:minshiftedGMC}. So by boostrapping we get the desired bound in \cref{eq:decaynumeratorunic}.

\part{ Appendix}

\section{Further research directions }\label{furtherresearchdirections}
\begin{enumerate}

\item\textbf{Conditional independence} Starting with variables $Q(b,c),Q_{a}$ and event $E:=\set{Q(b)\geq Q(a)+1}$, is there some sigma-algebra $\mathcal{G}$ that gives conditional independence
\begin{eqalign}
\Proba{Q(a)\geq t_{1},Q(b,c)\geq t_{2},E }=\Expe{\Proba{Q(a)\geq t_{1}\mid \mathcal{G}}\Proba{Q(b,c)\geq t_{2}\mid \mathcal{G}}}?    
\end{eqalign}
It seems unlikely because contrary to the case of random walks's conditional independence \cref{eq:randomwalkcondinde}, we don't have any independent increments. So if we decompose $Q_{b}=\ell$, then this event is still highly correlated with the increment $Q_{c-b}\bullet Q_{b}=Q_{c-b}\bullet \ell$.

  \item \textbf{Independence number}\\
 The independent interval graph with \iid endpoints has been studied in \cite{justicz1990random}. It would be interesting if a similar construction of some correlated Poisson process could give exact growth of the independence number.

    \item \textbf{Better bounds on the ratio moments.} In the similar spirit for \cref{prop:inverse_ratio_moments}, it would be interesting to possible experiment getting sharper exponents for the factor $x$ (in the decreasing numerator case) or even bounds that are uniform in $x$ (in equal length case).

    \item \textbf{Convergence to Lebesgue:} The extension to
\begin{eqalign}
\Proba{\abs{Q^{n}(a,b)-(b-a)}>\Delta} \tand \Proba{\eta(Q_{a},Q_{a}+x)-x\geq \Delta}
\end{eqalign}
is unclear. This is because the current proof using decoupling doesn't work as evaluating GMC over two distant subintervals $I_{1},I_{3}\subset \spara{Q_{a},Q_{a}+x} $ is still correlated.

\end{enumerate}

\begin{appendices}
\section{Comparison inequalities}
 We will use the Kahane's inequality (eg. \cite[corollary A.2]{robert2010gaussian}).
\begin{theorem}(Kahane Inequality)\label{Kahanesinequality}
Let $\rho$ be a Radon measure on a domain $D\subset\R^{n}$, $X(\cdot)$ and $Y(\cdot)$ be two continuous centred Gaussian fields on $D$, and $F: \Rplus \to \R$ be some smooth function with at most polynomial growth at infinity. For $t \in [0,1]$, define $Z_t(x) = \sqrt{t}X(x) + \sqrt{1-t}Y_t(x)$ and
\begin{eqalign}
\varphi(t) := \EE \left[ F(W_t)\right], \qquad
W_t := \int_D e^{Z_t(x) - \frac{1}{2}\EE[Z_t(x)^2]} \rho(dx).
\end{eqalign}

 Then the derivative of $\varphi$ is given by
\begin{equation}\label{eq:Kahane_int}
\begin{split}
\varphi'(t) & = \frac{1}{2} \int_D \int_D \left(\EE[X(x) X(y)] - \EE[Y(x) Y(y)]\right) \\
& \qquad \qquad \times \EE \left[e^{Z_t(x) + Z_t(y) - \frac{1}{2}\EE[Z_t(x)^2] - \frac{1}{2}\EE[Z_t(y)^2]} F''(W_t) \right] \rho(dx) \rho(dy).
\end{split}
\end{equation}

 In particular, if
\begin{eqalign}
\EE[X(x) X(y)] \le \EE[Y(x) Y(y)] \qquad \forall x, y \in D,
\end{eqalign}

 then for any convex $F: \Rplus \to \R$
\begin{eqalign}\label{eq:Gcomp}
\EE \left[F\left(\int_D e^{X(x) - \frac{1}{2} \EE[X(x)^2]}\rho(dx)\right)\right]
\le
\EE \left[F\left(\int_D e^{Y(x) - \frac{1}{2} \EE[Y(x)^2]}\rho(dx)\right)\right].
\end{eqalign}

 and the inequality is reversed if $F$ is concave instead. In the case of distributional fields $X,Y$, we use this result for the mollifications $X_{\e},Y_{\e}$ and pass to the limit due to the continuity of $F$.
\end{theorem}
 The interpolation has implications for the max/min of Gaussian fields.
\begin{theorem}\label{th:slepian}(Slepian's lemma)\cite[Theorem 7.2.1]{vershynin2018high}
Let $\set{X_t}_{t\in T}$ and $\set{Y_t}_{t\in T}$ be two mean zero Gaussian processes. Assume that for all $t, s \in T$, in some index set $T$, we have
\begin{equation}
\Expe{X_{t}^{2}}=\Expe{Y_{t}^{2}}    \tand \Expe{X_{t}X_{s}}\leq \Expe{Y_{t}Y_{s}}. 
\end{equation}
In particular, for every $\tau\in \R$, we have
\begin{equation}
\Proba{\sup_{t\in T}X_{t}\leq \tau }\leq \Proba{\sup_{t\in T}Y_{t}\leq \tau }\tand \Proba{\inf_{t\in T}X_{t}\geq -\tau }\leq \Proba{\inf_{t\in T}Y_{t}\geq -\tau }.    
\end{equation}
\end{theorem}
 We will also consider another comparison inequality for increasing/decreasing functionals which generalizes the Harris-inequality \cite[Theorem 2.15]{boucheron2013concentration} that is called FKG inequality for Gaussian fields \cite{pitt1982positively},\cite[Theorem 3.36]{berestycki2021gaussian}.
\begin{theorem}[FKG inequality]\label{FKGineq}
Let $\set{Z(x)}_{x\in U}$ be an a.s. continuous centred Gaussian field on $U \subset \R^d$ with $\Expe{Z(x)Z(y)} \geq 0$ for all $x, y \in U$. Then, if $f, g$ are two bounded, tincreasing measurable functions,
\begin{equation}
\Expe{f\para{\set{Z(x)}_{x\in U}}g\para{\set{Z(x)}_{x\in U}}}    \geq \Expe{f\para{\set{Z(x)}_{x\in U}}}\Expe{g\para{\set{Z(x)}_{x\in U}}}    
\end{equation}
and the opposite inequality if one function is increasing and the other one is decreasing
\begin{equation}
\Expe{f\para{\set{Z(x)}_{x\in U}}g\para{\set{Z(x)}_{x\in U}}}    \leq \Expe{f\para{\set{Z(x)}_{x\in U}}}\Expe{g\para{\set{Z(x)}_{x\in U}}}.    
\end{equation}
\end{theorem}

\section{Moment estimates for GMC }
\subsection{Moments bounds of GMC}
 For positive continuous bounded function $g_{\delta}:\Rplus\to \Rplus$ with $g_{\delta}(x)=0$ for all $x\geq \delta$ and uniform bound $M_{g}:=\sup_{\delta\geq 0}\norm{g_{\delta}}_{\infty}$,  we will need the following moment estimates for $\eta^{\delta}=\eta_{g}^{\delta}$ in order to compute the moments of the inverse $Q_{g}^{\delta}$. As mentioned in the notations we need to study more general fields such as $U^{\delta,\lambda}$ in \cref{eq:truncatedscaled}
\begin{equation}
\Expe{U_{ \varepsilon}^{\delta,g }(x_{1} )U_{ \varepsilon}^{  \delta,g }(x_{2} )  }=\left\{\begin{matrix}
\ln(\frac{\delta }{\varepsilon} )-\para{\frac{1 }{\e}-\frac{1}{\delta}}\abs{x_{2}-x_{1}}+g_{\delta}(\abs{x_{2}-x_{1}})&\tifc \abs{x_{2}-x_{1}}\leq \varepsilon\\ 
 \ln(\frac{\delta}{\abs{x_{2}-x_{1}}})-1+\frac{\abs{x_{2}-x_{1}}}{\delta}+g_{\delta}(\abs{x_{2}-x_{1}}) &\tifc \e\leq \abs{x_{2}-x_{1}}\leq \delta\\
  0&\tifc \delta\leq \abs{x_{2}-x_{1}}
\end{matrix}\right.    .
\end{equation}
The following proposition shown in \cite{binder2023inverse} is studying the moment bounds for this field.
\begin{proposition}\label{momentseta}
We have the following estimates: 
\begin{itemize}
    \item  For $q\in (-\infty,0)\cup [1,\frac{2}{\gamma^{2}})$ we have
\begin{equation}
\Expe{\para{\eta^{\delta}[0,  t]}^{q}}\leq c_{1}\branchmat{t^{\zeta(q)}\delta^{\frac{\gamma^{2}}{2}(q^{2}-q)} &t\in [0,\delta]\\ t^{q}& t\in [\delta,\infty] },
\end{equation}
where $\zeta(q):=q-\frac{\gamma^{2}}{2}(q^{2}-q)$ is the multifractal exponent and $c_{1}:=e^{\frac{\gamma^{2}}{2}(q^{2}-q)M_{g}}\Expe{\para{\int_{0}^{1}e^{ \overline{\omega}^{1}(x)}\dx}^{q}}$ and $\omega^1$ is the exact-scaling field of height $1$.

\item  When $t,\delta\in [0,1]$ then we have a similar bound with no $\delta$
\begin{equation}
\Expe{\para{\eta^{\delta}[0,  t]}^{q}}\leq c_{1} t^{\zeta(q)},
\end{equation}
but this exponent is not that sharp: if $\delta\leq t\leq 1$ and $q>1$, we have $q>\zeta(q)$ and so the previous bound is better. 

\item Finally, for all $t\in [0,\infty)$ and $q\in [0,1]$, we have
\begin{equation}
\Expe{\para{\eta^{\delta}[0,  t]}^{q}}\leq c_{2,\delta } t^{\zeta(q)}\delta^{\frac{\gamma^{2}}{2}(q^{2}-q)},
\end{equation}
where $c_{2,\delta}:=e^{\frac{\gamma^{2}}{2}(q^{2}-q)M_{g}}\Expe{\para{\int_{0}^{1}e^{ \overline{\omega}^{\delta}(x)}\dx}^{q}}$ where $\omega^{\delta}$ the exact-scaling field of height $\delta$.
\end{itemize}
Furthermore, for the positive moments in $p\in (0,1)$ of lower truncated GMC $\eta_{n}(0,t):=\int_{0}^{t}e^{U_{n}(s)}\ds$, we have the  bound
\begin{equation}\label{lowertrunp01}
\Expe{\para{\eta_{n}(0,t)}^{p}}\lessapprox~t^{\zeta(p)}, \forall t\geq 0.    
\end{equation}
When $p\in(-\infty,0)\cup (1,\frac{2}{\gamma^{2}})$, we have
\begin{equation}
\Expe{\para{\eta_{n}(0,t)}^{p}}\leq \Expe{\para{\eta(0,t)}^{p}} , \forall t\geq 0.       
\end{equation}
\end{proposition}


\subsection{Comparison of unit circle and real line}
Here we study the moments for $Q_{H}(a)$ because we need them in the ratio computations \cref{prop:inverse_ratio_moments}. We only study the convex case.
\begin{proposition}\label{prop:momentsunitcirclerealine}
Fix $q\in (-\infty,0]\cup [1,\frac{2}{\gamma^{2}})$. For $t\in [0,1)$ we have the bound
\begin{equation}
\Expe{\para{\eta_{H}(0,t)}^{q}}\leq c t^{\zeta(q)}    
\end{equation}
and for $t\in [1,\infty)$ we have the bound
\begin{equation}
\Expe{\para{\eta_{H}(0,t)}^{q}}\lessapprox \branchmat{\ceil{t}^{q}  &q>0\\\floor{t}^{q}& q<0  }.
\end{equation}    
\end{proposition}
\begin{proof}
As explained in \cref{sec:fieldonunitcircle}, we have the covariance: for points $x,\xi\in [0,1)$, with 0 and 1 identified,  and $\varepsilon\leq r$ we find for $y:=|x-\xi|$:
\begin{eqalign}
\Exp[H_{\varepsilon}(x)H_{r}(\xi)]:=&\branchmat{2\log(2)+\log\frac{1}{2sin(\pi y)}  &\tifc y> \frac{2}{\pi} arctan(\frac{\pi}{2}\e)\\ \log(1/\e)+(1/2)\log(\pi^{2}\varepsilon^{2}+4)+\frac{2}{\pi}\frac{arctan(\frac{\pi}{2}\e)}{\varepsilon}&\tifc y\leq \frac{2}{\pi} arctan(\frac{\pi}{2}\e)\\
-\log(\pi)- y \varepsilon^{-1}-\log(cos(\frac{\pi}{2}y))&}.
\end{eqalign}
For this covariance it is shown in \cite[Lemma 6.5]{SJ15}, that it is comparable to the exact-logarithmic field $\omega$ on any compact metric space  i.e. 
\begin{equation}\label{eq:covarianceunitcirclerealline}
 \sup_{n\geq 1}\sup_{x,\xi\in [0,1),0\sim 1}\abs{\Exp[H_{\e_{n}}(x)H_{\e_{n}}(\xi)]-\Exp[\omega_{\e_{n}}(x)\omega_{\e_{n}}(\xi)]}<K   
\end{equation}
for $\e_{n}\to 0$.
\proofparagraph{Case $t\in [0,1$ and $q\in (-\infty,0]\cup [1,\frac{2}{\gamma^{2}})$}
 For $q\in (-\infty,0]\cup [1,\frac{2}{\gamma^{2}})$ we have that $F(x)=x^{q}$ is convex, and so we use the \cref{eq:covarianceunitcirclerealline} and Kahane inequality, to dominate $H$ by the field $\omega+N(0,K)$ where $N(0,K)$ is an independent Gaussian. So from here we proceed as before to get the bound $t^{\zeta(q)}$.
\proofparagraph{Case $t\in [1,\infty)$ and $q\in (-\infty,0]\cup [1,\frac{2}{\gamma^{2}})$}
Here we have to be careful of the inherent periodicity of $H$. If we have $t=n+r$ for $n\in \mathbb{N},r\in [0,1]$, we write
\begin{equation}
\eta_{H}(0,t)=n\eta_{H}(0,1)+\eta_{H}(0,r).    
\end{equation}
We again apply Kahane inequality, to dominate by 
\begin{equation}
n\eta_{\omega}(0,1)+\eta_{\omega}(0,r).    
\end{equation}
So from here we get the bound $n^{q}$ for $q<0$ and the bound $(n+1)^{q}$ for $q>0$.
\end{proof}


\section{Small ball estimates for \sectm{$\eta$ }}
For small ball we can use the following theorem from \cite[theorem 1.1]{nikula2013small} and its application to GMC in \cite[theorem 3.2]{nikula2013small} for the truncated field with $\delta=1$. See also the work in \cite{lacoin2018path}, \cite{garban2018negative} for more small deviation estimates. We will use the notion of stochastic domination:
\begin{eqalign}
X\succeq Y \doncl \Proba{X\geq x}\geq \Proba{Y\geq x}.
\end{eqalign}
\begin{theorem}\cite[theorem 1.1]{nikula2013small}\label{smalldeviationeta}
Let $W$ and $Y$ be positive random variables satisfying stochastic dominating relations
\begin{eqalign}\label{eq:smoothing-inequality}
Y \succeq W_0 Y_0 + W_1 Y_1,
\end{eqalign}
where $(Y_0,Y_1)$ is an independent pair of copies of $Y$, independent of $(W_0,W_1)$, and $W_0 \eqdis W_1 \eqdis W$. Suppose further that there exist $\gamma > 1$ and $x' \in ]0,1[$ such that
\begin{eqalign}\label{eq:w-small-tail}
\Prob(W \leq x) \leq \exp\left( - c (-\log x)^\gamma \right) \quad \textrm{for all } x \leq x'.
\end{eqalign}
Then for any $\alpha \in [1,\gamma[$ there exists a constant $t_\alpha > 0$ such that for all $t \geq t_\alpha$ we have
\begin{eqalign}\label{eq:main-estimate}
\E e^{-t Y} \leq \exp\left(- c_\alpha (\log t)^\alpha \right) \quad \textrm{for all } t \geq t_\alpha.
\end{eqalign}
Therefore, we have the following small deviation estimate for $t\geq t_{\alpha}$ and $u=\frac{M}{t}$ for some $M>0$
\begin{eqalign}
    \Proba{u\geq Y}\leq e^{M} \exp\left(- c_\alpha (\log \frac{M}{u})^\alpha \right).
\end{eqalign}
\end{theorem}
Then in \cite[theorem 3.2]{nikula2013small}, it is shown that the above theorem applies to the total mass of the exact-scaling case $Y=\eta_{\omega}^{1}(0,1)$ by checking the two assumptions \Cref{eq:smoothing-inequality},\Cref{eq:w-small-tail} for $\alpha\in [1,2)$. But here we need it for $Y=\eta_{\omega}^{1}(0,t)$ for $t\geq 0$. Combining this with Kahane's inequality we have the following bound for the Laplace transform.
\begin{proposition}\label{laptraeta}
Fix $\alpha\in [1,\beta^{-1})$. We have the following estimate for the Laplace transform of $\eta^{\delta}$ with field $U^{\delta,g}$. Let $r=r(t)>0$ satisfying $rt\geq t_{\alpha}\geq 1$ from above \Cref{smalldeviationeta}. When $t\in  [0,\delta]$ then
\begin{eqalign}
\Expe{\expo{-r\eta^{\delta}(t)}}\lessapprox     \expo{-\frac{1}{2}\para{\frac{\e m_{2}(r,t)}{\gamma \sigma_{t}}}^{2}}\vee  \expo{-c_{\alpha}\para{m_{2}(r,t)(1-\e)}^{\alpha}},
\end{eqalign}
and when $t\in [\delta,\infty)$ then
\begin{eqalign}
\Expe{\expo{-r\eta^{\delta}(t)}}\lessapprox   \expo{-c_{\alpha}\para{\wt{m}_{2}(r,t)(1-\e)}^{\alpha}},
\end{eqalign}
where $\e\in (0,1)$ can be optimized,
\begin{eqalign}
\sigma_{t}^{2}:=\gamma^{2}(\ln\frac{\delta}{t}+M_{g})&\tand\wt{\sigma}^{2}:=\gamma^{2}M_{g},\\
m_{2}(r,t):=\ln{rt}-\frac{\gamma^{2}}{2}(\ln\frac{\delta}{t}+M_{g})>0&\tand \wt{m}_{2}(r,t):=\ln{rt}-\frac{\gamma^{2}}{2}M_{g}>0.
\end{eqalign}
\end{proposition}
\begin{proof}
\proofparagraph{Case $t\in [0,\delta]$}
We start with $\eta(  t)=  t\int_{0}^{1}e^{ \overline{U}^{\delta,g}(  tx)}\dx$. The field's $U^{\delta,g}(  tx)$ covariance is upper bounded by:
\begin{eqalign}
\para{\ln\frac{\delta}{t\abs{x-y}}-1+    \frac{t\abs{x-y}}{\delta}+g_{\delta}\para{t\abs{x-y}}}\one_{\abs{x-y}\leq \minp{1,\frac{\delta}{t}}}\leq \ln\frac{\delta}{t}+ \ln\frac{1}{\abs{x-y}}+M_{g}
\end{eqalign}
for all $\abs{x-y}\leq 1$. When $\delta, t\leq 1$, then we can simply ignore the nonpositive term $\ln\delta\leq 0$. (For the field $\omega^{\delta}$ we have the same covariance bound). Therefore, we apply Kahane's inequality for the fields $U^{\delta,g}(tx)$ and $N(0,\ln\frac{\delta}{t}+M_{g} )+\omega^{1}(x)$ to upper bound by
\begin{eqalign}\label{LTKahanebound}
  \Expe{\expo{-rt\int_{0}^{1}e^{ \overline{U}^{\delta}(tx)}\dx}}\leq& \Expe{\expo{-rte^{\gamma\overline{N}(0,\ln\frac{\delta}{t}+M_{g}) }\int_{0}^{1}e^{ \overline{\omega}^{1}(x)}\dx}}.
\end{eqalign}
Here one can possibly try to use the results on the Laplace transform of lognormal, in particular the asymptotic estimate \cite{asmussen2016laplace} to get better estimates. For now we just use the above estimate \cite[theorem 3.2]{nikula2013small} for $Y:=\int_{0}^{1}e^{ \overline{\omega}^{1}(x)}\dx$ and some $t_{\alpha}\geq 1$ to upper bound by
\begin{eqalign}
&\Expe{\expo{- c_\alpha \para{\log\para{rte^{\gamma\overline{N}(0,\ln\frac{\delta}{t}+M_{g}) }} }^\alpha}\ind{rte^{\gamma\overline{N}(0,\ln\frac{\delta}{t}+M_{g}) }\geq t_{\alpha}\geq 1}}\\
+&\Proba{\gamma\overline{N}(0,\ln\frac{\delta}{t}+M_{g})\leq -\ln\frac{rt}{t_{\alpha}}}.
\end{eqalign}
We study the first term. Using layercake and change of variables, the first expectation is
\begin{eqalign}
c_{\alpha}\int_{0}^{\infty}\Proba{\frac{x-m_{2}}{\gamma \sigma_{t}}\geq N(0,1)\geq -\frac{\abs{m_{1}}}{\gamma \sigma_{t}}}e^{-c_{\alpha}x^{\alpha}}x^{\alpha-1}\dx,
\end{eqalign}
where  $\sigma_{t}^{2}:=\gamma^{2}(\ln\frac{\delta}{t}+M_{g})$, $m_{1}:=\frac{1}{\gamma}\ln\frac{t_{\alpha}}{rt}+\frac{\gamma}{2}(\ln\frac{\delta}{t}+M_{g})<0$, $\tand     m_{2}:=\ln{rt}-\frac{\gamma^{2}}{2}(\ln\frac{\delta}{t}+M_{g})=\abs{m_{1}}+\ln t_{\alpha}>0$.  We split into ranges $[0,\infty)=[0,m_{2}(1-\e)]\cup [m_{2}(1-\e),\infty),$ for some $\e>0$. For the first range we estimate by
\begin{eqalign}
\Proba{\frac{-\e m_{2}}{\gamma \sigma_{t}}\geq N(0,1)}\int_{0}^{m_{2}(1-\e)} e^{-c_{\alpha}x^{\alpha}}x^{\alpha-1}\dx\leq c\expo{-\frac{1}{2}\para{\frac{\e m_{2}}{\gamma \sigma_{t}}}^{2}},
\end{eqalign}
and for the second range we estimate by
\begin{eqalign}
\int_{m_{2}(1-\e)}^{\infty} e^{-c_{\alpha}x^{\alpha}}x^{\alpha-1}\dx=\frac{1}{\alpha}    \int_{\para{m_{2}(1-\e)}^{\alpha}}^{\infty} e^{-c_{\alpha}x}\dx\lessapprox~e^{-c_{\alpha}\para{m_{2}(1-\e)}^{\alpha}}.
\end{eqalign}
For the probability term, we use the usual Gaussian tail estimate
\begin{eqalign}
&\Proba{\gamma\overline{N}(0,\ln\frac{\delta}{t}+M_{g})\leq -\ln\frac{rt}{t_{\alpha}}}\leq  c\expo{-\frac{m_{3}^{2}}{2}},
\end{eqalign}
where $m_{3}:=\frac{1}{\gamma \sigma_{t}}\para{\ln\frac{rt}{t_{\alpha}}-\frac{\gamma^{2}}{2}(\ln\frac{\delta}{t}+M_{g})}=\frac{1}{\gamma \sigma_{t}}\para{m_{2}-\ln t_{\alpha}}$. Due to the constraint $rt\geq t_{\alpha}\geq 1$, we see that as $t\to 0$, we have $r\to +\infty$. In this asymptotic, we see that $m_{3}^{2}$ and $\para{\frac{m_{2}}{\sigma_{t}}}^{2}$ have similar growths whereas $m_{2}^{\alpha}$ for $\alpha$ close to 2 has much faster growth due to the absence of $\frac{1}{\sigma_{t}}$. Therefore, for small $t\approx 0$ we can upper bound by the slowest bound, $3\cdot c\expo{-\frac{1}{2}\para{\frac{\e m_{2}}{\gamma \sigma_{t}}}^{2}}$. But for general $t\in [0,\delta]$, we simply take maximum.
\proofparagraph{Case $t\in [\delta,\infty)$}
Here for the field $U^{\delta,g}( tx)$ we upper bound its covariance by:
\begin{eqalign}
\para{\ln\frac{\delta}{t\abs{x-y}}-1+    \frac{t\abs{x-y}}{\delta}+g_{\delta}\para{t\abs{x-y}}}\one_{\abs{x-y}\leq \minp{1,\frac{\delta}{t}}}\leq \ln\frac{1}{\abs{x-y}}+M_{g}
\end{eqalign}
for all $\abs{x-y}\leq 1$.  (For the field $\omega^{\delta}$ we have the same covariance bound). So we will apply Kahane's inequality for the fields $U^{ \delta}(tx)$ and $N(0,M_{g})+\omega^{1}(x)$ to upper bound by:
\begin{eqalign}
  \Expe{\expo{-rt\int_{0}^{1}e^{ \overline{U}^{\delta}(tx)}\dx}}\leq& \Expe{\expo{-rte^{\gamma\overline{N}(0,M_{g}) }\int_{0}^{1}e^{ \overline{\omega}^{1}(x)}\dx}}.
\end{eqalign}
This expression is similar to \Cref{LTKahanebound} with the only difference of replacing $\ln\frac{\delta}{t}+M_{g}$ by $M_{g}$. So we have
\begin{eqalign}
\Expe{\expo{- c_\alpha \para{\log\para{rte^{\gamma\overline{N}(0,M_{g}) }} }^\alpha}\ind{rte^{\gamma\overline{N}(0,M_{g}) }\geq t_{\alpha}\geq 1}}+\Proba{\gamma\overline{N}(0,M_{g})\leq -\ln\frac{rt}{t_{\alpha}}}.
\end{eqalign}
Here we are not dividing $\frac{1}{\sigma_{t}}$ and so the slowest bound will be the $\alpha<2$ and so we bound by
\begin{eqalign}
ce^{-c_{\alpha}\para{\wt{m}_{2}(1-\e)}^{\alpha}},
\end{eqalign}
where $\wt{m}_{2}:=\ln{rt}-\frac{\gamma^{2}}{2}M_{g}>0$.
\end{proof}
We need corollaries here that simplifies that above bounds for the case of exponential-choices for the parameters and $g=0$. First, we study a simple case for $\delta=\rho_{*}$-scale.
\begin{corollary}\label{cor:smallballestimateexpocho0}
We fix $\rho\in (0,1)$. We also fix integer $\ell\geq 1$. We set
\begin{eqalign}
r:=\rho_{1}^{-\ell},t:=\rho_{2}^{\ell}    \tand \delta:=\rho,
\end{eqalign}
for $ \rho\geq \rho_{2}>\rho_{1}$ defined as
\begin{equation}
\rho_{1}=\rho^{a_{1} }\tand \rho_{2}=\rho^{a_{2}}.
\end{equation}
The exponents $a_{i}$ need to satisfy $\frac{1}{\ell}<a_{2}<a_{1}$. For small enough $\rho$ we have the bound
\begin{eqalign}
  \Expe{\expo{-r\eta^{\delta}(t)}}\leq  c\rho^{c_{1}\ell},
\end{eqalign}
where
\begin{eqalign}
c_{1}:= \frac{1}{32}\para{\frac{(a_{1}-a_{2}) }{\sqrt{\beta\para{a_{2}-\frac{1}{\ell}}  }   }-\sqrt{\beta\para{a_{2}-\frac{1}{\ell}}  }   }^{2}.
\end{eqalign}
\end{corollary}
\begin{proof}
The condition $rt\geq t_{\alpha}$ translates to
\begin{eqalign}
\ell(a_{1}-a_{2})\ln\frac{1}{\rho}\geq \ln t_{\alpha},
\end{eqalign}
which is true for small enough $\rho$ since $\ell\geq 1$. Here we have the case $\delta\geq t$ and so we study the exponents $\frac{ m_{2}(r,t)}{\sigma_{t}}$ and $m_{2}(r,t)$. We write explicitly the first exponent
\begin{eqalign}
\frac{ m_{2}(r,t)}{\sigma_{t}}=\frac{\ln{rt}}{\sqrt{\beta\ln\frac{\delta}{t}}   }-\sqrt{\beta\ln\frac{\delta}{t}}   =&\sqrt{\ell} \para{\frac{\ln{\frac{\rho_{2}}{\rho_{1}}}}{\sqrt{\beta\ln\para{\frac{\rho^{\frac{1}{\ell}}}{\rho_{2}} }}   }-\sqrt{\beta\ln\para{\frac{\rho^{\frac{1}{\ell}}}{\rho_{2}} }}  }\\
=&\sqrt{\ell} \sqrt{\ln{\frac{1}{\rho}}}\para{\frac{(a_{1}-a_{2})}{\sqrt{\beta\para{a_{2}-\frac{1}{\ell}}  }   }-\sqrt{\beta\para{a_{2}-\frac{1}{\ell}}  }   }\\
\geq &\sqrt{\ln{\frac{1}{\rho}}}\sqrt{\ell}b_{1,1},
\end{eqalign}
where
\begin{eqalign}
b_{1,1}:=\frac{(a_{1}-a_{2}) }{\sqrt{\beta\para{a_{2}-\frac{1}{\ell}}  }   }-\sqrt{\beta\para{a_{2}-\frac{1}{\ell}}  } .
\end{eqalign}
The second exponent is
\begin{eqalign}
m_{2}(r,t)=\ln{rt}-\beta\ln\frac{\delta}{t}    =&\ell\para{\ln{\frac{\rho_{2}}{\rho_{1}}}-\beta\ln\para{\frac{\rho^{\frac{1}{\ell}}}{\rho_{2}} }}\\
=&\ell\ln{\frac{1}{\rho}}\para{a_{1}-a_{2}-\beta\para{a_{2}-\frac{1}{\ell}}}\\
\geq&\ell\ln{\frac{1}{\rho}} b_{1,2},
\end{eqalign}
where
\begin{eqalign}
b_{1,2}:=a_{1}-a_{2}(1+\beta)+\beta\frac{1}{\ell}.
\end{eqalign}
Therefore, we are comparing
\begin{eqalign}\label{eq:twoexponentstocompare0}
\expo{-\frac{1}{32}\ell\ln{\frac{1}{\rho}}b_{1,1}^{2}}\tand  \expo{-c_{\alpha}\para{b_{1,2}0.5\ell\ln{\frac{1}{\rho}}}^{\alpha}}.
\end{eqalign}
To simplify the bounds we take $\alpha>1$ and small enough $\rho$ so that
\begin{eqalign}
\frac{1}{32}b_{1,1}^{2}<c_{\alpha}\para{b_{1,2}0.5}^{\alpha} \para{\ln{\frac{1}{\rho}}}^{\alpha-1}.
\end{eqalign}
That way we only keep the first exponential term in \Cref{eq:twoexponentstocompare0}.

\end{proof}

Next we study the more complicated situation of varying scale $\delta=\rho_{*}^{n}$.
\begin{corollary}\label{cor:smallballestimateexpocho}
We fix $\rho\in (0,1),\e\geq 0$. We also fix integers $n\geq k\geq 1$ and $z,\ell\geq 0$. We set
\begin{eqalign}
r:=\rho_{2}^{-n-z}\rho_{1}^{-\ell},t:=\rho_{2}^{k+z+\ell+\e}\rho_{3}^{n-k}    \tand \delta:=\rho_{2}^{n},
\end{eqalign}
for $ \rho_{3}>\rho_{2}>\rho_{1}$ defined as
\begin{equation}
\rho_{3}=\rho^{a_{3}},\rho_{2}=\rho^{a_{2}},\rho_{1}=\rho^{a_{1} }.
\end{equation}
The exponents $a_{i}$ need to satisfy
\begin{eqalign}
&a_{3}<a_{2}<a_{1}\\
&a_{1}>a_{2}(1+\e)(1+\beta)+\beta z.
\end{eqalign}
\proofparagraph{Case $n=k$ and $\ell\geq 1$}
For small enough $\rho$ we have the bound
\begin{eqalign}
  \Expe{\expo{-r\eta^{\delta}(t)}}\leq  c\rho^{c_{1}\ell},
\end{eqalign}
where
\begin{eqalign}
c_{1}:= \frac{1}{32}\para{\frac{(a_{1}-a_{2}(1+\e) }{\sqrt{\beta\para{a_{2}(1+\e)+z}  }   }-\sqrt{\beta\para{a_{2}(1+\e)+z}  }   }^{2}.
\end{eqalign}
\proofparagraph{Case $n>k$ and $\ell\geq 0$ and $z=0$}
We have
\begin{eqalign}
\Expe{\expo{-r\eta^{\delta}(t)}}\leq c\branchmat{\rho^{\ell c_{2} } & \tcwhen\ell> \frac{a_{2}-a_{3}}{a_{2}}(n-k)\\
\rho^{\ell c_{3,1}+(n-k) c_{3,2} } & \tcwhen 0\leq\ell\leq  \frac{a_{2}-a_{3}}{a_{2}}(n-k)},
\end{eqalign}
and
\begin{eqalign}
c_{2}:=&\frac{1}{32}\para{\frac{a_{1}-a_{2}(1+\e)(1+\beta )}{\sqrt{\beta a_{2}\para{1+\e}  }   }   }^{2}\\
c_{3,1}:=&(a_{1}-a_{2}(1+\e))\para{\ln\frac{1}{\rho}}^{\alpha-1}\tand c_{3,2}:=(a_{2}-a_{3})\para{\ln\frac{1}{\rho}}^{\alpha-1}.
\end{eqalign}
\end{corollary}
\begin{proof}
For ease of notation we let
\begin{eqalign}
&\rho_{2,\e}:=\rho_{2}^{1+\frac{\e}{\ell}} \tand a_{2,\e}:=a_{2}(1+\frac{\e}{\ell}),\\
&\rho_{2,\e,z}:=\rho_{2}^{1+\frac{\e}{\ell}+\frac{z}{a_{2}}} \tand a_{2,\e}:=a_{2}(1+\frac{\e}{\ell}+\frac{z}{a_{2}}).
\end{eqalign}
\proofparagraph{Case $n=k$ and $\ell\geq 1$}
The condition $rt\geq t_{\alpha}$ translates to
\begin{eqalign}
(\ell+\e)(a_{1}-a_{2})\ln\frac{1}{\rho}\geq \ln t_{\alpha},
\end{eqalign}
which is true for small enough $\rho$. Here we have the case $\delta\geq t$
\begin{eqalign}
\delta=\rho_{2}^{n}\geq \rho_{2}^{n+z+\ell+\e}=t,
\end{eqalign}
and so we study the exponents $\frac{ m_{2}(r,t)}{\sigma_{t}}$ and $m_{2}(r,t)$. We write explicitly the first exponent
\begin{eqalign}
\frac{ m_{2}(r,t)}{\sigma_{t}}=\frac{\ln{rt}}{\sqrt{\beta\ln\frac{\delta}{t}}   }-\sqrt{\beta\ln\frac{\delta}{t}}   =&\sqrt{\ell} \para{\frac{\ln{\frac{\rho_{2,\e}}{\rho_{1}}}}{\sqrt{\beta\ln\para{\frac{1}{\rho_{2,\e,z}} }}   }-\sqrt{\beta\ln\para{\frac{1}{\rho_{2,\e,z}} }}  }\\
=&\sqrt{\ell} \sqrt{\ln{\frac{1}{\rho}}}\para{\frac{(a_{1}-a_{2,\e})}{\sqrt{\beta\para{a_{2,\e,z}}  }   }-\sqrt{\beta\para{a_{2,\e,z}}  }   }\\
\geq &\sqrt{\ln{\frac{1}{\rho}}}\sqrt{\ell}b_{1,1},
\end{eqalign}
where
\begin{eqalign}
b_{1,1}:=\frac{a_{1}-a_{2}(1+\e) }{\sqrt{\beta\para{a_{2}(1+\e+\frac{z}{a_{2}})}  }   }-\sqrt{\beta\para{a_{2}(1+\e+\frac{z}{a_{2}})}  } .
\end{eqalign}
The second exponent is
\begin{eqalign}
m_{2}(r,t)=\ln{rt}-\beta\ln\frac{\delta}{t}    =&\ell\para{\ln{\frac{\rho_{2,\e}}{\rho_{1}}}-\beta\ln\para{\frac{1}{\rho_{2,\e,z}} }}\\
=&\ell\ln{\frac{1}{\rho}}\para{a_{1}-a_{2,\e}-\beta\para{a_{2,\e,z}}}\\
\geq&\ell\ln{\frac{1}{\rho}} b_{1,2},
\end{eqalign}
where
\begin{eqalign}
b_{1,2}:=a_{1}-a_{2}(1+\e)(1+\beta)-\beta z.
\end{eqalign}
Therefore, we are comparing
\begin{eqalign}\label{eq:twoexponentstocompare}
\expo{-\frac{1}{32}\ell\ln{\frac{1}{\rho}}b_{1,1}^{2}}\tand  \expo{-c_{\alpha}\para{b_{1,2}0.5\ell\ln{\frac{1}{\rho}}}^{\alpha}}.
\end{eqalign}
To simplify the bounds we take $\alpha>1$ and small enough $\rho$ so that
\begin{eqalign}
\frac{1}{32}b_{1,1}^{2}<c_{\alpha}\para{b_{1,2}0.5}^{\alpha} \para{\ln{\frac{1}{\rho}}}^{\alpha-1}.
\end{eqalign}
That way we only keep the first exponential term in \Cref{eq:twoexponentstocompare} because it has smaller decay.
\proofparagraph{Case $n>k$ and $\ell\geq 0$ and $z=0$}
The condition $rt\geq t_{\alpha}$ translates to
\begin{eqalign}
(n-k)(a_{2}-a_{3})\ln\frac{1}{\rho}\geq \ln t_{\alpha},
\end{eqalign}
which is true for small enough $\rho$.
\proofparagraph{Subcase $t< \delta$}
This subcase happens when
\begin{eqalign}
\frac{\delta}{t}> 1\doncl \ell+\e> \para{ \para{\ln\frac{1}{\rho_{2}} }^{-1}\ln\para{\frac{\rho_{3}}{\rho_{2}}}}  (n-k)=\frac{a_{2}-a_{3}}{a_{2}}(n-k)
\end{eqalign} and so we have $\ell>n- k\geq 1$. The first exponent now is more complicated because it contains $n-k$ too
\begin{eqalign}
\frac{ m_{2}(r,t)}{\sigma_{t}}=&\frac{\ln{rt}}{\sqrt{\beta\ln\frac{\delta}{t}}   }-\sqrt{\beta\ln\frac{\delta}{t}}   \\
=&\sqrt{\ell} \para{\frac{\ln{\frac{\rho_{2,\e}}{\rho_{1}}}}{\sqrt{\beta\ln\para{\para{\frac{\rho_{2}}{\rho_{3}}}^{(n-k)/\ell}\frac{1}{\rho_{2,\e}} }}   }-\sqrt{\beta\ln\para{\para{\frac{\rho_{2}}{\rho_{3}}}^{(n-k)/\ell}\frac{1}{\rho_{2,\e}} }}  }\\
&+\frac{n-k}{\sqrt{\ell}}\para{\frac{\ln{\frac{\rho_{3}}{\rho_{2}}}}{\sqrt{\beta\ln\para{\para{\frac{\rho_{2}}{\rho_{3}}}^{(n-k)/\ell}\frac{1}{\rho_{2,\e}} }}    }}\\
=&\sqrt{\ell} \sqrt{\ln{\frac{1}{\rho}}}\para{\frac{(a_{1}-a_{2,\e})}{\sqrt{\beta\para{a_{2,\e}-\frac{n-k}{\ell}(a_{2}-a_{3})}  }   }-\sqrt{\beta\para{a_{2,\e}-\frac{n-k}{\ell}(a_{2}-a_{3})}  }   }\\
&+\frac{n-k}{\sqrt{\ell}}\sqrt{\ln{\frac{1}{\rho}}}\para{\frac{(a_{2}-a_{3})}{\sqrt{\beta\para{a_{2,\e}-\frac{n-k}{\ell}(a_{2}-a_{3})}  } }}\\
\geq &\sqrt{\ln{\frac{1}{\rho}}}\para{\sqrt{\ell}b_{2,1}+\frac{n-k}{\sqrt{\ell}}b_{2,2}},
\end{eqalign}
where
\begin{eqalign}
b_{2,1}:=&\para{\frac{a_{1}-a_{2}(1+\e)(1+\beta )}{\sqrt{\beta a_{2}\para{1+\e}  }   }   }\\
b_{2,2}:=&\para{\frac{(a_{2}-a_{3})}{\sqrt{\beta a_{2}\para{1+\e}  } }}.
\end{eqalign}
The second exponent is
\begin{eqalign}
m_{2}(r,t)=&\ln{rt}-\beta\ln\frac{\delta}{t}    \\
=&\ell\para{\ln{\frac{\rho_{2,\e}}{\rho_{1}}}-\beta\ln\para{\para{\frac{\rho_{2}}{\rho_{3}}}^{(n-k)/\ell}\frac{1}{\rho_{2,\e}} }}+(n-k)\ln{\frac{\rho_{3}}{\rho_{2}}}\\
=&\ell\ln{\frac{1}{\rho}}\para{a_{1}-a_{2,\e}-\beta\para{a_{2,\e}-\frac{n-k}{\ell}(a_{2}-a_{3})}}+(n-k)\ln{\frac{1}{\rho}}(a_{2}-a_{3})\\
\geq&\ln{\frac{1}{\rho}}\para{\ell b_{3,1}+(n-k)b_{3,2}},
\end{eqalign}
where
\begin{eqalign}
b_{3,1}:=&a_{1}-a_{2}(1+\e)(1+\beta )\tand b_{3,2}:=a_{2}-a_{3}.
\end{eqalign}
As before to simplify the bounds we take $\alpha>1$ and $\rho$ small enough so that we keep the first bound. And we only keep the first part in terms of $\ell$ and let $c_{2}:=b_{2,1}.$
\proofparagraph{Subcase $t\geq \delta$}
In this case the exponent is simply
\begin{eqalign}\label{eq:tgeqdeltaexpo}
\para{\wt{m}_{2}(r,t)}^{\alpha}=\para{\ln{rt}}^{\alpha}=& \para{\ell(a_{1}-a_{2,\e})+(n-k)(a_{2}-a_{3})}^{\alpha}\para{\ln\frac{1}{\rho}}^{\alpha}.
\end{eqalign}
To assist with having large coefficients we take $\alpha>1$ and lower bound by
\begin{eqalign}
\eqref{eq:tgeqdeltaexpo}\geq  &\para{\ell c_{3,1}+(n-k)c_{3,2}}\para{\ln\frac{1}{\rho}},
\end{eqalign}
where
\begin{eqalign}
c_{3,1}:=&(a_{1}-a_{2}(1+\e))\para{\ln\frac{1}{\rho}}^{\alpha-1}\tand c_{3,2}:=(a_{2}-a_{3})\para{\ln\frac{1}{\rho}}^{\alpha-1}.
\end{eqalign}

\end{proof}
\subsection{ Rate of convergence to Lebesgue measure}\label{rateofconvergence}
This section is not used in this article but our above Laplace-transform estimate gives a small ball estimate for 1d-GMC which might be of independent interest. In particular, we use ideas from \cite[Proposition 6.2]{lacoin2018path} to estimate the rate of convergence 
\begin{eqalign}
\Proba{ \abs{\eta^{n}(A)-\abs{A}}>\Delta}\to 0,
\end{eqalign}
where the height of the field $U^{n}$ is $\delta_{n}$, as $\Delta\to 0$ or $n\to+\infty$, and $A=[M_{1},M_{2}]$ (or by translation invariance just $[0,M]$). Then we transfer that to the convergence of the inverse $Q^{n}(x)\to x$.
\begin{proposition}\label{ratetoLebesgue} Fix $A:=\spara{0,M}~\tforsome M>0$ and $\alpha\in [1,\beta^{-1})$. For $m>n$ and $\Delta>0$, we have the following estimates
\begin{eqalign}
\Proba{ \eta_{m}^{n}(A)<\abs{A}-\Delta}&\leq c_{1} e^{r\para{\abs{A}-\Delta}} \expo{-\frac{\abs{A}}{4\delta_{n}}\para{\ln r \delta_{n}}^{\alpha}}=:B_{n,m}^{1}(\Delta,A),\\
\tand \Proba{ \eta_{m}^{n}(A)>\abs{A}+\Delta}&\lessapprox
\frac{1}{\Delta^{p}} \abs{A}\delta_{n}^{p-1}=:B_{n}^{2}(\Delta,A),
\end{eqalign}
where $p\in (0,2]\cap [0,\frac{2}{\gamma^{2}})$, $c_{i}>0,i=1,2,3$ and the  $r=r(\delta_{n})>0$ is required from \Cref{laptraeta} to satisfy
\begin{eqalign}
r\delta_{n}\geq t_{\alpha}>1.
\end{eqalign}
If we allow $\gamma<1$, we have
\begin{eqalign}
\Proba{ \abs{\eta_{m}^{n}(A)-\abs{A}}>\Delta} \lessapprox &\branchmat{\frac{1}{\Delta^{2}}\abs{A}^{2-\gamma^{2}}\delta_{n}^{\gamma^{2}} &\tcwhen \abs{A}<\delta_{n}\\ \frac{1}{\Delta^{2}}\para{\delta_{n}^{2}+(\abs{A}-\delta_{n})\delta_{n} }&\tcwhen \delta_{n}\leq \abs{A} }.
\end{eqalign}
\end{proposition}
\begin{proof}
The case of measure zero sets $\abs{A}=0$ can be ignored because GMC is bi-\Holder and so the estimates are true immediately.
~\proofparagraph{Estimating \sectm{$\Proba{ \eta^{n}(A)<\abs{A}-\Delta}$}}
Here we will use the Laplace transform function $e^{-rX}$, a convex function of $X$. We use \cite[Lemma 6.]{bacry2003log} which says that $X_{m}:=\eta_{m}$ is a martingale in $m$. So its convex function is a submartingale and hence by Jensen's
\begin{eqalign}
 \Expe{-r\eta^{n}_{m}(A)}\leq \Expe{-r\eta^{n}(A)}.
\end{eqalign}
 In order to utilize the independence of $\eta^{n}(A)$ for distant intervals, we consider the collection
\begin{eqalign}
C:=\set{ [\delta_{n}k, \delta_{n}(k+1)]: k=0,...,\ceil{\frac{M}{\delta_{n}}}    }
\end{eqalign}
which will cover the set $A$ and use $C_{odd},C_{even}$ for odd and even $k$  respectively. Therefore, first we apply Chernoff inequality and then let $Z_{I}:=\eta^{n}\para{A\cap I}$ to obtain an upper bound
\begin{eqalign}
\Proba{ \eta^{n}_{m}(A)<\abs{A}-\Delta}\leq& e^{r\para{\abs{A}-\Delta}}\Expe{-r\eta^{n}(A)}\\
=&e^{r\para{\abs{A}-\Delta}}\Expe{\expo{-r\sum_{I\in C_{odd}  } Z_{I}-r\sum_{I\in C_{even}  } Z_{I} }}\\
\stackrel{Cauchy-Schwartz}{\leq} &e^{r\para{\abs{A}-\Delta}}\sqrt{\Expe{\expo{-2r\sum_{I\in C_{odd}  } Z_{I}}}     }\sqrt{\Expe{\expo{-2r\sum_{I\in C_{even}  } Z_{I}}}     }\\
\stackrel{ind}{=}& e^{r\para{\abs{A}-\Delta}} \sqrt{\prod_{I\in C_{odd}  }\Expe{\expo{-2r Z_{I}}}     }\sqrt{\prod_{I\in C_{even}  }\Expe{\expo{-2rZ_{I}}}     }\\
=&e^{r\para{\abs{A}-\Delta}} \sqrt{\prod_{I\in C  }\Expe{\expo{-2r Z_{I}}}     }.
\end{eqalign}
Here we only keep the intervals $I$ that are fully contained in $I\subset [0,M]$ and there are about
\begin{eqalign}
\#\set{ I\subset A}\geq \frac{1}{2}\frac{\abs{A}}{\delta_{n}}.
\end{eqalign}
These $Z_{I}$ are \iid and so using that we have at least $\frac{1}{2}\ceil{\frac{\abs{A}}{\delta_{n}}}$ intervals, we bound the above by
\begin{eqalign}
 e^{r\para{\abs{A}-\Delta}} \para{\Expe{\expo{-2r \etamu{0,\delta_{n}}{\delta_{n}} }}}^{\frac{\abs{A}}{4\delta_{n}}}.
\end{eqalign}
 We apply the exponential bound \Cref{laptraeta} for $t=\delta_{n}$ and $r=r(\delta_{n})$ satisfying $r\delta_{n}\geq t_{\alpha}\geq 1$ for $\alpha\in [1,\beta^{-1})$
\begin{align*}
&e^{r\para{\abs{A}-\Delta}} c   \expo{-c_{n}\frac{\abs{A}}{\delta_{n}}\frac{1}{4}\para{m_{2}}^{\alpha}}=e^{r\para{\abs{A}-\Delta}} \expo{-\frac{\abs{A}}{4\delta_{n}}\para{\ln r \delta_{n}}^{\alpha}},
\end{align*}
where in \Cref{laptraeta} we set $M_{g}=0$ and $t=\delta_{n}$ and $\e=\frac{1}{2}$.
\proofparagraph{Estimating \sectm{$\Proba{ \eta^{n}(A)>\abs{A}+\Delta}$}}
We start with a Markov inequality for  $p\in (0,\frac{2}{\gamma^{2}}\wedge 2)$ (but allowing $p=2$ when $\frac{2}{\gamma^{2}}> 2$)
\begin{eqalign}
\Proba{ \eta^{n}_m(A)>\abs{A}+\Delta}\leq \frac{1}{\para{\Delta}^{p}} \Expe{\abs{\eta_m^{n}(A)-\abs{A}}^{p}}.
\end{eqalign}
Here we will use the same decomposition with $Z_{I}:=\abs{\eta_m^{n}(A\cap I)-\abs{A\cap I}}$
\begin{eqalign}
\eta_m^{n}(A)-\abs{A}=\sum_{I\in C_{odd}  } Z_{I}+\sum_{I\in C_{even}  } Z_{I}
\end{eqalign}
to get the upper bound
\begin{eqalign}
\Expe{\para{\eta_m^{n}(A)-\abs{A}}^{p}}\leq 2^{p-1}\Expe{\para{\sum_{I\in C_{odd}  } Z_{I}}^{p}} +2^{p-1}\Expe{\para{\sum_{I\in C_{even}  } Z_{I} }^{p}}.
\end{eqalign}
Then we use that for independent centered random variables and $p\in[1,2]$, we have the upper bound by Bahr-Esseen   \cite{von1965inequalities,fazekas2017general}
\begin{eqalign}
2^{p}\sum_{I\in C_{odd}  }\Expe{ Z_{I}^{p}} +2^{p}\sum_{I\in C_{even}  }\Expe{ Z_{I}^{p}}
\end{eqalign}
and for $p\in (0,1)$ this immediately follows by subaddivity of concave functions. We use the scaling law $\Expe{(\eta^{n}(A\cap I))^{p}}\leq \Expe{(\delta_{n} \eta^{1}(0,1))^{p}}$ and bound by the sum:
\begin{eqalign}
\Expe{Z_{I}^{p}}\leq  2^{p}\delta_{n}^{p}\para{\Expe{\para{\eta^{1}_{m}(0,1)}^{p}}+1^{p}},
\end{eqalign}
When $p>1$, we can again use submartingale  \cite[Lemma 6.]{bacry2003log} to bound by
\begin{eqalign}
\Expe{\para{\eta^{1}_{m}(0,1)}^{p}}\leq \Expe{\para{\eta^{1}(0,1)}^{p}}
\end{eqalign}
and when $p\in (0,1)$, we use \Cref{lowertrunp01} in \Cref{momentseta}
\begin{eqalign}
\Expe{\para{\eta^{1}_{m}(0,1)}^{p}}\leq c1^{p}.
\end{eqalign}
So all together since there are $\frac{\abs{A}}{\delta_{n}}$ elements in the sum, we have for $p>1$
\begin{eqalign}
\Proba{ \eta_{m}^{n}(A)>\abs{A}+\Delta}&\leq \frac{c_{p}}{\Delta^{p}} \abs{A}\delta_{n}^{p-1}\para{\Expe{\para{\eta^{1}(0,1)}^{p}}+1   }
\end{eqalign}
and for $p\in (0,1)$
\begin{eqalign}
\Proba{ \eta_{m}^{n}(A)>\abs{A}+\Delta}&\leq\frac{c_{p}}{\Delta^{p}} \abs{A}\delta_{n}^{p-1}2.
\end{eqalign}
\proofparagraph{Case $\gamma<1$}
Here we simply take the $L^{2}$-moment
\begin{eqalign}
\Expe{\para{\eta^{n}(M)-M}^{2}}=&\Expe{\para{\eta^{n}(M)}^{2}}-M^{2}\\
=&\iint_{[0,M]^{2}}\Expe{e^{\bar{U}^{n}(s)+\bar{U}^{n}(t)}}\ds\dt-M^{2}\\
\leq &\iint_{[0,M]^{2}}\Expe{e^{\bar{U}^{n}(s)+\bar{U}^{n}(t)}}\ind{\abs{s-t}\leq \delta_{n}}\ds\dt\\
\leq &\iint_{[0,M]^{2}}\para{\frac{\delta_{n}}{\abs{s-t}}}^{\gamma^{2}}\ind{\abs{s-t}\leq \delta_{n}}\ds\dt   \\
\lessapprox &\branchmat{M^{2-\gamma^{2}}\delta_{n}^{\gamma^{2}} &\tcwhen M<\delta_{n}\\ \delta_{n}^{2}+(M-\delta_{n})\delta_{n} &\tcwhen \delta_{n}\leq M }.
\end{eqalign}
In the second inequality, we only kept the logarithmic part of the covariance by bounding the other part $\frac{\abs{s-t}}{\delta_{n}}-1\leq 0$.
\end{proof}

\begin{remark}\label{negativetailconstant}
The second estimate goes to zero as $n\to +\infty$ for all $\Delta>0$ and $p>1$. In the first estimate, we have to be careful on the choice of $\Delta>0$. To get a decaying bound in the first estimate $B_{n,m}^{1}$, we would need to make a careful choice of $\Delta$.  Because $\frac{1}{\delta_{n}}$ beats the $(\ln\delta_{n})^{2}$ growth, we can make choices of $\Delta$ and $r_{n}$ so that we have the assumption $r_{n}\delta_{n}\geq t_{a}\alpha\geq 1$ and the above bound goes to zero as $n\to +\infty$. For example, we take $r_{n}:=b\frac{t_{\alpha}}{\delta_{n}}$ for some fixed $b>1$, then the bound becomes
\begin{eqalign}\label{eq:largedevia}
 \expo{-\frac{1}{\delta_{n}}\para{\frac{1}{4}\abs{A}\para{\ln b t_{\alpha} }^{2}- b t_{\alpha} \para{\abs{A}-\Delta}}   }=:\expo{-\frac{1}{\delta_{n}} c_{\Delta}   }.
 \end{eqalign}
For fixed $\abs{A}$, we can make a choice of $\Delta$ close enough to $\abs{A}$ so that
\begin{eqalign}
\frac{1}{4}\abs{A}\para{\ln b t_{\alpha} }^{2}- b t_{\alpha} \para{\abs{A}-\Delta}>0
\end{eqalign}
and so the overall exponent decays in $n$ for all $\Delta\in [\abs{A}-\lambda,\abs{A})$ for some small enough $\lambda>0$.
\end{remark}
 By converting to the analogous estimate for $\eta^{n}$, we similarly have the rate of convergence for the inverse.
\begin{corollary}\label{rateofconverinverse}
For the inverse $Q^{n}$ we have the rates
\begin{eqalign}
\Proba{Q^{n}(0,x)-x>\Delta}=\Proba{\para{x+\Delta}-\etam{x+\Delta}>\Delta}\leq B_{n,m}^{1}(\Delta,x+\Delta)
\end{eqalign}
and
\begin{eqalign}
\Proba{x-Q^{n}_{m}(0,x)>\Delta}=\Proba{\etam{x-\Delta}-\para{x-\Delta}>\Delta}\leq B_{n}^{2}(\Delta,x-\Delta),
\end{eqalign}
where in this second estimate $x>\Delta$ is forced otherwise the first probability is zero.
\end{corollary}
\begin{remark}
The extension of \cref{rateofconverinverse} to intervals
\begin{eqalign}
\Proba{\abs{Q^{n}(a,b)-(b-a)}>\Delta} 
\end{eqalign}
is unclear. This is because the current proof using decoupling doesn't work as evaluating GMCs over two distant subintervals $I,J\subset \spara{Q_{a},Q_{a}+x} $ are still correlated.
\end{remark}

\section{Moments of the maximum and minimum of modulus of GMC }\label{maxminmodGMC}
In this section we study tail estimates and small ball estimates of the maximum/minimum of shifted GMC.  
We first recall the positive/negative moments of the supremum of GMC from \cite[appendix]{binder2023inverse}.
\begin{proposition}\label{prop:maxmoduluseta}
\proofparagraph{Moments $p\in [1,\frac{2}{\gamma^{2}})$ using union bound }
For $L,\delta,x\geq 0$  and $\delta\leq 1$ we have
\begin{equation}\label{eq:maxmodulusetapone}
\Expe{\para{\supl{T\in[0,L] }\etamu{T,T+x}{\delta}}^{p}}\leq c\ceil{\frac{L}{x}}\Expe{\para{\etamu{0,x}{\delta}}^{p} }.    
\end{equation}
\proofparagraph{Moments $p\in [1,\frac{2}{\gamma^{2}})$ using decoupling}
We have the bound
\begin{eqalign}
&\Expe{\para{\int_{0}^{x}\sup_{T\in [0,L]}e^{\overline{U(T+s)\cap U(T)}} \deta_{+}(s)}^{p}  }\leq \branchmat{c x^{p(1-\alpha)} & \frac{L}{x}\geq e\\c x^{p(1-\tilde{\alpha})} & \frac{L}{x}\leq e},    
\end{eqalign}
for $\alpha:=12\sqrt{2}\gamma+\beta (p-1)+\epsilon<1$ and $\tilde{\alpha}:=\beta (p-1)+\epsilon<1$.
\proofparagraph{Moments $p\in (0,1)$}
We have the bound
\begin{eqalign}
\Expe{\para{\int_{0}^{x}\sup_{T\in [0,L]}e^{\overline{U(T+s)\cap U(T)}} \deta_{+}(s)}^{p}  }\leq & \branchmat{c x^{p(1-12\sqrt{2}\gamma)}\para{\ln\frac{1}{x}}^{p/2},     & \frac{L}{x}\geq e\\c \para{\ln\frac{1}{x}}^{p/2},     & \frac{L}{x}\leq e},
\end{eqalign}
where $c$ is a constant that diverges as $12\sqrt{2}\gamma\to 1$.
\end{proposition}
\begin{proposition}\label{prop:minmodeta}
We have for $p>0$
\begin{eqalign}
\Expe{\para{\infl{T\in[0,L] }\etamu{T,T+x}{\delta}}^{-p}}\leq &c\ceil{\para{\frac{L}{x}}}\Expe{\para{\eta^{\delta}\spara{0,\frac{x}{2}}}^{-p} }\\
\leq& c\frac{L}{\delta} x^{\alpha_{2\delta}(p)},
\end{eqalign}
where $\alpha_{2\delta}(p)=\zeta(-p)-1$ if $\frac{x}{2}\leq \delta$ and $\alpha_{2\delta}(p)=-p$ if $\frac{x}{2}\geq \delta$.
\end{proposition}

\subsection{Maximum/Minimum of shifted GMC}
 Here we study the maximum/minimum of the doubly-shifted GMC. These estimates show up when studying the inverse in the context of the semigroup formula. Recall from \Cref{not:semigroupformula} the semigroup notation
\begin{equation}
\set{Q_{b}\bullet T\geq t}=\set{b\geq \eta(T,T+t)}.    
\end{equation}
\begin{lemma}\label{lem:maxdshiftedGMC}
For the moments $p\in [1,\frac{2}{\gamma^{2}})\tand \delta \leq 1$ we have for $L,x,b>0$
\begin{eqalign}
 &\Expe{\para{\maxl{0\leq T\leq L}\etamu{Q_{b}\bullet T+T,Q_{b}\bullet T+T+x}{\delta}}^{p}}\lessapprox \para{1+(L+2x)^{1/p}b^{q/p}}^{p} (2+L+x) x^{\alpha(p)},    
\end{eqalign}
where the exponents satisfy the constraints:
\begin{itemize}
    \item $\alpha(p)=\zeta(p)-1$ when $x\leq 1$ and $\alpha(p)=p$ when $x\geq 1$.
    \item $q$ is large enough so that  $\lambda q/p>1$ for some fixed $\lambda\in (0,1)$.
\end{itemize}
\end{lemma}
\begin{proof}
As in the proof of \Cref{cor:shiftedGMCmoments} for the increasing-case, we start with a decomposition $I_{k}:= [a_{k},a_{k+1}]$ ,for the diverging sequence $a_{k}:=k^{\lambda}$ for $\lambda\in (0,1)$ and $k\geq 0$,
\begin{equation}
O_{k,T}:=\set{Q_{b}\bullet T\in I_{k}}=\set{\eta^{\delta}(T,T+a_{k+1})\geq b\geq \eta^{\delta}(T,T+a_{k})}.
\end{equation}
Here we see that $a_{k}=0$ gives 
We temporarily take $1/p$ power in order to apply $L_{p}$-triangle inequality. We decompose for each $T$ and upper bound
\begin{equation}
\Expe{\para{\maxl{0\leq T\leq L}\sum_{k\geq 0}\maxl{Y\in I_{k}}\etamu{Y+T,Y+T+x}{\delta}\ind{O_{k,T}}  }^{p}}^{1/p}.    
\end{equation}
Then we apply triangle inequality for the $\max$ and the Lp-triangle inequality for $p\geq 1$
\begin{equation}
\sum_{k\geq 0}\para{ \Expe{\para{\maxl{Y\in I_{k}}\maxl{0\leq T\leq L}\etamu{Y+T,Y+T+x}{\delta}\maxl{0\leq T\leq L}\ind{O_{k,T}}}^{p} } }^{1/p}    
\end{equation}
and we also we used that for non-negative functions $\maxl{x} f_{1}(x)f_{2}(x)\leq \maxl{x}f_{1}(x)\maxl{x}f_{2}(x)$. Since $a_{k}\to +\infty$, in the event $O_{k,T}$ we keep only the lower bound and apply the FKG inequality as in \Cref{cor:shiftedGMCmoments} to bound by their product
\begin{equation}
\Expe{\para{\maxl{Y\in I_{k}}\maxl{0\leq T\leq L}\etamu{Y+T,Y+T+x}{\delta} }^{p} }\Expe{\maxl{0\leq T\leq L}\ind{b\geq \eta^{\delta}(T,T+a_{k})}}.
\end{equation}
The first factor was estimated in \Cref{prop:maxmoduluseta} for $p\in [1,\frac{2}{\gamma^{2}})$
\begin{equation}
\Expe{\para{\supl{\tilde{T}\in[0,L+\abs{I_{k}}] }\etamu{\tilde{T},\tilde{T}+x}{\delta}}^{p}}\lessapprox  c(1+L+\abs{I_{k}}+x)x^{\alpha(p)},    
\end{equation}
where $\alpha(p)=\zeta(p)-1$ when $x\leq 1$ and $\alpha(p)=p$ when $x\geq 1$.\\
Next we study the second factor and its sum in $k\geq 0$. The expression is
\begin{equation}
\sum_{k\geq 0}\para{\Expe{\maxl{0\leq T\leq L}\ind{b\geq \eta^{\delta}(T,T+a_{k})}}}^{1/p}=\sum_{k\geq 0}\para{\Proba{b\geq \minl{0\leq T\leq L}\eta^{\delta}(T,T+a_{k})}}^{1/p} .    
\end{equation}
This minimum was studied in \Cref{prop:minmodeta}
\begin{equation}
\Expe{\para{\infl{T\in[0,L] }\etamu{T,T+x}{\delta}}^{-q}}\lessapprox~x^{a_{\delta}(-q)}\para{\frac{L}{x}+2},
\end{equation}
where $a(q):=\zeta(-q)-1$ when $x\leq \delta$ and $a(q):=-q$ when $x\geq \delta$. \\
We set $a_{k}:=k^{\lambda}$ for $\lambda\in (0,1)$ and $k\geq 0$. Then by subadditivity we have $\abs{I_{k}}=(k+1)^{\lambda}-k^{\lambda}\leq 1$.  Since $\delta\leq 1$, for the $k=0$ term we just bound by $1$. Whereas for $k\geq 1$ we have $a_{k}\geq 1$ and so we can use the $a_{\delta}(-q):=-q$ exponent and in turn require
\begin{equation}
\sum_{k\geq 1} \frac{1}{k^{\lambda (q-\e_{q})/p}}<\infty, 
\end{equation}
which means $\lambda (q-\e_{q})/p>1$. This is possible by taking $q>0$ large enough and small enough $\e_{q}$.
\end{proof}
\begin{remark}
One could also possibly study the tail of the above maximum using disjoints unions:
\begin{eqalign}
&\Proba{\maxl{0\leq T\leq L}\etamu{Q_{b}\bullet T+T,Q_{b}\bullet T+T+x}{\delta}\geq t }\\
=&\Proba{\bigsqcup_{k\geq 0}\bigcup_{0\leq T\leq L}\set{\etamu{Q_{b}\bullet T+T,Q_{b}\bullet T+T+x}{\delta}\geq t}\cap O_{k,T} }\\
=&\sum_{k\geq 0}\Proba{\bigcup_{0\leq T\leq L}\set{\etamu{Q_{b}\bullet T+T,Q_{b}\bullet T+T+x}{\delta}\geq t}\cap O_{k,T} }.
\end{eqalign}
\end{remark}

 Next we study the negative moments of the minimum of shifted GMC. 
\begin{lemma}\label{lem:minshiftedGMC}
For $p,b>0$ and $\delta\leq 1$ we have
\begin{eqalign}
\Expe{\para{\min_{0\leq T\leq L}\etamu{Q_{b}\bullet T+T,Q_{b}\bullet T+T+x}{\delta}    }^{-p} }&\lessapprox \para{2+(L+2)^{1/pp_{2}}b^{q/pp_{2}}}^{p}(4+L+x)x^{(\alpha(p))}\\
&= c(L,x,b) x^{(\alpha(-p))},
\end{eqalign}
where $c(L,x,b)<\infty$ when uniformly bounded $L,x,b\leq B$ and $\alpha(-p)=\zeta(-p)-1$ when $x\leq \delta$ and $\alpha(-p)=-p$ when $x\geq \delta$.
\end{lemma}
\begin{proof}
The proof will be the analogue of the infimum case in \Cref{cor:shiftedGMCmoments}. We start with a decomposition
\begin{equation}
O_{k,T}:=\set{Q_{b}\bullet T\in [a_{k},a_{k+1}]=:I_{k}}=\set{\eta^{\delta}(T,T+a_{k+1})\geq b\geq \eta^{\delta}(T,T+a_{k})},
\end{equation}
for some diverging sequence $a_{k}$ to be chosen latter. We decompose for each $T$ and upper bound
\begin{eqalign}
&\Expe{\para{\min_{0\leq T\leq L}\etamu{Q_{b}\bullet T+T,Q_{b}\bullet T+T+x}{\delta}    }^{-p} }\\
&\leq \Expe{\para{\maxl{Y\in I_{k}}\maxl{0\leq T\leq L}\sum_{k\geq 0}\para{\etamu{Y+T,Y+T+x}{\delta}}^{-1} \ind{O_{k,T}}  }^{p}}.    
\end{eqalign}
We temporarily take power $1/p$ and we apply triangle inequality for the $\max$ and the $L_p$-triangle inequality for $p\geq 1$
\begin{eqalign}
&\Expe{\para{\maxl{Y\in I_{k}}\maxl{0\leq T\leq L}\sum_{k\geq 0}\para{\etamu{Y+T,Y+T+x}{\delta}}^{-1} \ind{O_{k,T}}  }^{p}}^{1/p}\\
&\leq \sum_{k\geq 0}\para{ \Expe{\para{\maxl{Y\in I_{k}}\maxl{0\leq T\leq L}\para{\etamu{Y+T,Y+T+x}{\delta}}^{-1} \maxl{0\leq T\leq L}\ind{O_{k,T}}}^{p} } }^{1/p}.    
\end{eqalign}
We set $a_{k}:=k^{\lambda}$ for $\lambda\in (0,1)$ and $k\geq 0$. Because the two factors are highly correlated we decouple them via \Holder
\begin{eqalign}
 &\para{\Expe{\para{\maxl{Y\in I_{k}}\maxl{0\leq T\leq L}\para{\etamu{Y+T,Y+T+x}{\delta}}^{-1} \maxl{0\leq T\leq L}\ind{O_{k,T}}}^{p} }}^{1/p} \\
\leq &\para{\Expe{\para{\minl{Y\in I_{k}}\minl{0\leq T\leq L}\etamu{Y+T,Y+T+x}{\delta} }^{-pp_{1}} }}^{1/pp_{1}}\para{\Expe{\maxl{0\leq T\leq L}\ind{b\geq \eta^{\delta}(T,T+a_{k})}}}^{1/pp_{2}},
\end{eqalign}
for $\frac{1}{p_{1}}+\frac{1}{p_{2}}=1$. The first factor was estimated in \Cref{prop:minmodeta}
\begin{equation}
\Expe{\para{\infl{T\in[0,L+\abs{I_{k}}] }\etamu{T,T+x}{\delta}}^{-pp_{1}}}\lessapprox~x^{a_{\delta}(-pp_{1})}\para{\frac{L+\abs{I_{k}}}{x}+2},
\end{equation}
where $a(q):=\zeta(-q)$ when $x\leq \delta$ and $a(q):=-q$ when $x\geq \delta$. \\
For the second factor, when $k\leq 2^{1/\lambda}$, we bound  by $1$. When $k\geq 2^{1/\lambda}$, we obtain sum as in the proof of the maximum-case in \Cref{lem:maxdshiftedGMC}. Namely we apply \Cref{prop:minmodeta} for $q>0$
\begin{equation}
\para{\Expe{\maxl{0\leq T\leq L}\ind{b\geq \eta^{\delta}(T,T+a_{k})}}}^{1/pp_{2}}\leq \sum_{k\geq 1} \frac{1}{k^{\lambda (q-1)/pp_{2}}}, 
\end{equation}
and for finiteness we require $\lambda (q-1)/pp_{2}>1$. This is possible by taking $q>0$ large enough. Finally, we are free to take $p_{1}=1+\e$ for arbitrarily small $\e>0$ since $\eta$ has all its negative moments. That will force $p_{2}=\frac{1+\e}{\e}$ to be arbitrarily large and so again we need to ensure that we take $q$ large enough.
\end{proof}

\subsection{Version with perturbation}
 Next we go over the versions we need for \Cref{sec:multipointmaximum}. We have the maximum version.

\begin{lemma}\label{lem:maxshiftedGMCpertubation}
Fix $L\geq 0$, $P:=[e^{-r},e^{r}]$ for $r<2$ and $c<1$. For $p>0$ we have 
\begin{eqalign}
&\Expe{\para{\max_{T\in [0,L]}\maxls{u\in P}e^{-u}\eta\spara{Q\para{\para{e^{u}-e^{-r}}c}\bullet T+T,Q\para{\para{e^{u}-e^{-r}}c}\bullet T+T+t}   }^{-p} }\lessapprox c(L,t) t^{\alpha(p)},
\end{eqalign}
for the same exponents as in \Cref{lem:maxdshiftedGMC} and $c(L,t)$ is uniformly bounded for bounded $t,L<B$.
\end{lemma}
 And the following is the minimum version. 
\begin{lemma}\label{lem:minshiftedGMCpertubation}
Fix $L\geq 0$, $P:=[e^{-r},e^{r}]$ for $r<2$ and $c<1$. For $p>0$ we have 
\begin{eqalign}
&\Expe{\para{\min_{T\in [0,L]}\minls{u\in P}e^{-u}\eta\spara{Q\para{\para{e^{u}-e^{-r}}c}\bullet T+T,Q\para{\para{e^{u}-e^{-r}}c}\bullet T+T+t}   }^{-p} }\lessapprox c(L,t) t^{(\alpha(-p)-1)},
\end{eqalign}
for the same exponents as in \Cref{lem:minshiftedGMC} and $c(L,t)$ is uniformly bounded for bounded $t,L<B$.
\end{lemma}
Their proofs are modifications of \Cref{lem:maxdshiftedGMC} and \Cref{lem:minshiftedGMC}, so we only provide the proof for the minimum case.
\begin{proof}[proof of \Cref{lem:minshiftedGMCpertubation}]
The proof is the same as of \Cref{lem:minshiftedGMC} with the difference that we have to keep track of $u$. We let $b_{u}:=\para{e^{u}-e^{-r}}c$ and start with a decomposition
\begin{equation}
O_{k,T,u}:=\set{Q_{b_{u}}\bullet T\in [a_{k},a_{k+1}]=:I_{k}}=\set{\eta(T,T+a_{k+1})\geq b_{u}\geq \eta(T,T+a_{k})},
\end{equation}
for some diverging sequence $a_{k}$ to be chosen latter. We decompose for each $T$ and upper bound
\begin{eqalign}
&\Expe{\para{\min_{0\leq T\leq L}\minls{u\in P}e^{-u}\etam{Q_{b_{u}}\bullet T+T,Q_{b_{u}}\bullet T+T+t}    }^{-p} }\\
&\leq \Expe{\para{\maxl{Y\in I_{k}}\maxls{u\in P}\maxl{0\leq T\leq L}\sum_{k\geq 0}\para{e^{-u}\etam{Y+T,Y+T+t}}^{-1} \ind{O_{k,T,u}}  }^{p}}.    
\end{eqalign}
We temporarily take power $1/p$ and we apply triangle inequality for the $\max$ and the $L_p$-triangle inequality for $p\geq 1$
\begin{eqalign}
&\Expe{\para{\maxl{Y\in I_{k}}\maxls{u\in P}\maxl{0\leq T\leq L}\sum_{k\geq 0}\para{e^{-u}\etamu{Y+T,Y+T+t}{\delta}}^{-1} \ind{O_{k,T,u}}  }^{p}}^{1/p}\\
&\leq e^{r} \sum_{k\geq 0}\para{ \Expe{\para{\maxl{Y\in I_{k}}\maxl{0\leq T\leq L}\para{\etamu{Y+T,Y+T+t}{\delta}}^{-1} \maxls{u\in P}\maxl{0\leq T\leq L}\ind{O_{k,T,u}}}^{p} } }^{1/p}.    
\end{eqalign}
We set $a_{k}:=k^{\lambda}$ for $\lambda\in (0,1)$ and $k\geq 0$. Because the two factors are highly correlated we decouple them via \Holder
\begin{eqalign}
 &\para{\Expe{\para{\maxl{Y\in I_{k}}\maxl{0\leq T\leq L}\para{\etamu{Y+T,Y+T+t}{\delta}}^{-1} \maxls{u\in P}\maxl{0\leq T\leq L}\ind{O_{k,T,u}}}^{p} }}^{1/p} \\
\leq &\para{\Expe{\para{\minl{Y\in I_{k}}\minl{0\leq T\leq L}\etamu{Y+T,Y+T+t}{\delta} }^{-pp_{1}} }}^{1/pp_{1}}\\
&\cdot\para{\Expe{\maxls{u\in P}\maxl{0\leq T\leq L}\ind{b_{u}\geq \eta^{\delta}(T,T+a_{k})}}}^{1/pp_{2}},
\end{eqalign}
for $\frac{1}{p_{1}}+\frac{1}{p_{2}}=1$. For the $u$-maximum we just upper bound
\begin{equation}
b_{u}=\para{e^{u}-e^{-r}}c\leq \para{e^{r}-e^{-r}}c \leq e^{2}. 
\end{equation}
From here the proof proceeds as in the proof of \Cref{lem:minshiftedGMC}.

\end{proof}

\section{Inverse and Shifted GMC moments}
Here we state the moments for the inverse and shifted GMC from \cite{binder2023inverse}.
\begin{proposition}\label{prop:momentsofshiftedinverse}
Let $p\in \left(-\dfrac{\left(1+\frac{\gamma^{2}}{2}\right)^{2}}{2\gamma^{2}}, \infty \right)$, the height $\delta\leq 1$ and $a,x>0$.
\begin{itemize}
    \item (Positive moments) for all $p>0$
\begin{eqalign}
\Expe{(Q^{\delta}(a,a+x))^{p}}\leq& c_{1}x^{p_{1}}C_{pos}(a,\delta)+c_{2}x^{p_{2}},
\end{eqalign}
for $p_{1},p_{2}$ constrained as 
\begin{equation}
 \zeta(-p_{1})+p>1\tand p_{2}>p.   
\end{equation}

\item (Negative moments)For all  $p\in \left(0,\dfrac{\left(1+\frac{\gamma^{2}}{2}\right)^{2}}{2\gamma^{2}} \right)$
\begin{eqalign}
\Expe{(Q^{\delta}(a,a+x))^{-p}}\leq& \tilde{c}_{1}x^{-\tilde{p}_{1}}+\tilde{c}_{2}x^{-\tilde{p}_{2}},
\end{eqalign}
for $\tilde{p}_{1},\tilde{p}_{2}$ constrained as
\begin{equation}
\zeta(\tilde{p}_{1})>p+1    \tand     p>\tilde{p}_{2}.
\end{equation}
\end{itemize}
The constants dependence is
\begin{eqalign}
C_{pos}(a,\delta)&:=\maxp{\delta^{\zeta(-p_{1})+p-2},C_{inf,1}(\alpha,\delta,\rho),C_{inf,2}(\alpha,\delta,\rho) }   \\
C_{neg}(a,\delta)&:=C_{sup}(a,\delta,\rho)
\end{eqalign}
with the constants $C_{inf,1}$ , $C_{inf,2}$,$C_{sup}$ coming from \Cref{cor:shiftedGMCmoments}. 
\end{proposition}
\begin{corollary}\label{cor:shiftedGMCmoments}
Let $t>0$.
\begin{itemize}
    \item  For positive moments $p\in [1,\beta^{-1})$ we have
 \begin{eqalign}
\Expe{\para{\etamu{Q^{\delta}_{a},Q^{\delta}_{a}+t}{\delta_{0}}}^{p}} &\leq
t^{\alpha(p)}C_{sup}(a,\delta,\rho),
\end{eqalign}
where $\alpha(p):=\zeta(p)-1$ when $t\leq \delta_{0}$ and $\alpha(p):=p$ when $t\geq \delta_{0}$.

\item  For positive moments $p\in (0,1)$  we have
 \begin{eqalign}\label{eq:momentsp01shifted}
\Expe{\para{\etamu{Q^{\delta}_{a},Q^{\delta}_{a}+t}{\delta_{0}}}^{p}} &\leq
C_{sup}(a,\delta,\rho)\branchmat{t^{p(1-12\sqrt{2}\gamma)}\para{\ln\frac{1}{t}}^{p/2} & t\leq \frac{\rho \delta}{e}\\ \para{\ln\frac{1}{t}}^{p/2} & t\geq \frac{\rho \delta}{e}}.
\end{eqalign}
For simplicity we further bound by $t^{p\alpha}$ for $\alpha=1-12\sqrt{2}\gamma-\epsilon$ for small $\e>0$ to avoid carrying the log-term around.

 \item  For negative moments we have
\begin{eqalign}\label{shiftednegativeappinfmom}
\Expe{\para{\etamu{Q^{\delta}_{a},Q^{\delta}_{a}+t}{\delta_{0}}}^{-p} } \leq &t^{\alpha_{\e}(p)}C_{inf}(\alpha,\delta,\delta_{0},\rho),
\end{eqalign}
where $\alpha_{\e}(p):=\frac{\zeta(p(1+\e))-1}{1+\e}$ for $t\leq 2\delta_{0}$ and $\alpha_{\e}(p):=p$ for $t\geq 2\delta_{0}$.
\end{itemize}
\end{corollary}
\subsection{Precomposed term \sectm{$\eta^{k}(Q^{k+m}(x))$}}
Here we use the change of variables formula \cite[(4.9) Proposition.]{revuz2013continuous}: Consider any increasing, possibly infinite, right-continuous function $A:[0,\infty)\to [0,\infty]$ with inverse $C_{s}$, then if $f$ is a nonnegative Borel function on $[0,\infty)$ we have
\begin{equation}
\int_{[0,\infty)}f(u)dA_{u}=    \int_{[0,\infty)}f(C_{s})\ind{C_{s}<\infty}\ds.
\end{equation}
So this change of variables gives 
\begin{equation}
  \eta^{k}(Q^{k+m}(x))=\int_{0}^{x}e^{U^{k}_{k+m}(Q^{k+m}(s))}\ds  
\end{equation}
and so we can now condition by $\mathcal{U}^{k+m}_{\infty}$. Here as $m\to +\infty$, we simply get $\eta^{k}(x)$.
\begin{lemma}\label{lemm:precomposedetainvuppertrun}
Fix $\gamma<\sqrt{73} - 6 \sqrt{2}$, $x,\delta_{k+m}$, then we have
\begin{equation}\label{Bl2uppertr}
\Expe{\para{\eta^{k}(Q^{k+m}(x))-x}^{2}}\lessapprox \delta_{k}^{p} x^{2-p},
\end{equation}
for $p\in (0,1)$.
\end{lemma}
\begin{proof}
We start by scaling the $\delta_{k}$
\begin{eqalign}
\Expe{\para{\eta^{k}(Q^{k+m}(x))-x}^{2}}=\delta_{k}^{2}\Expe{\para{\eta^{1}\para{Q^{m}\para{\frac{x}{\delta_{k}}}}-\frac{x}{\delta_{k}}}^{2}}.    
\end{eqalign}
In what follows we replace $\frac{x}{\delta_{k}}$ by $x$ for ease of notation and change it back for the final result. By first conditioning on $\mathcal{U}^{m}_{\infty}$ we have:
\begin{eqalign}
\Expe{\para{\eta^{1}(Q^{m}(x))-x}^{2}}
&\leq\Expe{\iint\limits_{[0,x]^{2}\cap \Delta_{1} }\para{\frac{\delta_{1}}{\abs{Q^{m}(s)-Q^{m}(t)}\vee  \delta_{m}} }^{\gamma^{2}}\ds\dt}- \Expe{\text{Area}\spara{\Delta_{1}}},
\end{eqalign}
where $\Delta_{1}:=\set{(s,t)\in [0,x]^{2}: \abs{Q^{m}(s)-Q^{m}(t)}\leq \delta_{1} }$.
First we estimate the expected size of the complement of the diagonal set
\begin{equation}\label{eq:areaofdiagonalset}
\Expe{\text{Area}\spara{\Delta_{1}}}=  2\iint\limits_{[0,x]^{2}, s<t } \Proba{Q^{m}(s,t)\leq  \delta_{1}}\ds\dt=2\iint\limits_{[0,x]^{2}, s<t } \Proba{t-s\leq \eta^{m}(Q^{m}(s),Q^{m}(s)+\delta_{1})}\ds\dt  
\end{equation}
and then bound by Markov for $p_{1}\in (0,1)$ and apply \cref{cor:shiftedGMCmoments}:
\begin{eqalign}
\eqref{eq:areaofdiagonalset}\leq&\iint\limits_{[0,x]^{2}, s<t }(t-s)^{-p_{1}}\Expe{\para{\eta^{m}(Q^{m}(s),Q^{m}(s)+\delta_{1})}^{p_{1}}}\ds\dt\\
\lessapprox &\delta_{1}^{p_{1}\alpha}\iint\limits_{[0,x]^{2}, s<t }(t-s)^{-p_{1}}\ds\dt\approx\delta_{1}^{p_{1}\alpha}x^{2-p_{1}}.
\end{eqalign}
For the double integral, we first study the case $\abs{Q^{m}(s)-Q^{m}(t)}\leq  \delta_{m}$. Here we get an upper bound for
\begin{eqalign}\label{eq:caselessthandeltam}
\para{\frac{\delta_{1}}{\delta_{m}} }^{\gamma^{2}} \Expe{\text{Area}\spara{\Delta_{m}}},    
\end{eqalign}
where $\Delta_{m}:=\set{(s,t)\in [0,x]^{2}: \abs{Q^{m}(s)-Q^{m}(t)}\leq \delta_{m} }$. We further scale
\begin{eqalign}
\eqref{eq:caselessthandeltam}\leq &\delta_{m}^{-\gamma^{2}}\iint\limits_{[0,x]^{2}, s<t }(t-s)^{-p_{1}}\Expe{\para{\eta^{m}(Q^{m}(s),Q^{m}(s)+\delta_{m})}^{p_{1}}}\ds\dt\\
=&\delta_{m}^{2-\gamma^{2}}\iint\limits_{\spara{0,\frac{x}{\delta_{m}}}^{2}, s<t }(t-s)^{-p_{1}}\Expe{\para{\eta^{1}(Q^{1}(s),Q^{1}(s)+\delta_{1})}^{p_{1}}}\ds\dt\\
\lessapprox&\delta_{m}^{2-\gamma^{2}} \para{\frac{x}{\delta_{m}}}^ {2-p_{1}}\approx\delta_{m}^{p_{1}-\gamma^{2}} x^{2-p_{1}}.
\end{eqalign}
This is non singular in $\delta_{m}$ by taking $p_{1}>\gamma^{2}$. Next we study the case $\abs{Q^{m}(s)-Q^{m}(t)}\geq  \delta_{m}$ to get:
\begin{eqalign}\label{eq:secondcasegreaterdeltam}
2\Expe{\iint\limits_{[0,x]^{2}\cap \Delta_{1,m} }\para{\frac{\delta_{1}}{Q^{m}(t)-Q^{m}(s)} }^{\gamma^{2}}\ds\dt},   
\end{eqalign}
for $\Delta_{1,m}:=\set{(s,t)\in [0,x]^{2}:t>s\tand  \delta_{m}\leq Q^{m}(t)-Q^{m}(s)\leq \delta_{1} }$. Using the layercake, we shift to GMC
\begin{eqalign}
\Expe{\ind{A_{1,m}}\para{\frac{1}{Q^{m}(t)-Q^{m}(s) } }^{\gamma^{2}}}    
=&\int_{c_{1}}^{c_{m}}\Proba{\eta^{m}\para{Q^{m}(s),Q^{m}(s)+\frac{1}{r^{1/\gamma^{2}}}}\geq t-s   }\dr,
\end{eqalign}
for $ A_{1,m}:=\set{\delta_{m}\leq Q^{m}(t)-Q^{m}(s)\leq \delta_{1}}$ and $c_{m}:=\para{\frac{1}{\delta_{m}}}^{\gamma^{2}}$. We again apply Markov for $p_{2}\in (0,1)$ and \cref{cor:shiftedGMCmoments}
\begin{eqalign}
\eqref{eq:secondcasegreaterdeltam}\lessapprox& \iint\limits_{[0,x]^{2}, s<t }(t-s)^{-p_{2}}\ds\dt \cdot \int_{c_{1}}^{c_{m}} \frac{1}{r^{\frac{p_{2}\alpha}{\gamma^{2}}}} \dr. 
\end{eqalign}
Here to avoid getting a divergent factor in $\delta_{m}$, we require $\frac{p_{2}\alpha}{\gamma^{2}}>1$ and so
\begin{eqalign}
\alpha>  \gamma^{2}\doncl \gamma<\sqrt{73} - 6 \sqrt{2}\approx 0.0587224  
\end{eqalign}
This is the most singular factor and so we just bound by $x^{2-p_{2}}$.
\end{proof}
\end{appendices}
\bibliography{sn-bibliography.bib}

\end{document}